\documentclass[11pt]{article}
\usepackage{mathrsfs}
\usepackage{bbm}
\usepackage{amsfonts}
\usepackage{amsmath,amssymb}
\usepackage{amsmath}
\usepackage{amssymb}
\usepackage{graphicx}
\usepackage{color}

\newtheorem{theorem}{Theorem}[section]
\newtheorem{corollary}[theorem]{Corollary}
\newtheorem{definition}[theorem]{Definition}

\newtheorem{lemma}[theorem]{Lemma}
\newtheorem{proposition}[theorem]{Proposition}
\newtheorem{remark}[theorem]{Remark}

\newenvironment{proof}[1][Proof]{\noindent \textbf{#1.} }{\  $\Box$}
\numberwithin{equation}{section}

\begin{document}

\title{\textbf{ Global maximum principle for mean-field forward-backward stochastic systems with delay
and application to finance}}
\author{Tao HAO\thanks{%
School of Statistics, Shandong University of Finance and Economics, Jinan 250014, P. R. China.
haotao2012@hotmail.com. Research supported by National Natural Science Foundation of China
(Grant Nos. 71671104,11871309,11801315,71803097), the Ministry of Education of Humanities and Social Science Project (Grant No. 16YJA910003),
Natural Science Foundation of Shandong Province (No. ZR2018QA001),
A Project of Shandong  Province Higher Educational Science and Technology Program (Grant Nos. J17KA162, J17KA163),
and Incubation Group Project of Financial Statistics and Risk Management of SDUFE.} \and %
Qingxin MENG\thanks{%
Qingxin  Meng is the corresponding author. Department of Mathematics, Huzhou University, Zhejiang 313000, P. R. China.
mqx@zjhu.edu.cn. Research supported by Natural Science Foundation of Zhejiang Province for Distinguished Young Scholar (Grant No. LR15A010001),
and the National Natural Science Foundation of China (Grant No. 11871121,11471079).}  }
\maketitle
\date{}

\begin{abstract}
The purpose of this paper is to explore the  necessary conditions for optimality of mean-field forward-backward delay control systems.
A new estimate is proved, which is a powerful
tool to deal with the optimal control problems of mean-field type with delay.
Different from the classical situation, in our case
the first-order adjoint system is an anticipated mean-field backward stochastic
differential equation, and the second-order adjoint system is a system of  matrix-valued
process, not mean-field type.
With the help of two adjoint systems, the second-order expansion of the variation of the state $Y$ is proved, and therewith
the Peng's stochastic maximum principle.
As an illustrative example, we apply our result to the mean-field game in Finance. Although
we just investigate  the case of one pointwise delay for convenience, but our method is adequate for analysing  the case of pointwise delay.

\end{abstract}

\textbf{Key words}: Stochastic control; Maximum principle;  Mean-field forward-backward stochastic
differential equation with delay, adjoint equation

\textbf{MSC-classification}: 93E20; 60H10\\\\


\noindent \section{{\protect \large {Introduction}}}

In this paper, we are interested in the following mean-field stochastic control
problem in $\mathbb{R}$ with one pointwise delay in the state:
\begin{equation}\label{equ 3.1}
  \left\{
\begin{aligned}
dX^v(t)&=b(t,X^v(t),X^v(t-l),P_{X^v(t)},v(t))dt+\sigma(t,X^v(t),X^v(t-l),P_{X^v(t)},v(t))dW_t,\ t\in[0,T],\\
X^v(0)&=x(\theta),\ \theta\in[-l,0],\\
 -dY^v(t)&=f(t,X^v(t),X^v(t-l),Y^v(t),Z^v(t),P_{(X^v(t),Y^v(t))},v(t))dt-Z^v(t)dW_t,\ t\in[0,T],\\
 Y^v(T)&=\Phi(X^v(T),P_{X^v(T)}),
 \end{aligned}
 \right.
\end{equation}
where $W$ is  a one dimensional Brownian motion, $P_\xi=P\circ \xi^{-1}$ is the law of  $\xi$, $l>0$ is a time delay parameter,
$T$ is a given time horizon satisfying  $l< T\leq 2l$, $x(\cdot)$ is a continuous  function defined on $[-l,0]$, $U$ is a subset of $\mathbb{R}$, not necessarily convex, $\mathcal{P}_2(\mathbb{R}^d)$
is the space of square integrable probability measures on $\mathbb{R}^d$ with 2-Wasserstein metric,
$b,\sigma:[0,T]\times\mathbb{R}\times\mathbb{R}\times \mathcal{P}_2(\mathbb{R})\times U\mapsto\mathbb{R},$
$f:[0,T]\times\mathbb{R}\times\mathbb{R}\times\mathbb{R}\times\mathbb{R}\times \mathcal{P}_2(\mathbb{R}^2)\times U\mapsto\mathbb{R}$,
$\Phi:\mathbb{R}\times \mathcal{P}_2(\mathbb{R})\mapsto\mathbb{R}$
satisfy some suitable conditions,
and the associated cost functional is defined  by
$J(v):=Y^v(0)$.

The purpose of the paper is to find the  necessary condition of the optimal control $u^*$ satisfying
$$
J(u^*)=\min_{v\in\mathcal{U}_{ad}}J(v).
$$

Apart from the potential enormous applications in Finance of mean-field control systems with
delay, see \cite{AHMP}, \cite{CH}, \cite{CW}, \cite{SMS},  or see
the example in Section 6,
the solvability of Peng's open problem, refer to \cite{Hu}, and the progress in studying
the necessary conditions of optimality for control systems with pointwise delay in the state and with control
dependent noise, refer to \cite{GM}, and the rapid development of the theories of mean-field forward-backward stochastic
differential equations (FBSDEs) and mean-field FBSDEs with delay in recent years,
refer to \cite{BDLP}, \cite{BLP}, \cite{BLPR}, \cite{CD}, \cite{CD1}, \cite{CCD}, \cite{GXZ},
give a great impulse for the research of present work.

First, by considering the second-order term in the Taylor expansion of the variation of the process $X$,
Peng \cite{Peng1} in 1990 obtained the necessary condition of  the optimality for control systems where the diffusion coefficient $\sigma$
depends on control and the control domain $U$ need not be convex. Later, he proposed an open problem in 1998: How to extend the
classical optimal control problem to the recursive case in the above situation, i.e., the coefficient $f$ of the BSDE depends on
$(y,z)$. There are two main obstacles met when $f$ depending on $z$ nonlinearly  (see  Yong \cite{Yong}):  What is the second-order
variational equation and how to obtain the corresponding second-order adjoint equation?
Hence, this open problem was not solved completely until 2017 by Hu \cite{Hu}. Hu overcome the above two difficulties by
building a new second-order Taylor expansion of the variation of $Y$. His idea has been borrowed to investigate the control
problems of fully coupled forward-backward stochastic control systems, see \cite{HJX}.

Second, in the real world, many stochastic differential equations evolve depending not only on the current state but also on the state at past some time, which is called stochastic delay differential equations. This kind of equations can be found frequently
in the fields of both natural and social sciences, for example,
Finance, Economics, and Physics. Hence, the optimal control problems with delay attracted many people's attention, such as Elsanousi, $ \phi$ksandal, Sulem \cite{EOS}, Shen, Meng, Shi \cite{SMS}. As we know,
it is difficult to deal with the stochastic delay control problems. The difficulties come on the one hand
from the infinite-dimensional problem, on the other hand from the absence of It\^{o}'s formula to handle the delay part of state.
Especially, when the control domain $U$ is not convex, there have not been corresponding works to give the necessary condition
of the optimal driven by a delay control system until Guatteri and Masiero \cite{GM} in 2018. The authors in \cite{GM} noticed that
in the delay case the square of variation can not be written  a closed form by one equation because of the mixed term of present
and past. Hence an auxiliary  matrix-value equation is brought in to deal with this problem.

Third,
mean-field stochastic differential equations (SDEs), also named McKean-Vlasov equations, were investigated
by Kac in the early 1950s \cite{Kac} and  McKean in the 1960s \cite{Mckean}. But, only by the year of 2009 was
the theory of mean-field backward stochastic differential equations (BSDEs) built by Buckdahn, Djehiche, Li and Peng  \cite{BDLP},
Buckdahn,   Li and Peng \cite{BLP}. It should be pointed that for these types of mean-field FBSDEs the coefficients depend
on the expectation of the solution, not the law of the solution.
Recently, inspired by the work of Lions \cite{Lions} in which the first-order derivative of a function with respect to a measure was given,
many scholars showed great enthusiasm for first-order mean-field games, see for example \cite{CD}.
In 2017, Buckdahn, Li, Peng, Rainer \cite{BLPR} defined the second-order derivative
of a function with respect to a measure, and studied the relationship between general mean-field SDEs and associated nonlocal PDEs.
Afterwards,  there were many works on the research of general mean-field FBSDEs and their applications, such as \cite{BLPR}, \cite{BLM1}, \cite{CCD}, \cite{HL}, \cite{Li}.
Note that since the expectation of a random variable (or a stochastic process) can be expressed by the law of this random variable (or stochastic process), we call these mean-field FBSDEs where
the coefficients depends on the law of the solution, the general mean-field FBSDEs.

As for the optimal control problems of mean-field delay control systems, there are also fruitful works.
For example, Guo, Xiong, Zheng \cite{GXZ} investigated the first-order necessary and sufficient conditions for optimality of mean-field delay control problems where the coefficients depend on the expectation.
Buckdahn, Li, Ma \cite{BLM2} considered a
stochastic control problem with
partial observations for general mean-field systems in the case where the coefficients  depend on the paths of the state and the conditional law
of the state.

Based on the previous works \cite{Hu}, \cite{GM},  \cite{GXZ}, \cite{BLPR}, \cite{BLM1}, a natural question is:
Is it possible to develop the  necessary condition for optimality of the general mean-field forward-backward control systems with delay?
We confirmed the question.

Although our method to some extent follows the schemes in \cite{GM}, there are still some potential obstacles:

(i) Since we consider the recursive case, i.e., the coefficient $f$ depending on $(y,z)$, the power of the term
$p(t)\delta\sigma(t)\mathbbm{1}_{E_\varepsilon}$ in the variation of $z$  is $O(\varepsilon)$, but not $o(\varepsilon)$.
As a consequent beyond that the expansion of $X^\varepsilon$ are considered, which consists in the classical case,
we also need the second-order expansion of the variation of the state $Y$, i.e.,
\begin{equation}\label{equ 1.2}
Y^\varepsilon(t)=Y^*(t)+p(t)(X^{1,\varepsilon}(t)+X^{2,\varepsilon}(t))
+\frac{1}{2}P(t)(X^{1,\varepsilon}(t))^2+P_1(t)X^{1,\varepsilon}(t)X^{1,\varepsilon}(t-l)
+\breve{Y}(t)+o(\varepsilon),
\end{equation}
which is different to the classical case given by Hu \cite{Hu}:
\begin{equation}\label{equ 41.3}
Y^\varepsilon(t)=Y^*(t)+p(t)(X^{1,\varepsilon}(t)+X^{2,\varepsilon}(t))+\frac{1}{2}P(t)(X^{1,\varepsilon}(t))^2+\breve{Y}(t)+o(\varepsilon).
\end{equation}

(ii) Due to considering the general mean-field FBSDEs in our case,
the first- and second-order derivatives of the coefficients with respect to the measure, and some
corresponding estimates need to be handled carefully.  We argue that  the power of the second-order
derivative with respect to the measure ($\partial^2_{\nu\nu}$,  see Definition 2.3) are all the higher order of $\varepsilon$. So
only the mixed second-order derivative $\partial^2_{\nu a}$  is in force.  For this phenomenon, the
reader can refer to \cite{BLPR} Lemma 2.1 and note in section 6 \cite{BLM1} for more details. Besides,
we also note that the estimates given  by Buckdahn, Li, Ma  in Proposition 4.3 \cite{BLM1} originally for estimating the derivatives of the coefficients with respect to the measure is not enough for the recursive case with delay. Instead, we establish  a more general estimate, see Lemma 3.1.

The result of this paper can be summarized as follows:
Suppose $u^*(t)$ is the optimal control and $(X^*(t),Y^*(t),Z^*(t))$ is the optimal trajectory.
Consider the Hamiltonian
\begin{equation}\label{equ 5.1}
\begin{aligned}
&H(t,x,x',y,z,\nu,\mu,v;p,q,P)\\
&=pb(t,x,x',\nu,v)+q\sigma(t,x,x',\nu,v)
+\frac{1}{2}P\Big(\sigma(t,x,x',\nu,v)-\sigma(t,X^*(t),X^*(t-l), P_{X^*(t)}, u^*(t))\Big)^2\\
&\quad+f(t,x,x',y,z+p(\sigma(t,x,x',\nu,v)-\sigma(t,X^*(t),X^*(t-l), P_{X^*(t)}, u^*(t)),\mu,v),
\end{aligned}
\end{equation}
where $(t,x,x',y,z,\nu,\mu,v,p,q,P)\in[0,T]\times\mathbb{R}^4\times \mathcal{P}_2(\mathbb{R})\times \mathcal{P}_2(\mathbb{R}^2)\times U\times \mathbb{R}^3$.

If
$$\widehat{f}^*_{\mu_2}(t):= (\frac{\partial f}{\partial \mu})_2(t,\widehat{X}^*(t),\widehat{X}^*(t-l),\widehat{Y}^*(t),\widehat{Z}^*(t),
P_{(X^*(t),Y^*(t))},\widehat{u}^*(t);X^*(t), Y^*(t))>0,\  t\in[0,T], $$
$\widehat{P}\otimes P$-a.s.,
then there exist two pairs of stochastic
processes $(p,q), ((P,Q),$   $(P_1,Q_1))$, which are the solutions of the first- and second-order adjoint equations, separately,
such that  for given $l,T$ with $l<T\leq2l$,  the following  the following
inequality holds true:
\begin{equation}\label{equ 5.2}
\begin{aligned}
&H(t,X^*(t), X^*(t-l),Y^*(t),Z^*(t),P_{X^*(t)}, P_{(X^*(t),Y^*(t))},v,p(t),q(t),P(t))\\
&\geq H(t,X^*(t), X^*(t-l),Y^*(t),Z^*(t),P_{X^*(t)}, P_{(X^*(t),Y^*(t))},u^*(t),p(t),q(t),P(t)),\\
&\qquad\qquad\qquad\qquad\qquad\qquad\qquad\qquad\qquad\qquad\qquad \qquad \qquad \qquad t\in[0,T], \text{a.s.}\ \text{a.e.}
\end{aligned}
\end{equation}

There are two points needed to be stressed:

a) We  assume that the coefficients just depend on the law of the
state $(X_t,Y_t)$, $P_{(X_t,Y_t)}$, however, there is no essential change when the coefficients depend on the laws of the state and the delay,
$P_{(X_t,X_{t-l},Y_t,Y_{t-l})}$, meanwhile.

b) If $T\leq l$, our control problem reduces to the non delay case. Hence, in this paper
we present the case of one pointwise  delay, i.e., $l< T\leq 2l.$ But our method can
also be used to handle the pointwise delay case, i.e., $Nl< T\leq (N+1)l,\ N>1,\ N\in\mathbb{N}$, see Remark 3.7.

The paper is arranged as follows: Section 2 introduces some element theory of derivative in the Wasserstein Space and some usual spaces. In
Section 3,  we formulate the control problem, and establish the variational and adjoint equations. The expansion of $Y^\varepsilon$
is stated in Section 4. Section 5 is devoted to introducing the main result---Stochastic Maximum Principle. An illustration  is given
in Section 6. In the last Section we list some notations and show the expansion of $I_{22}$ accurately.

\section{{\protect \large {Preliminaries}}}

\subsection{Derivative in the Wasserstein Space}

 Let $\mathcal{P}_2(\mathbb{R}^d)$ denote the space of all probability measure $\nu$ on $(\mathbb{R}^d,\mathcal{B}(\mathbb{R}^d))$
 with finite second moment, which is endowed with the 2-Wassertein metric:
$$
\begin{aligned}
 W_2(\nu_1,\nu_2)&=\inf\Big\{
 (\int_{\mathbb{R}^d\times\mathbb{R}^d}|y_1-y_2|^2\rho(dy_1,dy_2))^\frac{1}{2},\ \rho\in \mathcal{P}_2(\mathbb{R}^{2d})\ \ \text{satisfying}\ \\
 &\qquad\qquad\qquad\qquad\rho(A\times \mathbb{R}^d)=\nu_1(A),\ \rho(\mathbb{R}^d\times B)=\nu_2(B),\ A,B\in \mathcal{B}(\mathbb{R})\Big\}.
 \end{aligned}
$$

Let us recall the differentiable of a function $h:\mathcal{P}_2(\mathbb{R}^d)\rightarrow\mathbb{R}$ with respect to a measure. The
idea is to identify a distribution $\mu\in\mathcal{P}_2(\mathbb{R}^d)$ with a random variable $\xi\in L^2(\mathcal{F};\mathbb{R}^d)$ so
that $\mu=P_{\xi}$.
\begin{definition}
A function $h:\mathcal{P}_2(\mathbb{R}^d)\rightarrow\mathbb{R}$ is called differentiable in $\mu_0\in\mathcal{P}_2(\mathbb{R}^d)$,
if the functional $  \tilde{h}: L^2(\mathcal{F};\mathbb{R}^d)\rightarrow \mathbb{R}$ defined $\tilde{h}(\xi):=h(P_\xi)$ is differentiable (in Fr\'{e}chet sense) in $\xi_0$ with
$P_{\xi_0}=\mu_0$.
\end{definition}
That means, for any $\zeta\in L^2( \mathcal{F} ;\mathbb{R}^d)$,
$$
\tilde{h}(\xi_0+\zeta)-\tilde{h}(\xi_0)=D\tilde{h}(\xi_0)(\zeta)+o(||\zeta||_{L_2}).$$
But $D\tilde{h}(\xi_0)$ is a continuous linear operator in $L^2(\mathcal{F} ;\mathbb{R}^d)$. From Riesz's Representation Theorem,
there exists a unique random variable $\eta_0\in L^2( \mathcal{F} ;\mathbb{R}^d)$ such that
$D\tilde{h}(\xi_0)(\zeta)=<\eta_0,\zeta>=E[\eta_0\zeta], \ \zeta\in L^2( \mathcal{F} ;\mathbb{R}^d)$, where
$<\cdot,\cdot>$ denotes the``dual product" on $L^2(\mathcal{F};\mathbb{R}^d)$. Later, Cardaliaguet proved that
$\eta_0$ can be written $\eta_0=\phi(\xi_0)$, where $\phi$ is a Borel measure function and depends only on the law of $\xi_0$,
but not $\xi_0$ itself. We define  $\partial_\nu h(P_{\xi_0};a):=\phi(a),\ a\in\mathbb{R}^d.$ Note that
$\partial_\nu h(P_{\xi_0};a)$ is only $P_{\xi_0}(da)$-a.e. uniquely determined. Moreover, we can
identify $D\tilde{h}(\xi_0)$ by
$$
D\tilde{h}(\xi_0)=\phi(\xi_0)=\partial_\nu h(P_{\xi_0};\xi_0).
$$

Since we shall consider the functions $h$, which are differentiable on the whole space $\mathcal{P}_2(\mathbb{R}^d)$, we assume that
$\widetilde{h}: L^2(\mathcal{F} ;\mathbb{R}^d)\rightarrow\mathbb{R}$ is Fr\'{e}chet differentiable at all
elements of $L^2(\mathcal{F} ;\mathbb{R}^d)$. Thereby, we have, for all $\xi\in L^2(\mathcal{F} ;\mathbb{R}^d)$ the derivative $\partial_\nu h(P_{\xi};a),\ a\in\mathbb{R}^d$ defined
$P_\xi$-a.e. Lemma 3.2 in \cite{CD2} allows to show that
if the Fr\'{e}chet derivative $\xi\mapsto D\tilde{h}(\xi)=\partial_\nu h(P_{\xi};\xi)$ is Lipschitz continuous
with a positive Lipschitz constant $L$, then  there exists for all $\xi\in L^2(\mathcal{F} ;\mathbb{R}^d)$ a $P_\xi$-version of
$\partial_\nu h(P_{\xi};\cdot):\mathbb{R}^d\rightarrow\mathbb{R}^d$  such that $\partial_\nu h(P_{\xi};\cdot)$ is Lipschitz
continuous with the same constant  $L$.

\begin{definition}
By $C_b^{1,1}(\mathcal{P}_2(\mathbb{R}^d))$ we denote the space of continuously differentiable function $h$ on $\mathcal{P}_2(\mathbb{R}^d)$
with the properties:

$\mathrm{i)}$ (boundness)  for all $\nu\in \mathcal{P}_2(\mathbb{R}^d),a\in\mathbb{R}^d,$  $\partial_\nu h(\nu;a)$ is bounded by a positive constant $L$;

$\mathrm{ii)}$(Lipschitz continuous) $\partial_\nu h(\cdot;\cdot)$ is Lipschitz continuous in $(\nu,a).$
\end{definition}

\begin{definition}
By $C_b^{2,1}(\mathcal{P}_2(\mathbb{R}^d))$ we denote the space of function $h\in C_b^{1,1}(\mathcal{P}_2(\mathbb{R}^d))$ with the properties:

$\mathrm {i)}$\ for each $\nu\in\mathcal{P}_2(\mathbb{R}^d), (\partial_\nu h)_j(\nu;\cdot)\in C_b^{1,1}(\mathbb{R}^d),\ 1\leq j\leq d,$

$\mathrm {ii)}$\ for each $a\in\mathbb{R}^d$, $(\partial_\nu h)_j(\cdot;a)\in C_b^{1,1}(\mathcal{P}_2(\mathbb{R}^d)),\ 1\leq j\leq d,$

$\mathrm {iii)}$ all the derivatives of $h$ up to order 2, $(\partial_\nu h, \partial^2_{\nu\nu}h, \partial^2_{\nu a}h)$, are
bounded and Lipschitz continuous with respect to all variables.
\end{definition}

Let us recall the second-order Taylor expansion of $h:\mathcal{P}_2(R^d)\rightarrow \mathbb{R}$, which plays an
important role in our analysis, see \cite{BLPR} for more details.
Let $(\widehat{\Omega},\widehat{\mathcal{F}},\widehat{P})$ be an independent copy of the probability
space $(\Omega,\mathcal{F},P)$. By $(\widehat{\xi}_0,\widehat{\zeta})$ defined on
$(\widehat{\Omega},\widehat{\mathcal{F}},\widehat{P})$ we denote the copy of the pair
$(\xi_0,\zeta)$ defined on $(\Omega,\mathcal{F},P)$, i.e., $\widehat{P}_{(\widehat{\xi}_0,\widehat{\zeta})}=P_{(\xi_0,\zeta)}$.
The expectation $\widehat{E}[\cdot]$
only works on the variables with ``hat". By  $(\cdot,\cdot)$ we denote
the inner product of $\mathbb{R}^d$. We
now consider the product space
$$
(\Omega\times \widehat{\Omega}, \mathcal{F}\otimes\widehat{\mathcal{F}}, P\otimes\widehat{P})
=(\Omega\times \Omega, \mathcal{F}\otimes\mathcal{F}, P\otimes P),
$$
and identify  $(\widehat{\xi}_0,\widehat{\zeta})(\widehat{\omega},\omega):=(\widehat{\xi}_0,\widehat{\zeta})(\widehat{\omega}),\
(\widehat{\omega},\omega)\in \widehat{\Omega}\times \Omega.$

Let $h\in C^{2,1}_b(\mathcal{P}_2(\mathbb{R}^d))$, one has
\begin{equation}\label{equ 1.1}
\begin{aligned}
&h(P_{\xi_0+\zeta})-h(P_{\xi_0})=\widetilde{h}(\xi_0+\zeta)-\widetilde{h}(\xi_0)\\
&=\int_0^1<D\widetilde{h}(\xi_0+\rho\zeta),\zeta>d\rho
=\int_0^1E[\partial_\nu h(P_{\xi_0+\rho\zeta}; \xi_0+\rho\zeta)\cdot\zeta]d\rho\\
&=E[\partial_\nu h(P_{\xi_0}; \xi_0)\cdot\zeta]
+\int_0^1E[(\partial_\nu h(P_{\xi_0+\rho\zeta}; \xi_0+\rho\zeta)
-\partial_\nu h(P_{\xi_0}; \xi_0))
\cdot\zeta]d\rho.
\end{aligned}
\end{equation}
However, it is clear that
\begin{equation}\label{equ 1.1-1}
\begin{aligned}
&\partial_\nu h(P_{\xi_0+\rho\zeta}; \xi_0+\rho\zeta)-\partial_\nu h(P_{\xi_0}; \xi_0)
=D\widetilde{\partial_\nu h}(\xi_0+\rho\xi)-D\widetilde{\partial_\nu h}(\xi_0)\\
&=\rho\int_0^1<D\widetilde{\partial_\nu h}(\xi_0+\lambda\rho\zeta;a),\zeta>d\lambda\Big|_{a=\xi_0+\rho\xi}
+\rho\int_0^1(\partial_a\partial_\nu h(P_{\xi_0};\xi_0+\lambda\rho\zeta), \zeta)d\lambda\\
&=\rho <D\widetilde{\partial_\nu h}(\xi_0;a),\zeta>\Big|_{a=\xi_0}
+\rho(\partial_a\partial_\nu h(P_{\xi_0};\xi_0), \zeta)+R_1(\rho,\zeta),
\end{aligned}
\end{equation}
where
$$
\begin{aligned}
R_1(\rho, \zeta):&=\rho\int_0^1<D\widetilde{\partial_\nu h}(\xi_0+\lambda\rho\zeta;a),\zeta>d\lambda\Big|_{a=\xi_0+\rho\xi}
-\rho <D\widetilde{\partial_\nu h}(\xi_0;a),\zeta>\Big|_{a=\xi_0}\\
&+\rho\int_0^1(\partial_a\partial_\nu h(P_{\xi_0};\xi_0+\lambda\rho\zeta), \zeta)d\lambda
-\rho(\partial_a\partial_\nu h(P_{\xi_0};\xi_0), \zeta).
\end{aligned}
$$
Due to
$$
D\widetilde{\partial_\nu h}(\xi_0;a)=\partial_\nu(\partial_\nu h(\cdot;a))(P_{\xi_0};\xi_0)
:=\partial^2_\nu h(P_{\xi_0};a,b)\Big|_{b=\xi_0},
$$
then we have
$$
\begin{aligned}
D\widetilde{\partial_\nu h}(\xi_0+\rho\xi)-D\widetilde{\partial_\nu h}(\xi_0)
=\rho \widehat{E}[\partial_\nu^2 h(P_{\xi_0};a,b)|_{b=\widehat{\xi}_0}\cdot \widehat{\zeta}]|_{a=\xi_0}
+\rho E[\partial_a\partial_\nu h(P_\xi;a)|_{a=\xi_0}\cdot\zeta]+R_1(\rho,\zeta).
\end{aligned}
$$
Accordingly, we obtain
\begin{equation}\label{equ 1.1-2}
\begin{aligned}
h(P_{\xi_0+\zeta})-h(P_{\xi_0})&=E[\partial_\nu h(P_{\xi_0}; \xi_0)\cdot\zeta]
+\frac{1}{2}E\widehat{E}\Big\{tr[\partial_\nu^2h(P_{\xi_0};\xi_0,\widehat{\xi}_0)\zeta\otimes\widehat{\zeta}]\Big\}\\
&\quad+\frac{1}{2}E\Big\{tr[\partial_a\partial_\nu h(P_{\xi_0};\xi_0)\zeta\otimes\zeta]\Big\}
+R(\zeta),
\end{aligned}
\end{equation}
where $R(\zeta)=\int_0^1E[R_1(\rho,\zeta)\cdot\zeta]d\rho$. Clearly, $|R(\zeta)|\leq CE[|\zeta|^3\wedge|\zeta|^2]$.

For arbitrary $\nu_0\in \mathcal{P}_2(\mathbb{R}^d)$ with $\nu_0=P_{\xi_0}$, and
$\zeta\in L^2(\mathcal{F};\mathbb{R}^d)$, we  call $D^2_\zeta h(\nu_0)$ the second
order derivative of $h$ at $\nu_0\in \mathcal{P}_2(\mathbb{R}^d)$ in the direction $\zeta$,
where $D^2_\zeta h(\nu_0)$ is given by
 $$
 \begin{aligned}
 D^2_\zeta h(\nu_0)=D^2_\zeta h(P_{\xi_0})=E\widehat{E}\Big\{tr[\partial_\nu^2h(P_{\xi_0};\xi_0,\widehat{\xi}_0)\zeta\otimes\widehat{\zeta}]\Big\}
 +E\Big\{tr[\partial_y\partial_\nu h(P_{\xi_0};\xi_0)\zeta\otimes\zeta]\Big\}.
\end{aligned}
$$
Consequently, (\ref{equ 1.1-2}) can be written as
$$
h(P_{\xi_0+\zeta})-h(P_{\xi_0})=D_\zeta h(P_{\xi_0})+\frac{1}{2}D^2_\zeta h(P_{\xi_0})+R(\zeta),
$$
where $|R(\zeta)|\leq CE[|\zeta|^3\wedge|\zeta|^2]=O(||\zeta||_2^2)$.

\subsection{Function spaces}

Let $(\Omega, \mathcal{F},P; \mathbb{F}=\{\mathcal{F}_t\}_{t\geq0})$ be a complete, filtered probability space, on which a
1-dimensional Brownian motion $W$ is defined. $T>0$ is a fixed time horizon, $l>0$ is a time delay parameter
and $\mathcal{F}^*$ is a sub-$\sigma$-field of $\mathcal{F}$
satisfying the following  two properties:

i) $\mathcal{F}^*$ is independent of the Brownian motion $W$;

ii) $\mathcal{P}_2(\mathbb{R}^d)=\{P_\xi,\xi\in L^2(\mathcal{F}^*;\mathbb{R^d})\}$.\\
Let $\mathcal{F}_t=\sigma\{B_s,0\leq s\leq t\}\vee \mathcal{F}^*\vee \mathcal{N}$,
where $\mathcal{N}$ is the set of all $P$-null subset. By $U^\intercal$ we denote
the transpose of any vector or matrix $U$.

Let $\ell\geq1$, we introduce the following spaces, which are used frequently later:

$\bullet$ $L^\ell(\mathcal{F}_t;\mathbb{R}^d)$ denotes the set of
$\mathbb{R}^d$-valued, $\mathcal{F}_t$-measurable random variable $\xi$ with
$$
||\xi||_{L^\ell}:=(E[|\xi|^\ell])^\frac{1}{\ell}<+\infty;
$$

$\bullet$ $\mathcal{S}_{\mathbb{F}}^\ell(t,T;\mathbb{R}^d)$ denotes the set of $\mathbb{R}^d$-valued,  $\mathcal{F}_s$-predictable processes $\psi$
on $[t,T]$, such that
$$
||\psi||_{\mathcal{S}^\ell_{\mathbb{F}}}=\Big\{E[\sup_{s\in[t,T]}|\psi(s)|^\ell]\Big\}^{\frac{1}{\ell}}<+\infty;
$$

$\bullet$ $\mathcal{H}_{\mathbb{F}}^\ell(t,T;\mathbb{R}^d)$ denotes the set of $\mathbb{R}^d$-valued,  $\mathcal{F}_s$-progressively measurable processes $\psi$
on $[t,T]$, such that
$$||\psi||_{\mathcal{H}^\ell_{\mathbb{F}}}=\Big\{E[(\int_t^T|\psi(s)|^2ds)^{\frac{\ell}{2}} ]\Big\}^{\frac{1}{\ell}}<+\infty.$$

\section{{\protect \large {Formulation}}}

Suppose $U$ is a subset of $\mathbb{R}$, not necessarily convex. Let $\mathcal{U}_{ad}$ be the set of all admissible controls, that is,
the set of all  $\mathcal{F}_t$-adapted stochastic processes $v:[0,T]\times \Omega\rightarrow U$ with for all $\beta\geq1$,
$$E[\sup_{t\in[0,T]}|v(t)|^\beta dt]<+\infty.$$

Let the mappings
$$
\begin{aligned}
&(b,\sigma)(t,x,x',\nu,v):[0,T]\times \mathbb{R}\times \mathbb{R}\times \mathcal{P}_2(\mathbb{R})\times U\rightarrow \mathbb{R}, \ \Phi(x,\nu): \mathbb{R}\times\mathcal{P}_2(\mathbb{R})\rightarrow \mathbb{R},\\
&f(t,x,x',y,z,\mu,v):[0,T]\times \mathbb{R}\times \mathbb{R}\times \mathbb{R}\times  \mathbb{R}\times \mathcal{P}_2(\mathbb{R}^2)\times U\rightarrow \mathbb{R},\\
\end{aligned}
$$
satisfy the following assumptions:\\

\underline{(H3.1)} $b,\sigma,\Phi$ are bounded by $C(1+|x|+|x'|+(\int_{\mathbb{R}}x^2\nu(dx))^\frac{1}{2}+|v|)$,
 $f$ is bounded by $C(1+|x|+|x'|+|y|+|z|+(\int_{\mathbb{R}}x^2\nu(dx))^\frac{1}{2}+|v|)$, $b,\sigma,f$ and all
 the derivatives of $b,\sigma,f$ up to order 2 are Lipschitz continuous in $v$.\\

\underline{(H3.2)} For each $(t,v)\in[0,T]\times U,$  $b,\sigma\in C_b^{1,1}(\mathbb{R}^2\times \mathcal{P}_2(\mathbb{R})),\
 f\in C_b^{1,1}(\mathbb{R}^4\times \mathcal{P}_2(\mathbb{R}^2)),\
 \Phi\in C_b^{1,1}(\mathbb{R}\times \mathcal{P}_2(\mathbb{R}))$, i.e.,

i)\ For each $(t,x,x',y,z,v)\in [0,T]\times \mathbb{R}^4\times U$,
$(b,\sigma)(t,x,x',\cdot,v)$ belong to $C_b^{1,1}(\mathcal{P}_2(\mathbb{R})),$
$\Phi(x,\cdot)$ belongs to $C_b^{1,1}(\mathcal{P}_2(\mathbb{R})),$
$f(t,x,x',y,z,\cdot,v)$ belongs to $C_b^{1,1}(\mathcal{P}_2(\mathbb{R}^2))$;

ii)\ For each $(t,\nu,\mu,v)\in[0,T]\times \mathcal{P}_2(\mathbb{R})\times \mathcal{P}_2(\mathbb{R}^2)\times U$,
$(b,\sigma)(t,\cdot,\cdot,\nu,v)$ belong to $C_b^{1}(\mathbb{R}^2)$,
$\Phi (\cdot,\nu )$ belongs to $C_b^{1}(\mathbb{R}^2)$,
and
$f(t,\cdot,\cdot,\cdot,\cdot,\mu,v)$ belongs to $C_b^{1}(\mathbb{R}^4)$;

iii)\ The derivatives  of $b,\sigma, \Phi,f$ up to order $1$ are Lipschitz continuous and bounded.\\

\underline{(H3.3)} For each $(t,v)\in[0,T]\times U,$ $(b,\sigma)\in C_b^{2,1}(\mathbb{R}^2\times \mathcal{P}_2(\mathbb{R})),$
$f\in C_b^{2,1}(\mathbb{R}^4\times \mathcal{P}_2(\mathbb{R}^2))$,  $\Phi\in C_b^{2,1}(\mathbb{R}\times \mathcal{P}_2(\mathbb{R}))$,
i.e., $b,\sigma\in C_b^{1,1}(\mathbb{R}^2\times \mathcal{P}_2(\mathbb{R})),\
 f\in C_b^{1,1}(\mathbb{R}^4\times \mathcal{P}_2(\mathbb{R}^2)),\
 \Phi\in C_b^{1,1}(\mathbb{R}\times \mathcal{P}_2(\mathbb{R}))$ satisfy:\\
\indent i)\ For each $(t,v)\in[0,T]\times U$, $\partial_k\phi(t,\cdot,\cdot,\cdot,v)$ belongs to $C_b^{1,1}(\mathbb{R}\times \mathcal{P}_2(\mathbb{R})),$  $\phi=b,\sigma,\ k=x,x',$
$\partial_x\Phi(\cdot,\cdot)$ belongs to $C_b^{1,1}(\mathbb{R}\times \mathcal{P}_2(\mathbb{R})),$
$\partial_kf(t,\cdot,\cdot,\cdot,\cdot,\cdot,v)$ belongs to  $C_b^{1,1}(\mathbb{R}^4\times \mathcal{P}_2(\mathbb{R}^2)),\ k=x,x',y,z$;

ii) For each $(t,v)\in[0,T]\times U,$ $\partial_\nu\phi(t,\cdot,\cdot,v;\cdot)$ belongs to $C_b^{1,1}(\mathbb{R}^2\times \mathcal{P}_2(\mathbb{R})\times \mathbb{R}), \phi=b,\sigma,$ $\partial_\nu\Phi(\cdot,\cdot;\cdot)$ belongs to
$C_b^{1,1}(\mathbb{R}^2\times \mathcal{P}_2(\mathbb{R})\times \mathbb{R})$,
$(\partial_{\mu }f)_j(t,\cdot,\cdot,\cdot,\cdot,\cdot,v;\cdot,\cdot)$ belongs to $C_b^{1,1}(\mathbb{R}^4\times \mathcal{P}_2(\mathbb{R}^2)\times \mathbb{R}^2),\ j=1,2$;

iii) All the derivatives of $b,\sigma, \Phi,f$ up to order $2$ are Lipschitz continuous and bounded.\\

We consider the system (\ref{equ 3.1})
 and the cost functional
\begin{equation}\label{3.1-1}
\begin{aligned}
J(v)=Y^v(0).
\end{aligned}
\end{equation}
The purpose of this paper is to investigate the  necessary condition of the optimal control $u^*$.

In subsequent sections, if no confusion, by $u^*$ and $(X^*,Y^*,Z^*)$ we always denote the optimal control and the corresponding trajectory,
by $\rho:(0,+\infty)\rightarrow(0,+\infty)$ denote a measurable function with property: $\rho(\varepsilon)\rightarrow0$ as $\varepsilon\downarrow0$, as well as
by $C$ denote a positive constant. $\rho$ and $C$ may be different from line to line.

\subsection{Variational equations}
In this subsection we investigate the differentiation of the state process $X^\varepsilon$ by the spike variation
 method, which we now introduce.

Let $\varepsilon>0$ and $E_\varepsilon$ be a Borel subset  of $[0,T]$ with the Lebesgue measure $\lambda(E_\varepsilon)=\varepsilon.$
Suppose $v\in \mathcal{U}_{ad}$, the spike variation of $u^*$ is described as follows:
for $t\in[0,T]$,
\begin{equation}\label{equ 3.1-1}
v^\varepsilon(t)=v(t)\mathbbm{1}_{E_\varepsilon}(t)+u^*(t)\mathbbm{1}_{(E_\varepsilon)^C}(t).
\end{equation}
Accordingly, by $(X^\varepsilon,Y^\varepsilon,Z^\varepsilon):=(X^{v^\varepsilon},Y^{v^\varepsilon},Z^{v^\varepsilon})$
we denote the solution of the system (\ref{equ 3.1}) with $v^\varepsilon$.

Let  $ (\overline{\Omega} , \overline{\mathcal{F}},\overline{P}),
(\widehat{\Omega} , \widehat{\mathcal{F}},\widehat{P}),\ (\widetilde{\Omega} , \widetilde{\mathcal{F}},\widetilde{P})$
be three complete probability spaces.
By $W,\ E,\ \xi,\ X(\cdot)$ with ``bar", ``hat", ``tilde" we denote the Brown motion, expectation, random variable, stochastic
process defined on the spaces $(\overline{\Omega} , \overline{\mathcal{F}},\overline{P}),
(\widehat{\Omega} , \widehat{\mathcal{F}},\widehat{P}),\ (\widetilde{\Omega} , \widetilde{\mathcal{F}},\widetilde{P})$, respectively.
Note that the spaces $(\Omega , \mathcal{F},P),(\overline{\Omega} , \overline{\mathcal{F}},\overline{P}),
(\widehat{\Omega}, \widehat{\mathcal{F}},\widehat{P})$ and $(\widetilde{\Omega}, \widetilde{\mathcal{F}},\widetilde{P})$ are independent.\\
For $  \phi=b, \sigma, b_x, \sigma_x, b_{x'}, \sigma_{x'}, \psi=b, \sigma$, define
\begin{equation}\label{equ 3.1-1-1}
\begin{aligned}
&\Delta\phi(t):=\phi(t,X^*(t),X^*(t-l),P_{X^*(t)},v(t))-\phi(t,X^*(t),X^*(t-l),P_{X^*(t)},u^*(t)),\\
&(\psi_x,\psi_{xx})(t):=(\frac{\partial\psi}{\partial x} ,\frac{\partial^2\psi}{\partial x^2})(t,X^*(t),X^*(t-l),P_{X^*(t)},u^*(t)),\\
&(\widehat{\psi}_\nu,\widehat{\psi}_{\nu a})(t):=(\frac{\partial\psi}{\partial \nu} ,\frac{\partial^2\psi}{\partial \nu \partial a})(t,X^*(t),X^*(t-l),P_{X^*(t)},u^*(t);\widehat{X}^*(t)),\\
&(\widehat{\psi}^*_\nu,\widehat{\psi}^*_{\nu a}) (t):=(\frac{\partial\phi}{\partial\nu},\frac{\partial^2\phi}{\partial\nu \partial a})
(t,\widehat{X}^*(t),\widehat{X}^*(t-l),P_{X^*(t)},\widehat{u}^*(t);X^*(t)),\\
&\widetilde{\widehat{\psi}}_{\nu \nu}(t):=\frac{\partial^2\psi}{\partial \nu^2}(t,X^*(t),X^*(t-l),P_{X^*(t)},u^*(t);\widehat {X}^*(t), \widetilde{X}^*(t)).
\end{aligned}
\end{equation}

Inspired by the work of Peng \cite{Peng1}, when $\sigma$ depending on control and the control set $U$ being
not necessarily convex, for each $\varepsilon>0,$ we can find two stochastic processes $X^{1,\varepsilon}, X^{2,\varepsilon}$
such that
\begin{equation}\label{equ 3.1-2}
\begin{aligned}
&\mathrm{i)}\   X^\varepsilon(t)-X^*(t)-X^{1,\varepsilon}(t)=O(\varepsilon),\\
&\mathrm{ii)}\  X^\varepsilon(t)-X^*(t)-X^{1,\varepsilon}(t)-X^{2,\varepsilon}(t)=o(\varepsilon).
\end{aligned}
\end{equation}
In our case, it is easy to check that $X^{1,\varepsilon}, X^{2,\varepsilon}$ satisfy the following two
variational equations.
The first-order variational equation is
\begin{equation}\label{equ 3.2}
  \left\{
 \begin{aligned}
dX^{1,\varepsilon}(t)&=\Big\{ b_x(t)X^{1,\varepsilon}(t)+b_{x'}(t)X^{1,\varepsilon}(t-l)
          +\widehat{E}[\widehat{b}_\nu(t)\widehat{X}^{1,\varepsilon}(t)]+\Delta b(t)\mathbbm{1}_{E_\varepsilon}(t)\Big\}dt\\
        &\quad+\Big\{ \sigma_x(t)X^{1,\varepsilon}(t)+\sigma_{x'}(t)X^{1,\varepsilon}(t-l)
        +\widehat{E}[\widehat{\sigma}_\nu(t)\widehat{X}^{1,\varepsilon}(t)]\\
        &\qquad\qquad\qquad\qquad\qquad\qquad\qquad\qquad\qquad\       +\Delta \sigma(t)\mathbbm{1}_{E_\varepsilon}(t)\Big\}dW(t),\ t\in[0,T],\\
X^{1,\varepsilon}(t)&=0,\ t\in[-l,0],
\end{aligned}
\right.
\end{equation}
and the second-order variational equation, which is a linear mean-field forward SDE with delay, has the
following form:
\begin{equation}\label{equ 3.3}
\begin{aligned}
dX^{2,\varepsilon}(t)&=\Big\{ b_x(t)X^{2,\varepsilon}(t)+b_{x'}(t)X^{2,\varepsilon}(t-l)
                  +\widehat{E}[\widehat{b}_\nu(t)\widehat{X}^{2,\varepsilon}(t)]
                 +b_{xx'}(t)X^{1,
                 \varepsilon}(t)X^{1,\varepsilon}
                 (t-l)\\
           &\quad +\frac{1}{2}\Big(b_{xx}(t)(X^{1,\varepsilon}(t))^2+\widehat{E}[\widehat{b}_{\nu a}(t)(\widehat{X}^{1,\varepsilon}(t))^2]+b_{x'x'}(t)(X^{1,\varepsilon}(t-l))^2\Big)\\
           &\quad+\Big(\Delta b_x(t)X^{1,\varepsilon}(t)+\Delta b_{x'}(t)X^{1,\varepsilon}(t-l)
       +\widehat{E}[\Delta \widehat{b}_\nu(t)\widehat{X}^{1,\varepsilon}(t)]\Big)\mathbbm{1}_{E_\varepsilon}(t) \Big\}dt \\
\end{aligned}
\end{equation}
$$
\begin{aligned}
&\quad+\Big\{ \sigma_x(t)X^{2,\varepsilon}(t)+\sigma_{x'}(t)X^{2,\varepsilon}(t-l)
                  +\widehat{E}[\widehat{\sigma}_\nu(t)\widehat{X}^{2,\varepsilon}(t)]
                 +\sigma_{xx'}(t)X^{1,\varepsilon}(t)X^{1,\varepsilon}(t-l)\\
           &\quad +\frac{1}{2}\Big(\sigma_{xx}(t)(X^{1,\varepsilon}(t))^2+\widehat{E}[\widehat{\sigma}_{\nu a}(t)(\widehat{X}^{1,\varepsilon}(t))^2]+\sigma_{x'x'}(t)(X^{1,\varepsilon}(t-l))^2\Big)\\
           &\quad+\Big(\Delta \sigma_x(t)X^{1,\varepsilon}(t)+\Delta \sigma_{x'}(t)X^{1,\varepsilon}(t-l)
        +\widehat{E}[\Delta \widehat{\sigma}_\nu(t)\widehat{X}^{1,\varepsilon}(t)]\Big)\mathbbm{1}_{E_\varepsilon}(t) \Big\}dW_t,\ t\in[0,T],\\
X^{2,\varepsilon}(t)&=0,\ t\in[-l,0].
\end{aligned}
$$

Under the assumptions (H3.1)-(H3.3), the equation (\ref{equ 3.2}) and  equation (\ref{equ 3.3}) exist
unique solutions $X^{1,\varepsilon}=(X^{1,\varepsilon}(t))_{t\in[-l,T]}\in \mathcal{S}^\beta_{\mathbb{F}}(-l,T;\mathbb{R})$ ($\mathcal{S}^\beta_{\mathbb{F}}(-l,T)$, for short) and
$X^{2,\varepsilon}=(X^{2,\varepsilon}(t))_{t\in[-l,T]}\in \mathcal{S}^\beta_{\mathbb{F}}(-l,T)$
with $\beta\geq 2,$
respectively, see
Shen, Meng and Shi \cite{SMS}.

Let us  first introduce a useful estimate in studying stochastic
maximum principle especially driven by general mean-field forward-backward control systems with delay.

\begin{lemma}\label{le 3.1}
Under the assumptions (H3.1), (H3.2), let $(\overline{\Omega},\overline{\mathcal{F}},\overline{P})$
be an intermediate probability space and independent of space $(\Omega, \mathcal{F},P)$, and
 let $(\Psi_1(\omega,\overline{\omega},t))_{t\in[-l,T]}, (\Psi_2(\overline{\omega},t))_{t\in[-l,T]}$ be two stochastic
 processes defined on the product space $(\Omega\times\overline{\Omega},\mathcal{F}\times\overline{\mathcal{F}},P\otimes \overline{P})$
and the space $(\overline{\Omega},\overline{\mathcal{F}},\overline{P})$, respectively. Moreover, assume
$\Psi_i,\ i=1,2$
satisfies the
following properties:\\
\indent$\mathrm{i)}$\ There exists a constant $C>0$, such that for $t\in[-l,T],$ $|\Psi_1(\omega,\overline{\omega},t)|\leq C, P\otimes \overline{P}$-a.s.\\
\indent$\mathrm{ii)}$\ For $\beta\geq1$,\ $\overline{E}[\sup_{t\in[-l,T]}|\Psi_2(\overline{\omega},t)|^{2\beta}]\leq C_\beta$.\\
Then
\begin{equation}\label{equ 3.3-1}
E\Big[\int_{-l}^T|\overline{E}[\Psi_1(\omega,\overline{\omega},t)\Psi_2(\overline{\omega},t)\overline{X}^{1,\varepsilon}(t)]|^4 dt\Big]\leq \varepsilon^2\rho(\varepsilon).
\end{equation}
\end{lemma}

\begin{proof}
From the definitions of $\hat{b}_\nu(t),\hat{\sigma}_\nu(t)$ (see (\ref{equ 3.1-1-1})), for
the convenience of edit we denote, for $h=b,\sigma$,
\begin{equation}\label{equ 3.3-2}
\overline{\widehat{h}}_{\nu}(t)=\frac{\partial h}{\partial \nu}(t,\overline{X}^*(t),\overline{X}^*(t-l),P_{X^*(t)},\overline{u}^*(t);\widehat{X}^*(t)).
\end{equation}

Recall in this paper we consider the case of one pointwise delay, i.e., $l< T\leq 2l$.
From the property of integral interval additivity,  it is easy to see
$$
\begin{aligned}
&E\Big[\int_{-l}^T|\overline{E}[\Psi_1(\omega,\overline{\omega},t)\Psi_2(\overline{\omega},t)\overline{X}^{1,\varepsilon}(t)]|^4 dt\Big]
=E\Big[\int_{-l}^0|\overline{E}[\Psi_1(\omega,\overline{\omega},t)\Psi_2(\overline{\omega},t)\overline{X}^{1,\varepsilon}(t)]|^4 dt\Big]\\
&+E\Big[\int_{0}^l|\overline{E}[\Psi_1(\omega,\overline{\omega},t)\Psi_2(\overline{\omega},t)\overline{X}^{1,\varepsilon}(t)]|^4 dt\Big]
+E\Big[\int_{l}^T|\overline{E}[\Psi_1(\omega,\overline{\omega},t)\Psi_2(\overline{\omega},t)\overline{X}^{1,\varepsilon}(t)]|^4 dt\Big].
\end{aligned}
$$
Hence, in what follows we will prove (\ref{equ 3.3-1}) in three cases.

\textbf{Case 1}\  If $-l\leq t<0,$ $X^{1,\varepsilon}(t)\equiv0$. Obviously,
\begin{equation}\label{equ 3.3-3}
E\Big[\int_{-l}^0|\overline{E}[\Psi_1(\omega,\overline{\omega},t)\Psi_2(\overline{\omega},t)\overline{X}^{1,\varepsilon}(t)]|^4 dt\Big]\leq \varepsilon^2\rho(\varepsilon).
\end{equation}

\textbf{Case 2}\   If $0\leq t< l$, $X^{1,\varepsilon}(t-l)\equiv0$ which means that the delay terms will
vanish. Hence, the equation (\ref{equ 3.1}) becomes a mean-field FBSDE without delay, and the corresponding
first-order variational equation (\ref{equ 3.2}) can be rewritten as
\begin{equation}\label{equ 3.3-4}
  \left\{
 \begin{aligned}
dX^{1,\varepsilon}(t)&=\Big\{ b_x(t)X^{1,\varepsilon}(t)
          +\widehat{E}[\widehat{b}_\nu(t)\widehat{X}^{1,\varepsilon}(t)]+\Delta b(t)\mathbbm{1}_{E_\varepsilon}(t)\Big\}dt\\
        &\quad+\Big\{ \sigma_x(t)X^{1,\varepsilon}(t)
        +\widehat{E}[\widehat{\sigma}_\nu(t)\widehat{X}^{1,\varepsilon}(t)]
      +\Delta \sigma(t)\mathbbm{1}_{E_\varepsilon}(t)\Big\}dW(t),\ t\in[0,T],\\
X^{1,\varepsilon}(t)&=0,\ t\in[-l,0].
\end{aligned}
\right.
\end{equation}

Let us consider the following SDE:
\begin{equation}\label{equ 3.3-5}
  \left\{
 \begin{aligned}
d\eta^1(t)&=-(b_x(t)-|\sigma_x(t)|^2)\eta^1(t)dt-\sigma_x(t)\eta^1(t)dW(t),\ t\in[0,l],\\
\eta^1(0)&=1,\ \eta^1(t)=0,\ t\in[-l,0).
\end{aligned}
\right.
\end{equation}
Obviously,
\begin{equation}\label{equ 3.3-5}
  \left\{
 \begin{aligned}
\eta^1(t)&=\exp\Big\{-\int_0^t\sigma_x(s)dW(s)
-\int_0^t(b_x(s)-\frac{1}{2}|\sigma_x(s)|^2)ds\Big\},\ t\in[0,l],\\
\eta^1(t)&=0,\ t\in[-l,0).
\end{aligned}
\right.
\end{equation}
We define
$\rho^1(t)=(\eta^1(t))^{-1}=\exp\Big\{
\int_0^t\sigma_x(s)dW(s)+\int_0^t(b_x(s)-\frac{1}{2}|\sigma_x(s)|^2)ds\Big\},\ t\in[0,l]$
and $\rho^1(t)=0,\ t\in[-l,0).$

Thanks to the boundness of $b_x,\ \sigma_x$, one can check that for $\beta\geq1$, there
exists a positive constant $C_\beta$ such that
\begin{equation}\label{equ 3.3-6}
E\Big[\sup_{t\in[-l,l]}(|\rho^1(t)|^\beta+|\eta^1(t)|^\beta)\Big]\leq C_\beta.
\end{equation}
Applying It\^{o}'s formula to $\eta^1(t)X^{1,\varepsilon}(t)$ over $[0,l]$, we have,
for $t\in[0,l]$,
\begin{equation}\label{equ 3.3-7}
\begin{aligned}
X^{1,\varepsilon}(t)=\rho^1(t)\Delta_1(t)+\Delta_2(t),
\end{aligned}
\end{equation}
where
$$
\begin{aligned}
\Delta_1(t)&:=\int_0^t\eta^1(s)(\widehat{E}[\widehat{\sigma}_\nu(s)\widehat{X}^{1,\varepsilon}(s)]
+\Delta\sigma(s)\mathbbm{1}_{E_\varepsilon}(s))dW(s),\\
\Delta_2(t)&:=\rho^1(t)\int_0^t\eta^1(s)(\widehat{E}[\widehat{b}_\nu(s)\widehat{X}^{1,\varepsilon}(s)]
+\Delta b(s)\mathbbm{1}_{E_\varepsilon}(s))ds\\
&\quad-\rho^1(t)\int_0^t\eta^1(s)\sigma_x(s)(\widehat{E}[\widehat{\sigma}_\nu(s)\widehat{X}^{1,\varepsilon}(s)]
+\Delta \sigma(s)\mathbbm{1}_{E_\varepsilon}(s))ds\Big\}.
\end{aligned}
$$
Therefore,
$$
\begin{aligned}
&\overline{E}[\Psi_1(\omega,\overline{\omega},t)\Psi_2(\overline{\omega},t)\overline{X}^{1,\varepsilon}(t)]\\
&=\overline{E}[\Psi_1(\omega,\overline{\omega},t)\Psi_2(\overline{\omega},t)\overline{\rho}^1(t)\overline{\Delta}_1(t)]
+\overline{E}[\Psi_1(\omega,\overline{\omega},t)\Psi_2(\overline{\omega},t)\overline{\Delta}_2(t)]
:=\Theta_1(t,\varepsilon)+\Theta_2(t,\varepsilon).
\end{aligned}
$$

Next we shall analyse $\Theta_1(t,\varepsilon)$ and $\Theta_2(t,\varepsilon)$ one by one.

First, for each $\omega\in \Omega$, we consider the process
$\Psi_1^\omega(\overline{\omega},t)\Psi_2(\overline{\omega},t)\overline{\rho}^1(t):=
\Psi_1(\omega,\overline{\omega},t)\Psi_2(\overline{\omega},t)\overline{\rho}^1(t)$.
For each given $t\in[0,l]$, notice
$$
\begin{aligned}
&\overline{E}[|\Psi_1^\omega(\overline{\omega},t)\Psi_2(\overline{\omega},t)\overline{\rho}^1(t)|^2]
\leq C\overline{E}[\Psi_2(\overline{\omega},t)\overline{\rho}^1(t)|^2]\\
&\leq C\Big\{\overline{E}[|\Psi_2(\overline{\omega},t)|^4]    \Big\}^\frac{1}{2}
\Big\{\overline{E}[|\overline{\rho}^1(t)|^4]\Big\}^\frac{1}{2}\leq C,
\end{aligned}
$$
and recall $\mathbb{F}=\mathcal{F}^*\vee\mathbb{F}^B$, then from It\^{o}'s Martingale
Representation Theorem we know that there exists a unique $\overline{\gamma}^1_t(\cdot)\in \mathcal{H}^2_{\overline{\mathbb{F}}}(0,t)$
such that
\begin{equation}\label{equ 3.3-8}
\Psi_1(\omega,\overline{\omega},t)\Psi_2(\overline{\omega},t)\overline{\rho}^1(t)
=\overline{E}[\Psi_1(\omega,\overline{\omega},t)\Psi_2(\overline{\omega},t)\overline{\rho}^1(t)]
+\int_0^t\overline{\gamma}^1_t(s)d\overline{W}_s,\ t\in[0,l].
\end{equation}
Moreover, $(\overline{\gamma}^1_t)_{t\in[0,l]}$ enjoys the following property: For $\beta\geq1$, there exists a
$C_\beta>0$, such that
\begin{equation}\label{equ 3.3-9}
\overline{E}[(\int_0^t|\overline{\gamma}^1_t(s)|^2ds)^\frac{\beta}{2}]\leq C_\beta.
\end{equation}
Indeed, for $\beta>1$, with the help of the Burkholder-Davis-Gundy, Doob's maximal inequality
and H\"{o}lder inequalities, it follows
\begin{equation}\label{equ 3.3-10}
\begin{aligned}
\overline{E}[(\int_0^t|\overline{\gamma}^1_t(s)|^2ds)^\frac{\beta}{2}]
&\leq C_\beta \overline{E}\Big[\sup_{s\in[0,t]}|\int_0^s\overline{\gamma}^1_t(r)d\overline{W}(r)|^\beta\Big]\\
&\leq C_\beta (\frac{\beta}{\beta-1})^\beta\overline{E}[|\int_0^t\overline{\gamma}^1_t(r)d\overline{W}(r)|^\beta]\\
&\leq C_\beta \overline{E}[|\Psi_1(\omega,\overline{\omega},t)\Psi_2(\overline{\omega},t)\overline{\rho}^1(t)
-\overline{E}[\Psi_1(\omega,\overline{\omega},t)\Psi_2(\overline{\omega},t)\overline{\rho}^1(t)]|^\beta]\\
&\leq C_\beta  \overline{E}[|\Psi_1(\omega,\overline{\omega},t)\Psi_2(\overline{\omega},t)\overline{\rho}^1(t)|^\beta]
\leq C_\beta.
\end{aligned}
\end{equation}
For $\beta=1$, (\ref{equ 3.3-9}) can be obtained by H\"{o}lder inequality.

Let us now calculate $\Theta_1(t,\varepsilon)$. Clearly, according to (\ref{equ 3.3-8})
one has
\begin{equation}\label{equ 3.3-11}
\begin{aligned}
\Theta_1(t,\varepsilon)
&=\overline{E}[\int_0^t\overline{\gamma}^1_t(s)\overline{\eta}^1(s)\widehat{E}[\overline{\widehat{\sigma}}_\nu(s)\widehat{X}^{1,\varepsilon}(s)]ds]
+\overline{E}[\int_0^t\overline{\gamma}^1_t(s)\overline{\eta}^1(s)\Delta\overline{\sigma}(s) \mathbbm{1}_{E_\varepsilon}(s)ds]\\
&:=\Theta_{1,1}(t,\varepsilon)+\Theta_{1,2}(t,\varepsilon).
\end{aligned}
\end{equation}
On the one hand, H\"{o}lder inequality allows to show
\begin{equation}\label{equ 3.3-12}
\begin{aligned}
|\Theta_{1,1}(t,\varepsilon)|^2
&\leq \overline{E}\Big[\sup_{s\in[0,l]}|\overline{\eta}^1(s)|^2
\int_0^t|\overline{\gamma}^1_t(s)|^2ds \int_0^t|\widehat{E}[\overline{\widehat{\sigma}}_\nu(s)\widehat{X}^{1,\varepsilon}(s)]|^2ds\Big]\\
&\leq C\Big\{\overline{E}\Big[ \int_0^t|\widehat{E}[\overline{\widehat{\sigma}}_\nu(s)\widehat{X}^{1,\varepsilon}(s)]|^4ds\Big]\Big\}^\frac{1}{2}
\cdot
\Big\{E[\sup_{s\in[0,l]}|\overline{\eta}^1(s)|^8 ]\Big\}^\frac{1}{4}
\cdot\\
&\qquad\qquad\qquad\qquad\qquad\qquad\qquad\qquad\qquad
\Big\{\overline{E}[(\int_0^t|\overline{\gamma}^1_t(s)|^2ds)^4]\Big\}^\frac{1}{4}\\
&\leq C\Big\{\overline{E}\Big[ \int_0^t|\widehat{E}[\overline{\widehat{\sigma}}_\nu(s)\widehat{X}^{1,\varepsilon}(s)]|^4ds\Big]\Big\}^\frac{1}{2}.
\end{aligned}
\end{equation}
On the other hand, from H\"{o}lder inequality, again, it follows
\begin{equation}\label{equ 3.3-13}
\begin{aligned}
|\Theta_{1,2}(t,\varepsilon)|^2&=|\overline{E}[\int_0^t\overline{\gamma}^1_t(s)\overline{\eta}^1(s)\Delta\overline{\sigma}(s)
\mathbbm{1}_{E_\varepsilon}(s)ds]|^2\\
&\leq C\overline{E}\Big[ \sup_{s\in[0,l]}|\overline{\eta}^1(s)|^2\Big(\int_0^t|\overline{\gamma}^1_t(s)|\cdot|\overline{v}(s)-\overline{u}^*(s)|\mathbbm{1}_{E_\varepsilon}(s)\Big)^2\Big]\\
&\leq C\varepsilon\overline{E}\Big[\sup_{s\in[0,l]}|\overline{\eta}^1(s)|^2\cdot
\sup_{s\in[0,l]}|\overline{v}(s)-\overline{u}^*(s)|^2\cdot\int_0^t|\overline{\gamma}^1_t(s)|^2\mathbbm{1}_{E_\varepsilon}(s)ds\Big]\\
&\leq C\varepsilon\Big\{
\overline{E}[(\int_0^t|\overline{\gamma}^1_t(s)|^2\mathbbm{1}_{E_\varepsilon}(s)ds)^2]\Big\}^\frac{1}{2}
\Big\{\overline{E}[\sup_{s\in[0,l]}|\overline{\eta}^1(s)|^8]\Big\}^\frac{1}{4}
\Big\{\overline{E}[\sup_{s\in[0,l]}|\overline{v}(s)-\overline{u}^*(s)|^8]\Big\}^\frac{1}{4}\\
&\leq C\varepsilon\Big\{
\overline{E}[(\int_0^t|\overline{\gamma}^1_t(s)|^2\mathbbm{1}_{E_\varepsilon}(s)ds)^2]\Big\}^\frac{1}{2}.
\end{aligned}
\end{equation}
Combining    (\ref{equ 3.3-11}), (\ref{equ 3.3-12}), (\ref{equ 3.3-13}), we obtain
\begin{equation}\label{equ 3.3-14}
\begin{aligned}
|\Theta_1(t,\varepsilon)|^4\leq C\overline{E}[\int_0^t|\widehat{E}[\overline{\widehat{\sigma}}_\nu(s)\widehat{X}^{1,\varepsilon}(s)]|^4ds]
+C\varepsilon^2\overline{E}[(\int_0^t|\overline{\gamma}^1_t(s)|^2\mathbbm{1}_{E_\varepsilon}(s)ds)^2].
\end{aligned}
\end{equation}
Hence, for $0\leq r\leq l$,
\begin{equation}\label{equ 3.3-15}
\begin{aligned}
\int_0^rE[|\Theta_1(t,\varepsilon)|^4]dt
&\leq C\int_0^r
\overline{E}[\int_0^t|\widehat{E}[\overline{\widehat{\sigma}}_\nu(s)\widehat{X}^{1,\varepsilon}(s)]|^4ds]dt
+C\varepsilon^2E\overline{E}[\int_0^T(\int_0^t|\overline{\gamma}^1_t(s)|^2\mathbbm{1}_{E_\varepsilon}(s)ds)^2dt]\\
&\leq C \int_0^r\int_0^t
\overline{E}[|\widehat{E}[\overline{\widehat{\sigma}}_\nu(s)\widehat{X}^{1,\varepsilon}(s)]|^4]ds dt+\varepsilon^2\varrho_1(\varepsilon),
\end{aligned}
\end{equation}
where  $\varrho_1(\varepsilon):=C E\overline{E}[\int_0^T(\int_0^t|\gamma^1_t(s)|^2\mathbbm{1}_{E_\varepsilon}(s)ds)^2dt].$
Dominated Convergence Theorem allows to show $\varrho_1(\varepsilon)\rightarrow0$, as $\varepsilon\rightarrow0$.

Let us now focus on $\Theta_2(t,\varepsilon)$.
From the assumptions on $\Psi_1$, $\Psi_2$ and (\ref{equ 3.3-6}) , it follows
$$
\begin{aligned}
|\Theta_2(t,\varepsilon)|
&\leq C\overline{E}\Big[
|\Psi_2(\overline{\omega},t)\overline{\rho}^1(t)\int_0^t\overline{\eta}^1(s)\widehat{E}[\overline{\widehat{b}}_\nu(s)\widehat{X}^{1,\varepsilon}(s)]ds|\\
&\quad+|\Psi_2(\overline{\omega},t)\overline{\rho}^1(t)\int_0^t\overline{\eta}^1(s)\overline{\sigma}_x(s)
\widehat{E}[\overline{\widehat{\sigma}}_\nu(s)\widehat{X}^{1,\varepsilon}(s)]ds|\\
&\quad+|\Psi_2(\overline{\omega},t)\overline{\rho}^1(t)\int_0^t\overline{\eta}^1(s)
\Delta \overline{b}(s)\mathbbm{1}_{E_\varepsilon}(s)ds|\\
&\quad+|\Psi_2(\overline{\omega},t)\overline{\rho}^1(t)\int_0^t\overline{\eta}^1(s)\sigma_x(s)
\Delta \overline{\sigma}(s)\mathbbm{1}_{E_\varepsilon}(s)ds|\Big]\\
&\leq C\Big\{
\overline{E}[\int_0^t|\widehat{E}[\overline{\widehat{b}}_\nu(s)\widehat{X}^{1,\varepsilon}(s)]|^2ds]\Big\}^\frac{1}{2}
+C\Big\{
\overline{E}[\int_0^t|\widehat{E}[\overline{\widehat{\sigma}}_\nu(s)\widehat{X}^{1,\varepsilon}(s)]|^2ds]\Big\}^\frac{1}{2}+C\varepsilon,
\end{aligned}
$$
which implies
$$
|\Theta_2(t,\varepsilon)|^4\leq C\varepsilon^4
+C\int_0^t\overline{E}[|\widehat{E}[\overline{\widehat{b}}_\nu(s)\widehat{X}^{1,\varepsilon}(s)]|^4
+|\widehat{E}[\overline{\widehat{\sigma}}_\nu(s)\widehat{X}^{1,\varepsilon}(s)]|^4]ds.
$$
Hence,
\begin{equation}\label{equ 3.3-16}
\begin{aligned}
\int_0^rE[|\Theta_2(t,\varepsilon)|^4]dt
\leq C \varepsilon^4 +C\int_0^r\int_0^t
\overline{E}[|\widehat{E}[\overline{\widehat{b}}_\nu(s)\widehat{X}^{1,\varepsilon}(s)]|^4
+|\widehat{E}[\overline{\widehat{\sigma}}_\nu(s)\widehat{X}^{1,\varepsilon}(s)]|^4]dsdt.
\end{aligned}
\end{equation}
According to (\ref{equ 3.3-15}), (\ref{equ 3.3-16}), we have, for $r\in[0,l]$,
\begin{equation}\label{equ 3.3-16-1}
\begin{aligned}
&\int_0^rE|\overline{E}[\Psi_1(\omega,\overline{\omega},t)\Psi_2(\overline{\omega},t)\overline{X}^{1,\varepsilon}(t)]|^4dt\\
&\leq \varepsilon^2\rho(\varepsilon)
+C \int_0^r\int_0^t
E\Big[|\widehat{E}[\widehat{b}_\nu(s)\widehat{X}^{1,\varepsilon}(s)]|^4
+|\widehat{E}[\widehat{\sigma}_\nu(s)\widehat{X}^{1,\varepsilon}(s)]|^4
\Big]dsdt.
\end{aligned}
\end{equation}
Taking $(\overline{\Omega},\overline{\mathcal{F}},\overline{P})=(\widehat{\Omega},\widehat{\mathcal{F}},\widehat{P})$,
and set $\Psi_1(\omega,\widehat{\omega},t)=\widehat{b}_\nu(t), \widehat{\sigma}_\nu(t),\ \Psi_2(\widehat{\omega},t)=1,$  respectively, the
Gronwall inequality allows to show
\begin{equation}\label{equ 3.3-17}
\begin{aligned}
\int_0^lE\Big[
|\widehat{E}[\widehat{b}_\nu(s)\widehat{X}^{1,\varepsilon}(s)]|^4
+|\widehat{E}[\widehat{\sigma}_\nu(s)\widehat{X}^{1,\varepsilon}(s)]|^4
\Big]ds\leq \varepsilon^2\rho(\varepsilon).
\end{aligned}
\end{equation}

\textbf{Case 3} If $l\leq t\leq T$, we have known from (\ref{equ 3.3-7}) the
specific expression of $X^{1,\varepsilon}(t-l)$, i.e., for $t\in[l,T]$,
\begin{equation}\label{equ 3.3-18}
\begin{aligned}
X^{1,\varepsilon}(t-l)&=\rho^1(t-l)\int_l^t\eta^1(s-l)(\widehat{E}[\widehat{\sigma}_\nu(s-l)\widehat{X}^{1,\varepsilon}(s-l)]
+\Delta\sigma(s-l)\mathbbm{1}_{E_\varepsilon}(s-l))dW(s)\\
&\quad+\rho^1(t-l)\int_l^t\eta^1(s-l)(\widehat{E}[\widehat{b}_\nu(s-l)\widehat{X}^{1,\varepsilon}(s-l)]
+\Delta b(s-l)\mathbbm{1}_{E_\varepsilon}(s-l))ds\\
&\quad-\rho^1(t-l)\int_l^t\eta^1(s-l)\sigma_x(s-l)(\widehat{E}[\widehat{\sigma}_\nu(s-l)\widehat{X}^{1,\varepsilon}(s-l)]
+\Delta \sigma(s-l)\mathbbm{1}_{E_\varepsilon}(s-l))ds.
\end{aligned}
\end{equation}
In this case, $X^{1,\varepsilon}$ is of the form:
\begin{equation}\label{equ 3.3-19}
\begin{aligned}
dX^{1,\varepsilon}(t)&=\Big\{ b_x(t)X^{1,\varepsilon}(t)+b_{x'}(t)X^{1,\varepsilon}(t-l)
          +\widehat{E}[\widehat{b}_\nu(t)\widehat{X}^{1,\varepsilon}(t)]+\Delta b(t)\mathbbm{1}_{E_\varepsilon}(t)\Big\}dt\\
        &\quad+\Big\{ \sigma_x(t)X^{1,\varepsilon}(t)+\sigma_{x'}(t)X^{1,\varepsilon}(t-l)
        +\widehat{E}[\widehat{\sigma}_\nu(t)\widehat{X}^{1,\varepsilon}(t)]\\
        &\qquad\qquad\qquad\qquad\qquad\qquad\qquad\qquad\qquad\       +\Delta \sigma(t)\mathbbm{1}_{E_\varepsilon}(t)\Big\}dW(t),\ t\in(l,T],
\end{aligned}
\end{equation}
where $X^{1,\varepsilon}(t-l)$ is given in (\ref{equ 3.3-18}).

We define
\begin{equation}\label{equ 3.3-20}
\left\{
\begin{aligned}
d\eta^2(t)&=-(b_x(t)-|\sigma_x(t)|^2)dt-\sigma_x(t)dW(t),\ t\in[l,T],\\
\eta^2(l)&=1,
\end{aligned}
\right.
\end{equation}
and $\rho^2(t)=(\eta^2(t))^{-1},\ t\in[l,T]$.\\
Obviously, for $\beta\geq1$,
\begin{equation}\label{equ 3.3-20-1}
E\Big[\sup_{t\in[l,T]}(|\rho^2(t)|^\beta+|\eta^2(t)|^\beta)\Big]\leq C_\beta.
\end{equation}
Applying It\^{o}'s formula to $\eta^2(s)X^{1,\varepsilon}(s)$ over $[l,t]$,
we have
\begin{equation}\label{equ 3.3-21}
\begin{aligned}
X^{1,\varepsilon}(t)=J_1(t,\varepsilon)+J_2(t,\varepsilon)+J_3(t,\varepsilon)+J_4(t,\varepsilon),
\end{aligned}
\end{equation}
where
$$
\begin{aligned}
J_1(t,\varepsilon)&=\rho^2(t)\int_l^t\eta^2(s)\sigma_{x'}(s)X^{1,\varepsilon}(s-l)ds,\\
J_2(t,\varepsilon)&=\rho^2(t)\int_l^t\eta^2(s)b_{x'}(s)X^{1,\varepsilon}(s-l)dW(s),  \quad J_3(t,\varepsilon)=\rho^2(t)X^{1,\varepsilon}(l),\\
J_4(t,\varepsilon)&=\rho^2(t)\int_l^t(\eta^2(s)E[\widehat{\sigma}_\nu(s)\widehat{X}^{1,\varepsilon}(s)]
+\eta^2(s)\Delta\sigma(s)\mathbbm{1}_{E_\varepsilon}(s))dW(s)\\
&\quad+\rho^2(t)\int_l^t(\eta^2(s)E[\widehat{b}_\nu(s)\widehat{X}^{1,\varepsilon}(s)]
+\eta^2(s)\Delta b(s)\mathbbm{1}_{E_\varepsilon}(s))ds\\
&\quad-\rho^2(t)\int_l^t(\eta^2(s)\sigma_x(s)E[\widehat{\sigma}_\nu(s)\widehat{X}^{1,\varepsilon}(s)]
+\eta^2(s)\sigma_x(s) \Delta\sigma(s)\mathbbm{1}_{E_\varepsilon}(s))ds.\\
\end{aligned}
$$

In what follows, we mainly analyse $J_1(t,\varepsilon)$, since $J_2(t,\varepsilon), J_3(t,\varepsilon)$ can
be estimated with the similar argument. As for $J_4(t,\varepsilon)$, it can be calculated following the method used in Case 2.

To begin with, insert (\ref{equ 3.3-18}) into $J_1(t,\varepsilon)$  we can get
\begin{equation}\label{equ 3.3-23}
\begin{aligned}
J_1(t,\varepsilon)=J_{1,1}(t,\varepsilon)+J_{1,2}(t,\varepsilon)+J_{1,3}(t,\varepsilon),
\end{aligned}
\end{equation}
where
$$
\begin{aligned}
J_{1,1}(t,\varepsilon)&=\rho^2(t)\int_l^t\int_0^{s-l}\eta^2(s)\sigma_{x'}(s)\rho^1(s-l)
\eta^1(r)(\widehat{E}[\widehat{\sigma}_\nu(r)\widehat{X}^{1,\varepsilon}(r)]
+\Delta\sigma(r)\mathbbm{1}_{E_\varepsilon}(r))dW(r)ds,\\
J_{1,2}(t,\varepsilon)&=\rho^2(t)\int_l^t\int_0^{s-l}\eta^2(s)\sigma_{x'}(s)\rho^1(s-l)
\eta^1(r)(\widehat{E}[\widehat{b}_\nu(r)\widehat{X}^{1,\varepsilon}(r)]
+\Delta b(r)\mathbbm{1}_{E_\varepsilon}(r))drds,\\
J_{1,3}(t,\varepsilon)&=-\rho^2(t)\int_l^t\int_0^{s-l}\eta^2(s)\sigma_{x'}(s)\rho^1(s-l)
\eta^1(r)\sigma_x(r)(\widehat{E}[\widehat{\sigma}_\nu(r)\widehat{X}^{1,\varepsilon}(r)]
+\Delta\sigma(r)\mathbbm{1}_{E_\varepsilon}(r))drds.
\end{aligned}
$$
Observe $J_{1,1}(t,\varepsilon)$, exchanging the order of integration, it yields
$$
\begin{aligned}
J_{1,1}(t,\varepsilon)&=\rho^2(t)\int_l^t\eta^1(r-l)
(\widehat{E}[\widehat{\sigma}_\nu(r-l)\widehat{X}^{1,\varepsilon}(r-l)]
+\Delta\sigma(r-l)\mathbbm{1}_{E_\varepsilon}(r-l))\cdot\\
&\qquad\qquad\qquad\qquad\qquad\qquad\qquad\qquad\qquad\qquad
\int_r^t\eta^2(s)\sigma_{x'}(s)\rho^1(s-l)dsdW(r).
\end{aligned}
$$
For $l\leq u\leq T$,  set $\varphi(u)=\int^u_l\eta^2(s)\sigma_{x'}(s)\rho^1(s-l)ds$. Clearly,
for $\beta\geq1$,
\begin{equation}\label{equ 3.3-23-1}
E[\sup_{l\leq u\leq T}|\varphi(u)|^\beta]\leq C_\beta,
\end{equation}
where $C_\beta$ is a constant depending on $\beta$.

In addition, utilizing the similar argument in Case 2, we know that for each $t\in[l,T]$,
there exists a unique $\overline{\gamma}^2_t(\cdot)\in \mathcal{H}^2_{\overline{\mathbb{F}}}(l,T)$ with the
property: for any $\beta\geq1$,
\begin{equation}\label{equ 3.3-21-1}
\overline{E}[(\int_l^T|\overline{\gamma}^2_t(s)|^2ds)^\frac{\beta}{2}]\leq C_\beta,
\end{equation}
such that
\begin{equation}\label{equ 3.3-22}
\Psi_1(\omega,\overline{\omega},t)\Psi_2(\overline{\omega},t)\overline{\rho}^2(t)
=\overline{E}[\Psi_1(\omega,\overline{\omega},t)\Psi_2(\overline{\omega},t)\overline{\rho}^2(t)]
+\int_l^t\overline{\gamma}^2_t(s)d\overline{W}_s,\ t\in[l,T].
\end{equation}
Consequence, from (\ref{equ 3.3-6}), (\ref{equ 3.3-23-1}), (\ref{equ 3.3-21-1}), (\ref{equ 3.3-22}),
one can check that for $l\leq u\leq T$,
\begin{equation}\label{equ 3.3-24}
\begin{aligned}
&\int_l^uE|\overline{E}[\Psi_1(\omega,\overline{\omega},t)\Psi_2(\overline{\omega},t)\overline{J}_{1,1}(t,\varepsilon)]|^4dt\\
&=\int_l^uE\big|\overline{E}[\Psi_1(\omega,\overline{\omega},t)\Psi_2(\overline{\omega},t)
\overline{\rho}^2(t)\int_l^t\overline{\eta}^1(r-l)
(\widehat{E}[\overline{\widehat{\sigma}}_\nu(r-l)\widehat{X}^{1,\varepsilon}(r-l)]
+\Delta\overline{\sigma}(r-l)\mathbbm{1}_{E_\varepsilon}(r-l))\cdot\\
&\qquad\qquad\qquad\qquad\qquad\qquad\qquad\qquad\qquad\qquad\qquad\qquad\qquad\qquad\qquad
(\overline{\varphi}(t)-\overline{\varphi}(r))d\overline{W}(r)]\big|^4dt\\
&=\int_l^uE|\overline{E}[\int_l^t\overline{\gamma}^2_t(r)\overline{\eta}^1(r-l)
(\widehat{E}[\overline{\widehat{\sigma}}_\nu(r-l)\widehat{X}^{1,\varepsilon}(r-l)]
+\Delta\overline{\sigma}(r-l)\mathbbm{1}_{E_\varepsilon}(r-l))\cdot(\overline{\varphi}(t)-\overline{\varphi}(r))dr]|^4dt\\
&\leq C \int_l^u\int_l^tE[|\widehat{E}[\widehat{\sigma}_\nu(r-l)\widehat{X}^{1,\varepsilon}(r-l)]|^4]drdt
+C \varepsilon^2E\overline{E}[\int_l^T(\int_l^t|\overline{\gamma}^2_t(r)|^2\mathbbm{1}_{E_\varepsilon}(r-l)dr)^2dt]\\
&\leq \varepsilon^2\rho(\varepsilon)+\varepsilon^2\varrho_2(\varepsilon),
\end{aligned}
\end{equation}
where $\varrho_2(\varepsilon):=E\overline{E}[\int_l^T(\int_l^t|\overline{\gamma}^2_t(r)|^2\mathbbm{1}_{E_\varepsilon}(r-l)dr)^2dt]\rightarrow0,$ as $\varepsilon\rightarrow0$.\\
$J_{1,2}$ and $J_{1,3}$ can be estimated with the similar argument as above. Consequence,
\begin{equation}\label{equ 3.3-25}
\begin{aligned}
&\int_l^uE|\overline{E}[\Psi_1(\omega,\overline{\omega},t)\Psi_2(\overline{\omega},t)\overline{J}_{1,2}(t,\varepsilon)]|^4dt
+\int_l^uE|\overline{E}[\Psi_1(\omega,\overline{\omega},t)\Psi_2(\overline{\omega},t)\overline{J}_{1,3}(t,\varepsilon)]|^4dt
\leq\varepsilon^2\varrho_3(\varepsilon),
\end{aligned}
\end{equation}
where  $\varrho_3(\varepsilon)\rightarrow0$ as $\varepsilon\rightarrow0$.\\
Thanks to (\ref{equ 3.3-24}) and (\ref{equ 3.3-25}),  we obtain, for $l\leq u\leq T$,
\begin{equation}\label{equ 3.3-26}
\begin{aligned}
&\int_l^uE|\overline{E}[\Psi_1(\omega,\overline{\omega},t)\Psi_2(\overline{\omega},t)\overline{J}_{1}(t,\varepsilon)]|^4dt
\leq\varepsilon^2\rho(\varepsilon).
\end{aligned}
\end{equation}
Similarly, one also has
\begin{equation}\label{equ 3.3-27}
\begin{aligned}
&\int_l^uE|\overline{E}[\Psi_1(\omega,\overline{\omega},t)\Psi_2(\overline{\omega},t)\overline{J}_{2}(t,\varepsilon)]|^4dt
+\int_l^uE|\overline{E}[\Psi_1(\omega,\overline{\omega},t)\Psi_2(\overline{\omega},t)\overline{J}_{3}(t,\varepsilon)]|^4dt
\leq\varepsilon^2\rho(\varepsilon).
\end{aligned}
\end{equation}

Last but not least, combining (\ref{equ 3.3-26}) with (\ref{equ 3.3-27}), following the analyses in Case 2 we get,  for  $l\leq u\leq T$,
\begin{equation}\label{equ 3.3-28}
\begin{aligned}
&\int_l^uE|\overline{E}[\Psi_1(\omega,\overline{\omega},t)\Psi_2(\overline{\omega},t)\overline{X}^{1,\varepsilon}(t,\varepsilon)]|^4dt\\
&\leq\varepsilon^2\rho(\varepsilon)+\int_l^u E|\overline{E}[\Psi_1(\omega,\overline{\omega},t)\Psi_2(\overline{\omega},t)\overline{J}_4(t,\varepsilon)]|^4dt\\
&\leq \varepsilon^2\rho(\varepsilon)+\int_l^u\int_l^t
E\Big[|\widehat{E}[\widehat{b}_\nu(s)\widehat{X}^{1,\varepsilon}(s)]|^4
+|\widehat{E}[\widehat{\sigma}_\nu(s)\widehat{X}^{1,\varepsilon}(s)]|^4
\Big]dsdt.
\end{aligned}
\end{equation}
Let $(\overline{\Omega},\overline{\mathcal{F}},\overline{P})=(\widehat{\Omega},\widehat{\mathcal{F}},\widehat{P})$,
and $\Psi_1(\omega,\widehat{\omega},t)=\widehat{b}_\nu(t)$ (resp., $\Psi_1(\omega,\widehat{\omega},t)=\widehat{\sigma}_\nu(t)$), and
$\Psi_2(\widehat{\omega},t)=1$ (resp., $\Psi_2(\widehat{\omega},t)=1$),  the the Gronwall inequality implies the desired result.
The proof is completed.
\end{proof}

\begin{remark}\label{re 3.2}
In particular, in Lemma 3.1  we set $\Psi_1(\omega,\overline{\omega},t)=\overline{b}_\nu(t),\ \Psi_2(\overline{\omega},t)=1$ and
$\Psi_1(\omega,\overline{\omega},t)=\overline{\sigma}_\nu(t),\ \Psi_2(\overline{\omega},t)=1$, separately, then
\begin{equation}\label{equ 3.3-2}
E\Big[\int_{0}^T(|\overline{E}[\overline{b}_\nu(t) \overline{X}^{1,\varepsilon}(t)]|^4
+|\overline{E}[\overline{\sigma}_\nu(t) \overline{X}^{1,\varepsilon}(t)]|^4) dt\Big]\leq \varepsilon^2\rho(\varepsilon),
\end{equation}
which is just the estimate (4.14) given by Buckdahn, Li and Ma \cite{BLM1}.
\end{remark}
With the help of Remark \ref{re 3.2}, it is easy to check
\begin{proposition}\label{pro 3.3}
Let the assumptions (H3.1)-(H3.3)  be in force, then for $\beta\geq1$,
\begin{equation}\label{equ 3.4}
\begin{aligned}
&\mathrm{i)}\ E[\sup_{t\in[-l,T]}|X^\varepsilon(t)-X^*(t)|^{2\beta}]\leq C\varepsilon^\beta, \quad
\mathrm{ii)}\ E[\sup_{t\in[-l,T]}|X^{1,\varepsilon}(t)|^{2\beta}]\leq C\varepsilon^\beta, \\
&\mathrm{iii)}\ E[\sup_{t\in[-l,T]}|X^{2,\varepsilon}(t)|^{2\beta}]\leq C\varepsilon^{2\beta},\quad
\mathrm{iv)}\ E[\sup_{t\in[-l,T]}|X^\varepsilon(t)-X^*(t)-X^{1,\varepsilon}(t)|^{2\beta}]\leq C\varepsilon^{2\beta},\\
&\mathrm{v)}\ E[\sup_{t\in[-l,T]}|X^\varepsilon(t)-X^*(t)-X^{1,\varepsilon}(t)-X^{2,\varepsilon}(t)|^{2}]\leq \varepsilon^{2}\rho(\varepsilon).
\end{aligned}
\end{equation}
\end{proposition}
\begin{remark}\label{re 3.1}
We can rewrite Proposition \ref{pro 3.3}-v) as
\begin{equation}\label{equ 3.5}
X^\varepsilon(t)=X^*(t)+X^{1,\varepsilon}(t)+X^{2,\varepsilon}(t)+o(\varepsilon),
\end{equation}
where the convergence is in $L^2(\mathcal{F},C[-l,T])$ sense. In fact, we have a slightly strong
result, i.e., we  can find a $\alpha_0>2$, such that (\ref{equ 3.5}) holds true under
$L^\alpha(\mathcal{F};C[-l,T])$, $\alpha\in[2, \alpha_0)$. The reader can refer to \cite{BLM1} for more details.
\end{remark}

\subsection{Adjoint equations}

In this subsection two adjoint equations are introduced which are the basic-materials to apply duality method to study
our stochastic maximum principle. Compared with the classical case, see Hu \cite{Hu}, there are two notable differences. The
first one is the first-order adjoint equation is an anticipated mean-field BSDE. The second one
comes from the second-order adjoint equation. Due to the  delay term, the equation $(X^{1,\varepsilon})^2$ is
not sufficient to solve our control problem. It means that we have to bring in an additional auxiliary equation $X^{1,\varepsilon}(t)X^{2,\varepsilon}(t)$.
Hence, the second-order adjoint system  is a system of equations for matrix valued processes. Let us describe them in detail.

Denote for $\ell=x,x',y,z,$ $\theta=x,x',y,z,¡¢a_i$, $i,j=1,2$,
\begin{equation}\label{equ 3.6}
\begin{aligned}
&\Xi^*(t)=(X^*(t),X^*(t-l),Y^*(t),Z^*(t)),\  \Lambda^*(t)=(X^*(t),Y^*(t)),\
f_\ell(s)=\frac{\partial f}{\partial \ell}(s,\Xi^*(s), P_{ \Lambda^*(s)},u^*(s)),\\
&(\widehat{f}_{\mu_i},\widehat{f}_{\mu_i \theta})(s)=((\frac{\partial f}{\partial \mu})_i,
\frac{\partial}{\partial \theta}((\frac{\partial f}{\partial \mu})_i))(s,\Xi^*(s),
P_{ \Lambda^*(s)},u^*(s);\widehat{\Lambda}^*(s)),\\
&(\widehat{f}^*_{\mu_i},\widehat{f}^*_{\mu_i \theta})(s)=((\frac{\partial f}{\partial \mu})_i,
\frac{\partial}{\partial \theta}((\frac{\partial f}{\partial \mu})_i))(s,\widehat{\Xi}^*(s),
P_{ \Lambda^*(s)},\widehat{u}^*(s);\Lambda^*(s)),\\
&\widetilde{\widehat{f}}_{\mu_i \mu_j}(s)=
(\frac{\partial}{\partial \mu}((\frac{\partial f}{\partial \mu})_i))_j
(s,\Xi^*(s), P_{ \Lambda^*(s)},u^*(s);\widehat{\Lambda}^*(s),\widetilde{\Lambda}^*(s)).\\
\end{aligned}
\end{equation}

The first order
adjoint equation in our case is the following anticipated mean-field BSDE:
\begin{equation}\label{equ 3.9}
  \left\{
   \begin{aligned}
-dp(s)&=F(s)ds-q(s)dW(s), s\in[0,T],\\
p(T)&=\Phi_x(T)+\widehat{E}[\widehat{\Phi}^*_\nu(T)],\ p(t)=0,\ t\in(T,T+l],\ q(t)=0,\ t\in[T,T+l],
\end{aligned}
   \right.
\end{equation}
where $\widehat{\Phi}^*_\nu(T)=\frac{\partial\Phi}{\partial \nu}(\widehat{X}^*_T,P_{X^*_T};X^*_T)$ and
$$
 \begin{aligned}
 F(s)&=p(s)\Big(f_y(s)+\widehat{E}[\widehat{f}^*_{\mu_2}(s)]+f_z(s)\sigma_x(s)+b_x(s)\Big)
+\widehat{E}\Big[\widehat{p}(s)\Big(\widehat{f}_z(s)\widehat{\sigma}^*_\nu(s)+\widehat{b}^*_\nu(s)\Big)\Big]\\
&\quad +E^{\mathcal{F}_s}[p(s+l)(f_z(s+l)\sigma_{x'}(s+l)+b_{x'}(s+l))]\\
&\quad+q(s)\Big(f_z(s)+\sigma_x(s)\Big)+\widehat{E}[\widehat{q}(s)\widehat{\sigma}^*_\nu(s)]+E^{\mathcal{F}_s}[q(s+l)\sigma_{x'}(s+l)]\\
&\quad+f_x(s)+\widehat{E}[\widehat{f}^*_{\mu_1}(s)]+E^{\mathcal{F}_s}[f_{x'}(s+l)].  \\
\end{aligned}
$$

According to Theorem 3.1 in Guo, Xiong and Zheng \cite{GXZ}
under the assumptions (H3.1)-(H3.2) the anticipated mean-field BSDE (\ref{equ 3.9}) admits a unique solution
$(p,q)\in \mathcal{S}_{\mathbb{F}}^2(0,T+l)\times \mathcal{H}_{\mathbb{F}}^2(0,T+l)$, and
moreover, from the boundness of the derivatives of the coefficients it follows  that, for $\beta\geq2$, there
exists a constant $C_\beta$ depending on $\beta$   such that
\begin{equation}\label{equ 3.9-1}
E\Big[\sup_{t\in[0,T+l]}|p(t)|^\beta+(\int_{0}^{T+l}|q(t)|^2dt)^{\frac{\beta}{2}}\Big]\leq C_\beta.
\end{equation}

From Lemma \ref{le 3.1}, we have the following estimates.
\begin{corollary}\label{cor 3.5}
Let the assumptions (H3.1)-(H3.3) be in force, and set for $\ell=x,x',y,z,$
$$
\begin{aligned}
\widehat{M}_1(s)&:=(\widehat{f}_{\mu_1},\widehat{f}_{\mu_2})(s),\
\widehat{M}_2(s):=(\widehat{f}_{\ell\mu_1},\widehat{f}_{\ell\mu_2})(s),\
\widetilde{\widehat{M}}_3(s):=(\widetilde{\widehat{f}}_{\mu_1 \mu_1},\widetilde{\widehat{f}}_{\mu_1 \mu_2}, \widetilde{\widehat{f}}_{\mu_2 \mu_2})(s).
\end{aligned}
$$
Let $X^{1,\varepsilon}$ and $p$ be the solutions of (\ref{equ 3.2}) and  (\ref{equ 3.9}), separately, then
\begin{equation}\label{equ 3.32}
\begin{aligned}
&\mathrm{i)}\ E\Big[\int_0^T\Big(|\widehat {E}[\widehat{M}_1(s)\widehat{X}^{1,\varepsilon}(s)]|^4
+|\widehat {E}[\widehat{M}_2(s)\widehat{X}^{1,\varepsilon}(s)]|^4\Big)ds\Big]\leq\varepsilon^2\rho(\varepsilon);\\
&\mathrm{ii)}\ E\Big[\int_0^T\Big(|\widehat {E}[\widehat{M}_1(s)\widehat{p}(s)\widehat{X}^{1,\varepsilon}(s)]|^4
+|\widehat{E}[\widehat{M}_2(s)\widehat{p}(s)\widehat{X}^{1,\varepsilon}(s)]|^4\Big)ds\Big]\leq\varepsilon^2\rho(\varepsilon);\\
&\mathrm{iii)}\ E\widetilde{E}\Big[\int_0^T |\widehat{E}[\widetilde{\widehat{M}}_3(s)\widehat{p}(s)\widehat{X}^{1,\varepsilon}(s)]|^4ds\Big]\leq\varepsilon^2\rho(\varepsilon).\\
\end{aligned}
\end{equation}

\end{corollary}

\begin{remark}
 Estimates (\ref{equ 3.3-2}) and (\ref{equ 3.32}) are powerful tools to calculate
 the mean-field terms, especially those involving the derivatives of the coefficients with respect to
 the measure and the first-order variation $X^{1,\varepsilon}$, appearing in the  expansion of $X^\varepsilon$
 and $Y^\varepsilon$. As an example, the reader can refer to the proof of (\ref{equ 4.8-4}) and  the expansion of $I_{22}(s)$ in Appendix.
\end{remark}

Let us now analyse the second-order adjoint matrix-valued system. We set
$K^\varepsilon(t):=(X^{1,\varepsilon}(t))^2,$ $K_1^\varepsilon(t):= X^{1,\varepsilon}(t)X^{1,\varepsilon}(t-l) $.
First, applying It\^{o}'s formula to $K^\varepsilon(t)$ we have
\begin{equation}\label{equ 3.10}
  \left\{
 \begin{aligned}
d K^\varepsilon(t)&=\Big\{
\Big(2b_x(t)+(\sigma_x(t))^2\Big)K^\varepsilon(t)+\Big(2b_{x'}(t)+2\sigma_x(t)\sigma_{x'}(t)\Big)K_1^\varepsilon(t)
+\Big(\sigma_{x'}(t)\Big)^2K^\varepsilon(t-l)\\
&\quad+(\Delta \sigma(t))^2\mathbbm{1}_{E_\varepsilon}(t)\Big\}dt
+\Big\{
2\sigma_x(t)K^\varepsilon(t)+2\sigma_{x'}(t)K^\varepsilon_1(t)\Big\}dW_t+L_1(t)dt+L_2(t)dW_t,\\
K^\varepsilon(t)&=0, \ t\in[-l,0],
\end{aligned}
   \right.
\end{equation}
where
\begin{equation}\label{equ 3.11}
 \begin{aligned}
L_1(t)&=\mathbbm{1}_{E_\varepsilon}(t)\Big\{2 \Delta b(t)X^{1,\varepsilon}(t)+2\sigma_x(t)\Delta \sigma(t)X^{1,\varepsilon}(t)
+2\sigma_{x'}(t)\Delta \sigma(t)X^{1,\varepsilon}(t-l)\Big\}\\
&\quad+\Big\{2\widehat{E}[\widehat{b}^*_\nu(t)\widehat{X}^{1,\varepsilon}(t)]X^{1,\varepsilon}(t)
+(\widehat{E}[\widehat{\sigma}^*_\nu(t)\widehat{X}^{1,\varepsilon}(t)])^2
+2\sigma_x(t)X^{1,\varepsilon}(t)\widehat{E}[\widehat{\sigma}^*_\nu(t)\widehat{X}^{1,\varepsilon}(t)]\\
&\quad +2\sigma_{x'}(t)X^{1,\varepsilon}(t-l)\widehat{E}[\widehat{\sigma}^*_\nu(t)\widehat{X}^{1,\varepsilon}(t)]
+2\Delta \sigma(t)\widehat{E}[\widehat{\sigma}^*_\nu(t)\widehat{X}^{1,\varepsilon}(t)]\mathbbm{1}_{E_\varepsilon}(t)\Big\},\\
L_2(t)&= 2\Delta \sigma(t)X^{1,\varepsilon}(t)\mathbbm{1}_{E_\varepsilon}(t)
+\widehat{E}[\widehat{\sigma}^*_\nu(t)\widehat{X}^{1,\varepsilon}(t)]X^{1,\varepsilon}(t).
\end{aligned}
\end{equation}
From Remark \ref{re 3.2}, the continuous property of $f$ with respect to $v$,
the boundness of $b_{\nu},\ \sigma_\nu$ it yields
\begin{equation}\label{equ 3.12}
E\Big [\Big(\int_0^T |L_1(t)|+|L_2(t)|dt\Big)^2\Big ]\leq \varepsilon^2\rho(\varepsilon)
\end{equation}
(see Appendix for more details).

From the structure of the equation of $K^\varepsilon(t)$, we know that $K^\varepsilon(t)$ depends not only
on the quadratic terms $K^\varepsilon(t), K^\varepsilon(t-l)$, but also on the mixed term  $K_1^\varepsilon(t)$.
That means that it is impossible that $K^\varepsilon(t)$ solves  an equation in a closed form of $K^\varepsilon(t), K^\varepsilon(t-l)$.
So we are ready to consider an auxiliary equation $K_1^\varepsilon(t)$.
Let us first introduce the equation of $X^{1,\varepsilon}(t-l)$.
Notice $X^{1,\varepsilon}$ is not null only after positive times, then
\begin{equation}\label{equ 3.13}
  \left\{
 \begin{aligned}
d X^{1,\varepsilon}(t-l)&=\Big\{
b_x(t-l)X^{1,\varepsilon}(t-l)+b_{x'}(t-l)X^{1,\varepsilon}(t-2l)+\widehat{E}[\widehat{b}^*_\nu(t-l)\widehat{X}^{1,\varepsilon}(t-l)]\Big\}dt\\
&\quad+\Big\{
\sigma_x(t-l)X^{1,\varepsilon}(t-l)+\sigma_{x'}(t-l)X^{1,\varepsilon}(t-2l)+\widehat{E}[\widehat{\sigma}^*_\nu(t-l)\widehat{X}^{1,\varepsilon}(t-l)]\Big\}dW_t\\
&\quad +\Big\{\Delta b(t-l)\mathbbm{1}_{E_\varepsilon}(t-l)\Big\}dt
+\Big\{\Delta \sigma(t-l)\mathbbm{1}_{E_\varepsilon}(t-l)\Big\}dW_t,\ t\in[l,T],\\
X^{1,\varepsilon}(t-l)&=0,\ t\in[-l,l].
\end{aligned}
\right.
\end{equation}
In our case,  due to $t<T\leq2l$, hence $X^{1,\varepsilon}(t-2l)\equiv0,\ t\in[0,T].$ Consequently,

\begin{equation}\label{equ 3.14}
\left\{
 \begin{aligned}
dK^\varepsilon_1(t)&=\Big\{
\Big(b_x(t)+b_x(t-l)+\sigma_x(t)\sigma_x(t-l)\Big)K^\varepsilon_1(t)+
\Big(b_{x'}(t)+\sigma_{x'}(t)\sigma_x(t-l)\Big)K^\varepsilon(t-l)
\Big\}dt\\
&\quad+\Big\{\Big(\sigma_x(t)+\sigma_x(t-l)\Big)K^\varepsilon_1(t)+\sigma_{x'}(t)K^\varepsilon(t-l)\Big\}dW_t
+L_3(t)dt+L_4(t)dW_t,\\
&\quad+\{\Delta \sigma(t)\Delta \sigma(t-l)\mathbbm{1}_{E_\varepsilon}(t-l)\mathbbm{1}_{E_\varepsilon}(t)\}dt,\ t\in[l,T],\\
K^\varepsilon_1(t)&=0,\ t\in[-l, l],
\end{aligned}
\right.
\end{equation}
where
$$
\begin{aligned}
L_3(t)&=X^{1,\varepsilon}(t)\Delta b(t-l)\mathbbm{1}_{E_\varepsilon}(t-l)+\Delta b(t)X^{1,\varepsilon}(t-l)\mathbbm{1}_{E_\varepsilon}(t)
+\sigma_x(t)\Delta \sigma(t-l)X^{1,\varepsilon}(t)\mathbbm{1}_{E_\varepsilon}(t-l)\\
&\quad+\sigma_{x'}(t)\Delta\sigma(t-l)X^{1,\varepsilon}(t-l)\mathbbm{1}_{E_\varepsilon}(t-l)
+\sigma_x(t-l)\Delta \sigma(t)X^{1,\varepsilon}(t-l)\mathbbm{1}_{E_\varepsilon}(t)\\
&\quad+\Delta \sigma(t-l)\mathbbm{1}_{E_\varepsilon}(t-l)\widehat{E}[\widehat{\sigma}^*_\nu(t)\widehat{X}^{1,\varepsilon}(t)]
+\Delta \sigma(t)\mathbbm{1}_{E_\varepsilon}(t)\widehat{E}[\widehat{\sigma}^*_\nu(t-l)\widehat{X}^{1,\varepsilon}(t-l)]\\
\end{aligned}
$$
$$
\begin{aligned}
&\quad+X^{1,\varepsilon}(t-l)\widehat{E}[\widehat{b}^*_\nu(t)\widehat{X}^{1,\varepsilon}(t)]
+\widehat{E}[\widehat{\sigma}^*_\nu(t-l)\widehat{X}^{1,\varepsilon}(t-l)](\sigma_x(t)X^{1,\varepsilon}(t)+\sigma_{x'}(t)X^{1,\varepsilon}(t-l))\\
&\quad+\sigma_x(t-l)X^{1,\varepsilon}(t-l)\widehat{E}[\widehat{\sigma}^*_\nu(t)\widehat{X}^{1,\varepsilon}(t)]
+X^{1,\varepsilon}(t)\widehat{E}[\widehat{b}^*_\nu(t-l)\widehat{X}^{1,\varepsilon}(t-l)],\\
L_4(t)&=\Delta \sigma(t-l)X^{1,\varepsilon}(t)\mathbbm{1}_{E_\varepsilon}(t-l)+\Delta \sigma(t)X^{1,\varepsilon}(t-l))\mathbbm{1}_{E_\varepsilon}(t)\\
&\quad+X^{1,\varepsilon}(t)\widehat{E}[\widehat{\sigma}^*_\nu(t-l)\widehat{X}^{1,\varepsilon}(t-l)]
+X^{1,\varepsilon}(t-l)\widehat{E}[\widehat{\sigma}^*_\nu(t)\widehat{X}^{1,\varepsilon}(t)].
\end{aligned}
$$
Note that  $\mathbbm{1}_{E_\varepsilon}(t-l)\mathbbm{1}_{E_\varepsilon}(t)\equiv0$ for enough small $\varepsilon$, for example $\varepsilon<\frac{l}{3}$. Moreover, following the proof of (\ref{equ 3.12}) (see Appendix),
from Remark \ref{re 3.2} again, it yields
\begin{equation}\label{equ 3.15}
E\Big [\Big(\int_0^T |L_3(t)|+|L_4(t)|dt\Big)^2\Big ]\leq \varepsilon^2\rho(\varepsilon).
\end{equation}

The coupled characteristic of the equations (\ref{equ 3.10}) and (\ref{equ 3.14}) inspires us to consider the
following coupled second-order adjoint equation, which is different to the
classical case:
\begin{equation}\label{equ 3.16}
  \left\{
   \begin{aligned}
-dP(s)&=G(s)ds-Q(s)dW(s)\ s\in[0,T],\\
P(T)&=\Phi_{xx}(T)+\widehat{E}[\widehat{\Phi}^*_{\nu a}(T)], P(t)=0,\ t\in(T,T+l],\ Q(t)=0,\ t\in[T, T+l],\\
-dP_1(s)&=G_1(s)ds-Q_1(s)dW(s)\ s\in[0,T],\\
P_1(T)&=0,\ t\in[T,T+l],\ Q_1(t)=0,\ t\in[T, T+l],
\end{aligned}
   \right.
\end{equation}
where
\begin{equation}\label{equ 3.17}
\begin{aligned}
G(t)&=P(t)\Big(f_y(t)+\widehat{E}[\widehat{f}^*_{\nu_2}(t)]+2f_z(t)\sigma_x(t)+2b_x(t)+(\sigma_x(t))^2\Big)
+E^{\mathcal{F}_t}[P(t+l)(\sigma_{x'}(t+l))^2]\\
&\quad+Q(t)(f_z(t)+2\sigma_x(t))
+E^{\mathcal{F}_t}[2P_1(t+l)(\sigma_{x'}(t+l)+b_{x'}(t+l)+\sigma_{x'}(t+l)\sigma_x(t))]\\
&\quad+E^{\mathcal{F}_t}[2Q_1(t+l)\sigma_{x'}(t+l)]
+p(t)(b_{xx}(t)+f_z(t)\sigma_{xx}(t))+\widehat{E}[\widehat{p}(t)(\widehat{b}^*_{\nu a}(t)+\widehat{f}_z(t)\widehat{\sigma}^*_{\nu a}(t))]\\
&\quad+E^{\mathcal{F}_t}[p(t+l)(b_{x'x'}(t+l)+f_z(t+l)\sigma_{x'x'}(t+l))]
+q(t)\sigma_{xx}(t)+\widehat{E}[\widehat{q}(t)\widehat{\sigma}^*_{\nu a}(t)]\\
&\quad+E^{\mathcal{F}_t}[q(t+l)\sigma_{x'x'}(t+l)]+UD^2fU^\intercal
+E^{\mathcal{F}_t}[f_{x'x'}(t+l)]
+\widehat{E}[\widehat{f}^*_{\mu_1 a_1}(t)]+(p(t))^2\widehat{E}[\widehat{f}^*_{\mu_2 a_2}(t)],\\
G_1(t)&=P_1(t)\Big(b_x(t)+b_x(t-l)+\sigma_x(t)\sigma_x(t-l)+f_y(t)+f_z(t)(\sigma_{x'}(t)+\sigma_x(t-l))+\widehat{E}[\widehat{f}^*_{\mu_2}(t)]\Big)\\
&\quad+Q_1(t)\Big(\sigma_x(t)+\sigma_x(t-l)+f_z(t)\Big)
+P(t)\Big( b_{x'}(t)+\sigma_x(t)\sigma_{x'}(t)+f_z(t)\sigma_{x'}(t)\Big)\\
&\quad+Q(t) \sigma_{x'}(t)
+p(t)\Big(b_{xx'}(t)+f_z(t)\sigma_{xx'}(t)+f_{y x'}(t)+\sigma_x(t)f_{zx'}(t)\Big)+q(t)(\sigma_{xx'}(t)+f_{zx'}(t))\\
&\quad+f_{xx'}(t),
\end{aligned}
\end{equation}
and $U=(1,p(t),p(t)\sigma_x(t)+q(t)),\
D^2f=\begin{pmatrix}
f_{xx}(t)&f_{xy}(t)&f_{xz}(t)\\
f_{yx}(t)&f_{yy}(t)&f_{yz}(t)\\
f_{zx}(t)&f_{zy}(t)&f_{zz}(t)
\end{pmatrix}.$

Obviously, according to Guo, Xiong, Zheng  \cite{GXZ}, under the assumptions (H3.1)-(H3.3)
the  anticipated linear BSDE (\ref{equ 3.16})-(\ref{equ 3.17}) possesses a unique solution
$\Big((P,Q),(P_1,Q_1)\Big)$ satisfying, for $\beta\geq2$,
\begin{equation}\label{equ 3.17-1}
E\Big[\sup_{t\in[0,T+l]}(|P(t)|^\beta+|P_1(t)|^\beta)+(\int_0^{T+l}|Q(t)|^2+|Q_1(t)|^2dt)^\frac{\beta}{2}\Big]\leq C_\beta.
\end{equation}

\begin{remark}\label{re 3.5}
If $T\leq l$, the control problem becomes the non delay case. In fact, the term $X^{1,\varepsilon}(t-l)$ in the first-order
variational equation will vanish. In this setting, our stochastic maximum principle reduces to the global stochastic maximum principle
but driven by a mean-field forward-backward control system.
\end{remark}

\begin{remark}\label{re 3.6}
Although in above we present the case of one pointwise delay, our approach is also suitable for the pointwise delay case,
i.e., $N l< T\leq(N+1)l,\ N\in\mathbb{N}, N>1$. Indeed, in this situation the first-order variational equation
(\ref{equ 3.2}) can be extended to $(X^{1,\varepsilon}(t), X^{1,\varepsilon}(t-l),\cdot\cdot\cdot,
X^{1,\varepsilon}(t-Nl))$ naturally,
and accordingly, the equation $K^\varepsilon(t)$ will depend on $(X^{1,\varepsilon}(t))^2, (X^{1,\varepsilon}(t-l))^2,\cdot\cdot\cdot,
(X^{1,\varepsilon}(t-Nl))^2$.
Hence, some corresponding auxiliary  equations, similar to $K^1(t)$,
$K^2(t):=X^{1,\varepsilon}(t)X^{1,\varepsilon}(t-2l),\cdot\cdot\cdot, K^N(t):=X^{1,\varepsilon}(t)X^{1,\varepsilon}(t-Nl),$
will be considered, and subsequently, the second-order adjoint system changes to
a system of equations for high-dimensional matrix-valued processes, just like the statement in Remark \ref{re 4.2}.

\end{remark}

\section{{\protect \large {Expansion of cost functional $Y^\varepsilon$}}}

In order to obtain the general stochastic maximum principle, Peng \cite{Peng1} put forward to consider the
second-order Taylor expansion of the variation $Y^\varepsilon$, i.e.,
\begin{equation}\label{equ 4.0}
Y^\varepsilon(t)=Y^*(t)+p(t)(X^{1,\varepsilon}(t)+X^{2,\varepsilon}(t))+\frac{1}{2}P(t)(X^{1,\varepsilon}(t))^2+o(\varepsilon),
\end{equation}
where $p,P$ are the solutions of the first- and second-order adjoint equation, respectively.\\
Generally speaking, when the coefficient $f$ depends on $(y,z)$,  (\ref{equ 4.0}) does not hold true any more since
the term $p(t)\delta\sigma(t)\mathbbm{1}_{E_\varepsilon}(t)$ in the variation of $z$ is $O(\varepsilon)$, but not $o(\varepsilon)$.
Hence, Hu \cite{Hu} constructed an auxiliary BSDE and considered the following expansion:
\begin{equation}\label{equ 4.0-1}
Y^\varepsilon(t)=Y^*(t)+p(t)(X^{1,\varepsilon}(t)+X^{2,\varepsilon}(t))+\frac{1}{2}P(t)(X^{1,\varepsilon}(t))^2+\breve{Y}(t)+o(\varepsilon),\
t\in[0,T],
\end{equation}
where $\breve{Y}$ is the solution of an auxiliary  BSDE.

Our argument in some degree follows their scheme, but there are still potential difficulties  due to
the appearance of delay term in the general  mean-field case. But it turns out that these difficulties can
be solved by considering the auxiliary equation of mixed term $X^{1,\varepsilon}(s)X^{1,\varepsilon}(s-l)$,  by constructing a
new auxiliary mean-field BSDE and by applying the new generic estimate, see Lemma \ref{le 3.1}. Let us now state it in detail.

We first prove that there exists a process $(\breve{Y}(t))_{t\in[0,T]}$ with $\breve{Y}(T)=0$, such that
for $t\in[0,T]$, $P$-a.s.,
\begin{equation}\label{equ 4.0-2}
Y^\varepsilon(t)=Y^*(t)+p(t)(X^{1,\varepsilon}(t)+X^{2,\varepsilon}(t))
+\frac{1}{2}P(t)(X^{1,\varepsilon}(t))^2+P_1(t)X^{1,\varepsilon}(t)X^{1,\varepsilon}(t-l)
+\breve{Y}(t)+o(\varepsilon),
\end{equation}
where the convergence is in $L^2(\Omega,C[0,T])$ sense.

For this, consider the following mean-field BSDE:
\begin{equation}\label{equ 4.1}
  \left\{
   \begin{aligned}
-d\breve{Y}(s)&= f_y(t)\breve{Y}(t)+f_z(t)\breve{Z}(t)+\widehat{E}[\widehat{f}_{\mu_2}(t)\widehat{\breve{Y}}(t)]
+\Big(A_1(t)+\Delta f(t)\mathbbm{1}_{E_\varepsilon}(t)\Big)dt\\
&\quad-\breve{Z}(t)dW(t),\ t\in[0,T),\\
\breve{Y}(t)&=0, \ \breve{Z}(t)=0,\ t\in[T,T+l],
\end{aligned}
   \right.
\end{equation}
where
\begin{equation}\label{equ 4.1-1}
   \begin{aligned}
A_1(t)&:=p(t)\Delta b(t)+q(t)\Delta \sigma(t)+\frac{1}{2}P(t)( \Delta\sigma(t))^2,\\
\Delta f(t)&:=f(t,X^*(t),X^*(t-l),Y^*(t),Z^*(t)+p(t)\Delta \sigma(t), P_{(X^*(t),Y^*(t))},v(t))\\
            &\quad-f(t,X^*(t),X^*(t-l),Y^*(t),Z^*(t), P_{(X^*(t),Y^*(t))},u^*(t)).
\end{aligned}
\end{equation}

It is clear that under the assumptions (H3.1), (H3.2) the equation (\ref{equ 4.1}) has a unique solution
$(\breve{Y},\breve{Z})\in \mathcal{S}^2_{\mathbb{F}}(0,T+l)\times \mathcal{H}^2_{\mathbb{F}}(0,T+l)$.
Furthermore, from standard argument of classical BSDE, we have, for $\beta\geq2$,
\begin{equation}\label{equ 4.2}
   \begin{aligned}
E[\sup_{t\in[0,T+l]}|\breve{Y}(t)|^\beta+(\int_0^{T+l}|\breve{Z}(t)|^2dt)^\frac{\beta}{2}]\leq \varepsilon^{\frac{\beta}{2}}\rho(\varepsilon).
\end{aligned}
\end{equation}

\begin{theorem}\label{th 4.1}
Suppose the assumptions (H3.1)-(H3.3) hold true, then we have the following
second-order expansion of $Y^\varepsilon$
\begin{equation}\label{equ 4.3}
   \begin{aligned}
&E\Big[\sup_{t\in[0,T+l]}|Y^\varepsilon(t)-Y^*(t)-p(t)(X^{1,\varepsilon}(t)+X^{2,\varepsilon}(t))
-\frac{1}{2}P(t)(X^{1,\varepsilon}(t))^2\\
&\qquad\qquad\qquad\qquad\qquad\qquad\qquad -P_1(t)X^{1,\varepsilon}(t)X^{1,\varepsilon}(t-l)
-\breve{Y}(t)|^2 \Big]\leq \varepsilon^2\rho(\varepsilon).
\end{aligned}
\end{equation}
\end{theorem}
\begin{proof}
The proof is split into two steps.

\textbf{Step 1.}\ Define
\begin{equation}\label{equ 4.4}
M(t):=p(t)(X^{1,\varepsilon}(t)+X^{2,\varepsilon}(t))
+\frac{1}{2}P(t)K^\varepsilon(t)
 +P_1(t)K^\varepsilon_1(t).
 \end{equation}
Following It\^{o}'s formula, one can check
\begin{equation}\label{equ 4.5}
\begin{aligned}
dM(t)&=\Big\{ A_1(t)\mathbbm{1}_{E_\varepsilon}(t)
+A_2(t)+A_3(t)+\overline{A}_3(t)+A_4(t)\Big\}dt\\
&\quad +\Big\{ p(t)\Delta \sigma(t)\mathbbm{1}_{E_\varepsilon}(t)+B_2(t)+B_3(t)+\overline{B}_3(t)+B_4(t)\Big\}dW_t.
\end{aligned}
\end{equation}
Here $A_1,A_2,\cdot\cdot\cdot, B_4$ see appendix.\\
From Remark {\ref{re 3.2}}, we have
\begin{equation}\label{equ 4.6}
E\Big[\Big(\int_0^T|A_4(s)+B_4(s)|ds\Big)^2\Big]\leq \varepsilon^2\rho(\varepsilon).
\end{equation}
Hence, (\ref{equ 4.5}) can be written as
\begin{equation}\label{equ 4.6-1}
\begin{aligned}
dM(t)&=\Big\{ A_1(t)\mathbbm{1}_{E_\varepsilon}(t)
+A_2(t)+A_3(t)+\overline{A}_3(t)\Big\}dt\\
&\quad +\Big\{ p(t)\Delta \sigma(t)\mathbbm{1}_{E_\varepsilon}(t)+B_2(t)+B_3(t)+\overline{B}_3(t) \Big\}dW_t+o(\varepsilon).
\end{aligned}
\end{equation}

\textbf{Step 2.} Define
$$
\begin{aligned}
\Delta Y(t)&=Y^\varepsilon(t)-Y^*(t)-\breve{Y}(t)-M(t),\\
\Delta Z(t)&=Z^\varepsilon(t)-Z^*(t)-\breve{Z}(t)-( p(t)\Delta \sigma(t)\mathbbm{1}_{E_\varepsilon}(t)+B_2(t)+B_3(t)+\overline{B}_3(t)).
\end{aligned}
$$
Then
\begin{equation}\label{equ 4.7}
\begin{aligned}
\Delta Y(t)=\Delta Y(T)
&+\int_t^T\Big(f(s,X^\varepsilon(s),X^\varepsilon(t-l),Y^\varepsilon(s),Z^\varepsilon(s),P_{(X^\varepsilon(s),Y^\varepsilon(s))},u^\varepsilon(s))\\
&\quad-f(s,X^*(s),X^*(t-l),Y^*(s),Z^*(s),P_{(X^*(s),Y^*(s))},u^*(s))\\
&\quad-(f_y(s)\breve{Y}(s)+f_{z}(s)\breve{Z}(s)+\widehat{E}[\widehat{f}_{\mu_2}(s)\widehat{\breve{Y}}(s)]+\Delta f(s)\mathbbm{1}_{E_\varepsilon}(s))\\
&\quad-(A_2(t)+A_3(t)+\overline{A}_3(t))\Big)ds-\int_t^T\Delta Z(s)dW_s+o(\varepsilon).
\end{aligned}
\end{equation}

Let us now analyse
\begin{equation}\label{equ 4.8}
\begin{aligned}
I(s):&=f(s,X^\varepsilon(s),X^\varepsilon(s-l),Y^\varepsilon(s),Z^\varepsilon(s),P_{(X^\varepsilon(s),Y^\varepsilon(s))},u^\varepsilon(s))\\
&-f(s,X^*(s),X^*(s-l),Y^*(s),Z^*(s),P_{(X^*(s),Y^*(s))},u^*(s)).
\end{aligned}
\end{equation}
First, recall the definitions of $\Delta Y, \Delta Z$, (\ref{equ 4.8}) can read as
\begin{equation}\label{equ 4.8-1}
\begin{aligned}
&I(s):=I_1(s)+I_2(s)
+f(s,X^*(s),X^*(s-l),Y^*(s),Z^*(s)+p(s)\Delta(s)\mathbbm{1}_{E_\varepsilon}(s),P_{(X^*(s),Y^*(s))},u^\varepsilon(s))\\
&\qquad-f(s,X^*(s),X^*(s-l),Y^*(s),Z^*(s),P_{(X^*(s),Y^*(s))},u^*(s))\\
&\qquad=I_1(s)+I_2(s)+\mathbbm{1}_{E_\varepsilon}(s)\triangle f(s),
\end{aligned}
\end{equation}
where
\begin{equation}\label{equ 4.8-1}
\begin{aligned}
I_1(s):&=f(s,X^{\varepsilon}(s),X^\varepsilon(s-l),Y^{\varepsilon}(s),Z^{\varepsilon}(s),P_{X^{\varepsilon}(s),Y^{\varepsilon}(s)},u^\varepsilon(s))\\
       & -f(s,X^*(s)+X^{1,\varepsilon}(s)+X^{2,\varepsilon}(s),
       X^*(s-l)+X^{1,\varepsilon}(s-l)+X^{2,\varepsilon}(s-l),\\
       &\quad Y^*(s)+\breve{Y}(s)+M(s),Z^*(s)+\breve{Z}(s)+(p(s)\Delta \sigma(s)\mathbbm{1}_{E_\varepsilon}(s)+B_2(t)+B_3(s)+\overline{B}_3(s)),\\
        &\quad P_{(X^*(s)+X^{1,\varepsilon}(s)+X^{2,\varepsilon}(s),Y^*(s)+\breve{Y}(s)+M(s))},u^\varepsilon(s)),\\
I_2(s):&=f(s,X^*(s)+X^{1,\varepsilon}(s)+X^{2,\varepsilon}(s),
       X^*(s-l)+X^{1,\varepsilon}(s-l)+X^{2,\varepsilon}(s-l),\\
       &\quad Y^*(s)+\breve{Y}(s)+M(s),Z^*(s)+\breve{Z}(s)+(p(s)\Delta \sigma(s)\mathbbm{1}_{E_\varepsilon}(s)+B_2(s)+B_3(s)+\overline{B}_3(s)),\\
        &\quad P_{(X^*(s)+X^{1,\varepsilon}(s)+X^{2,\varepsilon}(s),Y^*(s)+\breve{Y}(s)+M(s))},u^\varepsilon(s))\\
        &-f(s,X^*(s),X^*(s-l),Y^*(s),Z^*(s)+p(s)\Delta\sigma(s)\mathbbm{1}_{E_\varepsilon}(s),P_{(X^*(s),Y^*(s))}, u^\varepsilon(s)).
\end{aligned}
\end{equation}
From the Lipschitz property of $f$ and the fact $W(P_\xi,P_\eta)\leq \{E|\xi-\eta|^2\}^\frac{1}{2},\ \xi,\eta\in L^2(\Omega,\mathcal{F}_T,P)$, we know
\begin{equation}\label{equ 4.8-2}
\begin{aligned}
|I_1(s)|&\leq |X^\varepsilon(s)-X^*(s)-X^{1,\varepsilon}(s)-X^{2,\varepsilon}(s)|
+|X^\varepsilon(s-l)-X^*(s-l)-X^{1,\varepsilon}(s-l)-X^{2,\varepsilon}(s-l)|\\
&+|\Delta Y(s)|+|\Delta Z(s)|
+\{E|X^\varepsilon(s)-X^*(s)-X^{1,\varepsilon}(s)-X^{2,\varepsilon}(s)|^2\}^\frac{1}{2}
+\{E|\Delta Y(s)|^2\}^\frac{1}{2}.
\end{aligned}
\end{equation}
As for $I_2(s)$, obviously, $$I_2(s)=\mathbbm{1}_{E_\varepsilon}(s)I_{2,1}(s)+\mathbbm{1}_{(E_\varepsilon)^c}(s)I_{2,2}(s)
=I_{2,2}(s)+\mathbbm{1}_{E_\varepsilon}(s)(I_{2,1}(s)-I_{2,2}(s)),$$
here
\begin{equation}\label{equ 4.8-3}
\begin{aligned}
I_{2,1}(s)&=f(s,X^*(s)+X^{1,\varepsilon}(s)+X^{2,\varepsilon}(s),
       X^*(s-l)+X^{1,\varepsilon}(s-l)+X^{2,\varepsilon}(s-l),\\
       &\quad \ Y^*(s)+\breve{Y}(s)+M(s),Z^*(s)+\breve{Z}(s)+p(s)\Delta \sigma(s)+B_2(s)+B_3(s)+\overline{B}_3(s),\\
        &\quad \ P_{(X^*(s)+X^{1,\varepsilon}(s)+X^{2,\varepsilon}(s),Y^*(s)+\breve{Y}(s)+M(s))},v(s))\\
        &-f(s,X^*(s),X^*(s-l),Y^*(s),Z^*(s)+p(s)\Delta\sigma(s),P_{(X^*(s),Y^*(s))}, v(s)),\\
I_{2,2}(s)&=f(s,X^*(s)+X^{1,\varepsilon}(s)+X^{2,\varepsilon}(s),
       X^*(s-l)+X^{1,\varepsilon}(s-l)+X^{2,\varepsilon}(s-l),\\
       &\quad \ Y^*(s)+\breve{Y}(s)+M(s),Z^*(s)+\breve{Z}(s)+B_2(s)+B_3(s)+\overline{B}_3(s),\\
        &\quad \ P_{(X^*(s)+X^{1,\varepsilon}(s)+X^{2,\varepsilon}(s),Y^*(s)+\breve{Y}(s)+M(s))},u^*(s))\\
        &-f(s,X^*(s),X^*(s-l),Y^*(s),Z^*(s),P_{(X^*(s),Y^*(s))}, u^*(s)).
        \end{aligned}
\end{equation}
Since $f$ is Lipschitz continuous, we can obtain
$$
\begin{aligned}
|I_{2,1}(s)-I_{2,2}(s)|&\leq |X^{1,\varepsilon}(s)+X^{2,\varepsilon}(s)|+|X^{1,\varepsilon}(s-l)+X^{2,\varepsilon}(s-l)|
+|\breve{Y}(s)+M(s)|\\
&+|\breve{Z}(s)+B_2(s)+B_3(s)+\overline{B}_3(s)|
+\{E|X^{1,\varepsilon}(s)+X^{2,\varepsilon}(s)|^2\}^{\frac{1}{2}}\\
&+\{E|\breve{Y}(s)+M(s)|^2\}^{\frac{1}{2}}.\\
\end{aligned}
$$
From (\ref{equ 3.3-2}), (\ref{equ 3.4}), (\ref{equ 4.2}) we have
\begin{equation}\label{equ 4.8-4}
\begin{aligned}
E[(\int_0^T|I_{2,1}(s)-I_{2,2}(s)|\mathbbm{1}_{E_\varepsilon}(s)ds)^2]\leq \varepsilon^2\rho(\varepsilon)
 \end{aligned}
\end{equation}
(The proof of (\ref{equ 4.8-4}) refers to Appendix).

In order to complete the proof, it remains to calculate  $I_{2,2}(s)$.  Applying the Taylor expansion, see Appendix for details,  it follows
\begin{equation}\label{equ 4.9}
\begin{aligned}
I_{2,2}(s)&=f_y(s)\breve{Y}(s)+f_z(s)\breve{Z}(s)+\widehat{E}[\widehat{f}_{\mu_2}(s)\widehat{\breve{Y}}(s)]
                    +(X^{1,\varepsilon}(s)+X^{2,\varepsilon}(s))\Big(f_x(s)+f_y(s)p(s) \\
        &+f_z(s)(p(s)\sigma_x(s)+q(s))\Big)+(X^{1,\varepsilon}(s-l)+X^{2,\varepsilon}(s-l))\Big(f_{x'} (s)
                     +f_z(s)p(s)\sigma_{x'}(s)\Big)\\
      &+\widehat{E}\Big[ (\widehat{X}^{1,\varepsilon}(s)+\widehat{X}^{2,\varepsilon}(s))
             \Big (f_z(s)p(s)\widehat{\sigma}_{\nu}(s)+\widehat{f}_{\mu_1}(s)+\widehat{f}_{\mu_2}(s)\widehat{p}(s)\Big)\Big]\\
      &+\frac{1}{2}K^\varepsilon(s)\Big(f_y(s)P(s)+f_z(s)\big(p(s)\sigma_{xx}(s)+2P(s)\sigma_x(s)+Q(s)\big)\Big)\\
      &+\frac{1}{2}\widehat{E}\Big[\widehat{K}^\varepsilon(s)\Big(f_z(s)p(s)\widehat{\sigma}_{\nu a}(s)+\widehat{f}_{\mu_2}(s)\widehat{P}(s)\Big)\Big]
     +\frac{1}{2}K^\varepsilon(s-l)\Big(f_z(s)(p(s)\sigma_{x'x'}(s)+2P_1(s)\sigma_{x'}(s)) \Big)\\
       &+K^\varepsilon_1(s)
       \Big(f_y(s)P_1(s)+f_z(s)\Big(p(s)\sigma_{xx'}(s)+P(s)\sigma_{x'}(s)+P_1(s)\sigma_x(s)+P_1(s)\sigma_x(s-l)+Q_1(s)\Big)\Big)\\
       &+\widehat{E}[\widehat{f}_{\mu_2}(s)\widehat{P}_1(s)\widehat{K}^\varepsilon_1(s)]
        +\frac{1}{2}(X^{1,\varepsilon}(s))^2 UD^2fU^\intercal +\frac{1}{2}f_{x'x'}(s)K^\varepsilon(s-l)\\
        &+K^\varepsilon_1(s)(f_{xx'}(s),f_{yx'}(s),f_{zx'}(s))U^\intercal
       +\frac{1}{2}\widehat{E}[\widehat{K}^\varepsilon(s)(\widehat{f}_{\mu_1 a_1}(s)+\widehat{f}_{\mu_2 a_2}(s)(\widehat{p}(s))^2)]+o(\varepsilon),
\end{aligned}
\end{equation}
where $U=(1,p(s), p(s)\sigma_x(s)+q(s))$, $D^2f$ is the Hessian matrix of $f$ with respect to $(x,y,z)$.\\
Notice the equality
$$E[\int_0^TE^{\mathcal{F}_s}[b_{x'}(s+l)p(s+l)X^{1,\varepsilon}(s)]ds]=E[\int_0^Tb_{x'}(s)p(s)X^{1,\varepsilon}(s-l)ds],$$
which comes from a change of variable combining with the final condition for $p(s)$ and the initial condition for $X^{1,\varepsilon}(s)$,
and notice
the fact
$$
E\widehat{E}[\widehat{f}_{\mu_2}(s)\widehat{P}_1(s)\widehat{X}^{1,\varepsilon}(s)\widehat{X}^{1,\varepsilon}(s-l)]
=E\widehat{E}[\widehat{f}^*_{\mu_2}(s)P_1(s)X^{1,\varepsilon}(s)X^{1,\varepsilon}(s-l)],
$$
which comes from the definitions of $\widehat{f}_{\mu_2}(s)$ and  $\widehat{f}^*_{\mu_2}(s)$ (see (\ref{equ 3.6}))
and the independent copy assumption, i.e., $\widehat{P}_{(\widehat{X}^*(s),\widehat{Y}^*(s))}=P_{(X^*(s),Y^*(s))}$,
as well as recall the definitions of $A_2(s), A_3(s), \overline{A_3}(s),$
it yields
\begin{equation}\label{equ 4.10}
\begin{aligned}
\Delta Y(t)= \int_t^T\Big(I_1(s)+(I_{2,2}(s)-I_{2,1}(s) )\mathbbm{1}_{E_\varepsilon}(s)
\Big)ds-\int_t^T\Delta Z(s)dW_s+o(\varepsilon),\ t\in[0,T].
\end{aligned}
\end{equation}
Thanks to (\ref{equ 4.8-2}), (\ref{equ 4.8-4}) and Gronwall lemma, one gets
\begin{equation}\label{equ 4.10}
\begin{aligned}
E[\sup_{t\in[0,T]}|\Delta Y(s)|^2+\int_0^T|\Delta Z(s)|^2ds]\leq \varepsilon^2\rho(\varepsilon).
\end{aligned}
\end{equation}
We finish the proof.
\end{proof}

\begin{remark}\label{re 4.2}
Notice $K^\varepsilon(t)=(X^{1,\varepsilon}(t))^2,\ K^\varepsilon_1(t)= X^{1,\varepsilon}(t)X^{1,\varepsilon}(t-l),$
then (\ref{equ 4.0-2}) can be written as
\begin{equation}\label{equ 4.111}
\begin{aligned}
Y^\varepsilon(t)&=Y^*(t)+p(t)(X^{1,\varepsilon}(t)+X^{2,\varepsilon}(t))+\breve{Y}(t)+o(\varepsilon)\\
&\quad+\frac{1}{2}(X^{1,\varepsilon}(t),X^{1,\varepsilon}(t-l))
\begin{pmatrix}
P(t)&P_1(t)\\
P_1(t)&0
\end{pmatrix}
\begin{pmatrix}
X^{1,\varepsilon}(t)&\\
X^{1,\varepsilon}(t-l)&
\end{pmatrix},
\end{aligned}
\end{equation}
where $(P(t),P_1(t))_{t\in[0,T+l]}$ is the solution of (\ref{equ 3.16}).\\
It implies that for the pointwise delay case $Nl<T\leq(N+1)l,\ N>1, N\in\mathbb{N}$ we naturally have
the following second-order expansion of the variation process
\begin{equation}\label{equ 4.112}
\begin{aligned}
&Y^\varepsilon(t)=Y^*(t)+p(t)(X^{1,\varepsilon}(t)+X^{2,\varepsilon}(t))+\breve{Y}(t)+o(\varepsilon)\\
&+\frac{1}{2}\begin{pmatrix}
X^{1,\varepsilon}(t)&\\
X^{1,\varepsilon}(t-l)&\\
\cdot&\\
\cdot&\\
\cdot&\\
X^{1,\varepsilon}(t-Nl)
\end{pmatrix}^\intercal
\begin{pmatrix}
P_{00}(t)&\cdot\cdot\cdot&P_{0(N-1)}(t)&P_{0N}(t)\\
P_{10}(t)&\cdot\cdot\cdot&P_{1(N-1)}(t)&P_{1N}(t)\\
\cdot\cdot\cdot&\cdot\cdot\cdot&\cdot\cdot\cdot&\cdot\cdot\cdot\\
P_{(N-1)0}(t)&\cdot\cdot\cdot&P_{(N-1)(N-1)}(t)&P_{(N-1)N}(t)\\
P_{N0}(t)&\cdot\cdot\cdot&P_{N(N-1)}(t)&0\\
\end{pmatrix}
\begin{pmatrix}
X^{1,\varepsilon}(t)&\\
X^{1,\varepsilon}(t-l)&\\
\cdot&\\
\cdot&\\
\cdot&\\
X^{1,\varepsilon}(t-Nl)
\end{pmatrix}.
\end{aligned}
\end{equation}
Here $P_{ij},\ i,j=0,1,\cdot\cdot\cdot, N$ denotes the adjoint process with $P_{ij}=P_{ji}$ and $P_{NN}\equiv0$.
Clearly, here $(P_{00},P_{01})$ is just $(P,P_1)$ in (\ref{equ 4.111}).

\end{remark}

\section{{\protect \large {Stochastic maximum principle}}}

\emph{\underline{Hamiltonian function}} We define
\begin{equation}\label{equ 5.1}
\begin{aligned}
&H(t,x,x',y,z,\nu,\mu,v;p,q,P)\\
&=pb(t,x,x',\nu,v)+q\sigma(t,x,x',\nu,v)
+\frac{1}{2}P\Big(\sigma(t,x,x',\nu,v)-\sigma(t,X^*(t),X^*(t-l), P_{X^*(t)}, u^*(t))\Big)^2\\
&\quad+f(t,x,x',y,z+p(\sigma(t,x,x',\nu,v)-\sigma(t,X^*(t),X^*(t-l), P_{X^*(t)}, u^*(t)),\mu,v),
\end{aligned}
\end{equation}
where $(t,x,x',y,z,\nu,\mu,v,p,q,P)\in[0,T]\times\mathbb{R}^4\times \mathcal{P}_2(\mathbb{R})\times \mathcal{P}_2(\mathbb{R}^2)\times U\times \mathbb{R}^3$.

\begin{theorem}\label{th 5.1}
Let the assumptions (H3.1)-(H3.3) be in force, and $\widehat{f}^*_{\mu_2}(t)>0,\ t\in[0,T], \widehat{P}\otimes P$-a.s., and
let $u^*(t)$ be the optimal control and $(X^*(t),Y^*(t),Z^*(t))$ be the optimal trajectory. Then there exist two pairs of stochastic
processes $(p,q), ((P,Q), (P_1,Q_1))$, which are the solutions of the first- and second-order adjoint equations, separately, satisfying
(\ref{equ 3.9-1}) and (\ref{equ 3.17-1}), such that for given $l,T$ with $l<T\leq2l$,  the following
inequality holds true
\begin{equation}\label{equ 5.2}
\begin{aligned}
&H(t,X^*(t), X^*(t-l),Y^*(t),Z^*(t),P_{X^*(t)}, P_{(X^*(t),Y^*(t))},v,p(t),q(t),P(t))\\
&\geq H(t,X^*(t), X^*(t-l),Y^*(t),Z^*(t),P_{X^*(t)}, P_{(X^*(t),Y^*(t))},u^*(t),p(t),q(t),P(t)),\\
&\qquad\qquad\qquad\qquad\qquad\qquad\qquad\qquad\qquad\qquad\qquad \qquad \qquad \qquad t\in[0,T], \text{a.s.}\ \text{a.e.}
\end{aligned}
\end{equation}
\end{theorem}
\begin{proof}
From the definition of $J(v(\cdot))$, (\ref{equ 4.0-1}) and  $X^{1,\varepsilon}(0)=X^{2,\varepsilon}(0)=0$, we derive
$$
J(u^\varepsilon(\cdot))-J(u^*(\cdot))=Y^\varepsilon(0)-Y^*(0)=\breve{Y}(0)+o(\varepsilon)\geq0,
$$
where $\breve{Y}$ is the solution of the equation (\ref{equ 4.1}).

In order to complete the proof, let us
introduce the following mean-field SDE
\begin{equation}\label{equ 5.3}
  \left\{
   \begin{aligned}
d\Gamma(t)&=(f_y(t)\Gamma(t)+E[\widehat{f}^*_{\mu_2}(t)\widehat{\Gamma}(t)])dt+f_z(t)\Gamma(t)dW_t,\ t\in[0,T],\\
\Gamma(0)&=1.
\end{aligned}
   \right.
\end{equation}
We argue that
\begin{equation}\label{equ 5.4}
\Gamma(t)>0,\ t\in[0,T],\ P\text{-a.s.}
\end{equation}
Indeed, consider a forward stochastic differential equation:
\begin{equation}\label{equ 5.3}
  \left\{
   \begin{aligned}
d\Gamma^o(t)&=f_y(t)\Gamma^o(t)dt+f_z(t)\Gamma^o(t)dW_t,\ t\in[0,T],\\
\Gamma^0(0)&=1.
\end{aligned}
   \right.
\end{equation}
Set $\Delta\Gamma(t):=\Gamma^o(t)-\Gamma(t)$.  Applying It\^{o}'s formula to $(\Delta\Gamma(t)^+)^2$
and notice  $\widehat{f}^*_{\mu_2}>0$, we have
\begin{equation}\label{equ 070501}
\begin{aligned}
&E((\Delta\Gamma(t))^+)^2\\
=&E\int_0^t2\Delta\Gamma(s)^+f_y(s)\Delta\Gamma(s)+1_{\{\Delta\Gamma(s)>0\}}f_z^2(s)(\Delta\Gamma(s))^2-2\Delta\Gamma(s)^+
\widehat{E}[\widehat{f}^*_{\mu_2}(s)
\widehat{\Gamma}(s)]ds\\
=&E\int_0^t[2f_y(s)+f_z^2(s)](\Delta\Gamma(s)^+)^2ds+2E\int_0^t\Delta\Gamma(s)^+\widehat{E}[\widehat{f}^*_{\mu_2}(s)
(\Delta\widehat{\Gamma}(s)-\widehat{\Gamma}^o(s))]ds\\
\leq &E\int_0^t[2f_y(s)+f_z^2(s)](\Delta\Gamma(s)^+)^2ds+2E\int_0^t\Delta\Gamma(s)^+\widehat{E}[\widehat{f}^*_{\mu_2}(s)
\Delta\widehat{\Gamma}(s)]ds\\
\leq &CE\int_0^t(\Delta\Gamma(s)^+)^2ds.
\end{aligned}
\end{equation}
The Gronwall lemma allows to show $\Gamma(t)\geq\Gamma^o(t)>0,\ t\in[0,T],\ P$-a.s.\\
Now applying It\^{o}'s formula to  $\breve{Y}(t)\Gamma(t)$, one has
\begin{equation}\label{equ 5.4}
\breve{Y}(0)=\widehat{E}[\int_0^T\Gamma(s)(A_1(s)+\Delta f(s))\mathbbm{1}_{E_\varepsilon}(s)ds].
\end{equation}
Combining $\Gamma(t)\geq\Gamma^o(t)>0,\ t\in[0,T],\ P$-a.s.,
and applying the Lebesgue differentiation theorem,
we obtain (\ref{equ 5.2}).
\end{proof}

\begin{remark}\label{re 5.2}
From the above proof, we know that
nothing would change for the pointwise delay case: $Nl<T\leq(N+1)l,\ N\in \mathbb{N},\ N>1$ once obtaining the expansion (\ref{equ 4.112}).
\end{remark}

Hence, as for the pointwise delay case, we give the following result without proof.
\begin{corollary}
Let the assumptions (H3.1)-(H3.3) be in force, and $\widehat{f}^*_{\mu_2}(t)>0,\ t\in[0,T], \widehat{P}\otimes P$-a.s., and
let $u^*(t)$ be the optimal control and $(X^*(t),Y^*(t),Z^*(t))$ be the optimal trajectory. Then there exist two pair of stochastic
processes $(p,q), ((P_{00},Q_{00}),(P_{01},Q_{01}),\cdot\cdot\cdot,(P_{(N-1)N},Q_{(N-1)N}))$
which are the solutions of the first- and second-order adjoint equations,
separately,
 such that for given $l,T$ with $(N-1)l<T\leq Nl,\ N\in \mathbb{N}, N\geq1$,  the following
inequality holds true
\begin{equation}\label{equ 5.7}
\begin{aligned}
&H(t,X^*(t), X^*(t-l),Y^*(t),Z^*(t),P_{X^*(t)}, P_{(X^*(t),Y^*(t))},v,p(t),q(t),P_{00}(t))\\
&\geq H(t,X^*(t), X^*(t-l),Y^*(t),Z^*(t),P_{X^*(t)}, P_{(X^*(t),Y^*(t))},u^*(t),p(t),q(t),P_{00}(t)),\\
&\qquad\qquad\qquad\qquad\qquad\qquad\qquad\qquad\qquad\qquad\qquad \qquad \qquad \qquad t\in[0,T], \text{a.s.}\ \text{a.e.}
\end{aligned}
\end{equation}
\end{corollary}

\section{{\protect \large {Application to Finance}}}

Let us consider a continuous-time financial market with a risk-free bond and a risky share being traded. The bond price process denoted by
$\{S_0(t)\}_{t\in[0,T]}$ is described by
\begin{equation}\label{equ 6.1}
  \left\{
   \begin{aligned}
dS_0(t)&=\alpha(t)S_0(t)dt,\ t\in[0,T],\\
S_0(0)&=1,
\end{aligned}
   \right.
\end{equation}
where $\alpha(t)$ is the risk-free interest rate. It is a uniformly bounded, deterministic function.\\
The price process of the risky share $\{S_1(t)\}_{t\in[0,T]}$ evolves following the equation
\begin{equation}\label{equ 6.2}
  \left\{
   \begin{aligned}
dS_1(t)&=b(t)S_1(t)dt+\sigma(t)S_1(t)dW(t),\ t\in[0,T],\\
S_1(0)&=s_1>0,
\end{aligned}
   \right.
\end{equation}
where $b:[0,T]\rightarrow\mathbb{R},\ \sigma:[0,T]\rightarrow(0,+\infty)$ are uniformly bounded deterministic function
of $t$, which denote the appreciation rate and the volatility, respectively.

By $v(t)$ we denote the amount of an investor's wealth invested in the share at time $t$. Then the value $U(t)$ of the
investor of a self-financing portfolio consisting of the risk-free and the risky assets is
 \begin{equation}\label{equ 6.3}
 \begin{aligned}
dU(t)=(\alpha(t)U(t)+v(t)(b(t)-\alpha(t)))dt+v(t)\sigma(t)dW(t).
\end{aligned}
\end{equation}
But in practice there must be some unavoidable delays arising in various situations. Hence,
Shen, Meng and Shi \cite{SMS} considered the following wealth process with delays and jumps
 \begin{equation}\label{equ 6.4}
 \left\{
 \begin{aligned}
dU(t)&=(\alpha(t)U(t)+\beta(t)U(t-l)+v(t)(b(t)-\alpha(t)))dt+v(t)\sigma(t)dW(t)\\
      &\quad+\int_{\mathbb{R}_0}v(t)r(t,z)\widetilde{N}(dz,dt),\ t\in[0,T],\\
  U(t)&=x>0,\ t\in[-l,0].
\end{aligned}
\right.
\end{equation}

In our case, we want to consider the corresponding mean-field game without the jump term. More precisely,
suppose $T< l\leq 2T$ and
let $g^i: \mathbb{R}\rightarrow\mathbb{R}$ be a twice continuously differentiable functions. Consider $N$
individual investors and suppose that the $i$-th investor's wealth process is of the form
 \begin{equation}\label{equ 6.5}
 \left\{
 \begin{aligned}
dU^{i,v^i}(t)&=(\alpha^i(t)U^{i,v^i}(t)+\beta^{i}(t)U^{i,v^i}(t-l)+g^i(\overline{\nu}^N(t))
+v^i(t)(b(t)-\alpha(t))dt\\
              &\quad+v^i(t)\sigma(t)dW^i(t),\  t\in[0,T],\\
  U^{i,v^i}(0)&=x^i>0,\ t\in[-l,0],
\end{aligned}
\right.
\end{equation}
where \ $\overline{\nu}^N(s)=\frac{1}{N}\sum_{i=1}^N\delta_{U^{i,v^i}(s)}$ and $\delta_x$ is the Dirac measure at $x$.

Assume that the $i$-th investor wants to minimize the utility $V^{i,v^i}(t)$ resulting from $v^i$.
Without loss of generality, we define
 \begin{equation}\label{equ 6.6}
V^{i,v^i}(t)=E^{\mathcal{F}_t}[U^{i,v^i}(T)+\int_t^Tf^i(s,V^{i,v^i}(s),\overline{\overline{\nu}}^N(s),v^i(s))ds].
\end{equation}
Here  $f^i$ satisfies the assumptions (H3.1)-(H3.3) and
$\overline{\overline{\nu}}^N(s)=\frac{1}{N}
\sum_{i=1}^N\delta_{(U^{i,v^i}(s),V^{i,v^i}(s))}.$\\
Following El Karoui, Peng and Quenez \cite{KPQ}  the recursive utility $V^{i,v^i}(t)$ satisfies the equation
\begin{equation}\label{equ 6.7}
 \left\{
 \begin{aligned}
 -dV^{i,v^i}(t)&=f^i(t,V^{i,v^i}(t),\overline{\overline{\nu}}^N(t),v^i(t))dt-R^{i,v^i}(t)dW^i(t),\ t\in[0,T],\\
      V^{i,v^i}(T)&= U^{i,v^i}(T).
\end{aligned}
\right.
\end{equation}
Consequently, the control problem can be formulated as
\begin{equation}\label{equ 6.8}
 \left\{
 \begin{aligned}
 dU^{i,v^i}(t)&=(\alpha^i(t)U^{i,v^i}(t)+\beta^{i}(t)U^{i,v^i}(t-l)+g^i(\overline{\nu}^N(t))
+v^i(t)(b(t)-\alpha(t))dt\\
              &\quad+v^i(t)\sigma(t)dW^i(t),\  t\in[0,T],\\
  U^{i,v^i}(0)&=x^i>0,\ t\in[-l,0],\\
  -dV^{i,v^i}(t)&=f^i(t,V^{i,v^i}(t),\overline{\overline{\nu}}^N(t),v^i(t))dt-R^{i,v^i}(t)dW^i(t),\ t\in[0,T],\\
      V^{i,v^i}(T)&= U^{i,v^i}(T),
\end{aligned}
\right.
\end{equation}
and the target is to minimum $J^i(v^i(\cdot))=V^{i,v^i}(0),\ 1\leq i\leq N.$

Suppose that the game is symmetric, in other words, $\alpha^i(t)=\alpha(t), \beta^i(t)=\beta(t), x^i=x,f^i=f, g^i=g$
are all independent of $i$, and let $N$ converge to $+\infty$, then following the scheme of Lasry and Lions we can find
approximate Nash equilibriums by a limiting dynamics and assign representative investor the unified strategy $v$, which is
determined by the following mean-field FBSDE
\begin{equation}\label{equ 6.9}
 \left\{
 \begin{aligned}
 dX^v(t)&=(\alpha(t)X^v(t)+\beta (t)X^v(t-l)+g(P_{X^v(t)})
+v(t)(b(t)-\alpha(t))dt\\
              &\quad+v(t)\sigma(t)dW(t),\  t\in[0,T],\\
  X^v(0)&=x >0,\ t\in[-l,0],\\
  -dY^v(t)&=f(t,Y^v(t),P_{(X^v(t),Y^v(t))},v(t))dt-Z^v(t)dW(t),\ t\in[0,T],\\
      Y^v (T)&= X^v(T).
\end{aligned}
\right.
\end{equation}
Accordingly, the cost functional changes to
\begin{equation}\label{equ 6.10}
J(v(\cdot))=Y^v(0).
\end{equation}

Let $u^*(t)$ be the optimal control, $(X^*, Y^*, Z^*)$ be the optimal trajectory. We define \emph{Hamiltonian function}
\begin{equation}\label{equ 6.11}
\begin{aligned}
H(t,x,x',y,\nu,\mu,v;p,q,P)&=p(\alpha(t)x+\beta(t)x'+\nu+v(b(t)-\alpha(t)))\\
&\quad+qv\sigma(t)+\frac{1}{2}P(v-u^*(t))^2(\sigma(t))^2+f(t,y,\mu,v).
\end{aligned}
\end{equation}
Then from Theorem \ref{th 5.1} we know that $u^*$ satisfies
\begin{equation}\label{equ 6.12}
\begin{aligned}
&H(t,X^*(t),X^*(t-l),Y^*(t),P_{X^*(t)},P_{(X^*(t),Y^*(t))},v;p(t),q(t),P(t))\\
&\geq   H(t,X^*(t),X^*(t-l),Y^*(t),P_{X^*(t)},P_{(X^*(t),Y^*(t))},u^*(t);p(t),q(t),P(t))
\end{aligned}
\end{equation}
i.e.,
\begin{equation}\label{equ 6.13}
\begin{aligned}
&\Big(p(t)(b(t)-\alpha(t))-q(t)\sigma(t)\Big)(v-u^*(t))+\frac{1}{2}P(t)(\sigma(t))^2(v-u^*(t))^2\\
&+f(t,Y^*(t),P_{(X^*(t),Y^*(t))},v)-f(t,Y^*(t),P_{(X^*(t),Y^*(t))},u^*(t))\geq0,
\end{aligned}
\end{equation}
where $(p,q), ((P,Q),(P_1,Q_1))$ are the solutions of the first- and second-order adjoint equations, respectively,
\begin{equation}\label{equ 6.14}
 \left\{
 \begin{aligned}
-dp(t)&=\Big(p(t)(f_y(t)+\widehat{E}[\widehat{f}_{\mu_2}^*(t)]+\alpha(t) )+\widehat{E}[\widehat{p}(t)\widehat{g}^*_\nu(t)]
+E^{\mathcal{F}_t}[p(t+l)\beta(t+l)]+\widehat{E}[\hat{f}^*_{\mu_1}(t)]\Big)dt\\
&\quad-q(t)dW(t),\ t\in[0,T],\\
p(T)&=1,\ p(t)=0,\ t\in(T,T+l],\ q(t)=0,\ t\in[T,T+l].
\end{aligned}
\right.
\end{equation}

\begin{equation}\label{equ 6.15}
 \left\{
 \begin{aligned}
-dP(t)&=\Big(P(t)(f_y(t)+\widehat{E}[\widehat{f}^*_{\mu_2}(t)]+2\alpha(t))+2E^{\mathcal{F}_t}[P_1(t+l)\beta(t+l)]
+\widehat{E}[\widehat{p}(t)\widehat{g}^*_{\nu a}(t)]\\
&\quad +(p(t))^2f_{yy}(t)+\widehat{E}[\widehat{f}^*_{\mu_1 a_1}(t)]+(p(t))^2\widehat{E}[\widehat{f}^*_{\mu_2 a_2}(t)]\Big)dt-Q(t)dW(t),\ t\in[0,T),\\
P(t)&=0,\ Q(t)=0,\ t\in[T,T+l],\\
-dP_1(t)&=\Big(P_1(t)(\alpha(t)+\alpha(t-l)+f_y(t)+\widehat{E}[\widehat{f}^*_{\mu_2}(t)])+P(t)\beta(t)\Big)dt-Q_1(t)dW(t),\ t\in[0,T),\\
P_1(t)&=0,\ Q_1(t)=0,\ t\in[T,T+l],\\
\end{aligned}
\right.
\end{equation}
where $$
 \begin{aligned}
(\widehat{g}^*_\nu, \widehat{g}^*_{\nu a})(t)&=(\frac{\partial g}{\partial \nu},\frac{\partial^2 g}{\partial \nu  \partial a})(P_{X^*(t)};\widehat{X}^*(t)),\\
(\widehat{f}^*_{\mu_i},\widehat{f}^*_{\mu_i a_i})(t)
&=((\frac{\partial f}{\partial \mu})_i,\frac{\partial}{\partial{a_i}}((\frac{\partial f}{\partial \mu})_i))
(t,\widehat{Y}^*(t),P_{(X^*(t),Y^*(t))},\widehat{u}^*(t);X^*(t),Y^*(t)),\ i=1,2.
\end{aligned}$$

\begin{remark}\label{re 6.1}
In (\ref{equ 6.14}), (\ref{equ 6.15}), we don't  give the specific expression of the derivatives of the coefficients $g, f$ with respect to the
measure. But we argue that it is possible to show the accurate derivatives for some special $f, g$ with respect to the measure. For example,
let $h_0, h_1:\mathbb{R}\rightarrow\mathbb{R},\ h_2:\mathbb{R}^2\rightarrow\mathbb{R}$ be twice continuously differentiable functions with bounded
derivatives of all order, and $f_1:[0,T]\times \mathbb{R}\times U\rightarrow\mathbb{R}$ satisfy assumptions (H3.1)-(H3.3).
For $\xi,\eta\in L^2(\Omega,\mathcal{F},P)$, define
$$
 \begin{aligned}
g(P_\xi)&:=h_0(E[h_1(\xi)]), \quad
f(t,y,P_{(\xi,\eta)},v):=f_1(t,y,v)+h_0(E[h_2(\xi,\eta)]).
\end{aligned}
$$
Then
$$
\begin{aligned}
&\frac{\partial g}{\partial\nu}(P_\xi;a)=h_0'(E[h_1(\xi)])h_1'(a),\quad
\frac{\partial^2g}{\partial \nu \partial a}(P_\xi;a)=h_0'(E[h_1(\xi)])h_1''(a).\\
\end{aligned}
$$

\noindent For fixed $\eta\in L^2(\Omega,\mathcal{F},P)$,
$$
\begin{aligned}
&(\frac{\partial f}{\partial\mu})_1(t,y,P_{(\xi,\eta)},v;a_1,\eta)
=h'_0(E[h_2(\xi,\eta)])\frac{ \partial h_2}{\partial a_1}(a_1,\eta),\\
&\frac{\partial}{\partial a_1}((\frac{\partial f}{\partial\mu})_1)(t,y,P_{(\xi,\eta)},v;a_1,\eta)=h'_0(E[h_2(\xi,\eta)])
\frac{ \partial^2 h_2}{\partial a_1^2}(a_1,\eta).\\
\end{aligned}
$$
$(\frac{\partial f}{\partial\mu})_2$ and $\frac{\partial}{\partial a_2}((\frac{\partial f}{\partial\mu})_2)$
can be understood in the same meaning.

\end{remark}

\section{{\protect \large {Appendix}}}

\subsection{{\protect \large {Notations}}}
$$
\begin{aligned}
M(t)&=p(t)(X^{1,\varepsilon}(t)+X^{2,\varepsilon}(t))
+\frac{1}{2}P(t)K^\varepsilon(t)
 +P_1(t)K^\varepsilon_1(t),\\
 A_1(t)&= p(t)\Delta b(t)+q(t)\Delta\sigma(t)+\frac{1}{2}P(t)(\Delta \sigma(t))^2,\\
A_2(t)&=(X^{1,\varepsilon}(t)+X^{2,\varepsilon}(t))\Big(p(t)b_x(t)+q(t)\sigma_x(t)-F(t)\Big)
+(X^{1,\varepsilon}(t-l)+X^{2,\varepsilon}(t-l))\cdot\\
&\quad\ \Big(p(t)b_{x'}(t)+q(t)\sigma_{x'}(t)\Big)
+\widehat{E}\Big[(\widehat{X}^{1,\varepsilon}(t)+\widehat{X}^{2,\varepsilon}(t))\Big(p(t)\widehat{b}_{\nu}(t)+q(t)\widehat{\sigma}_{\nu}(t)\Big)\Big],\\
A_3(t)&=\frac{1}{2}(X^{1,\varepsilon}(t))^2\Big(
b_{xx}(t)p(t)+\sigma_{xx}(t)q(t)+2b_x(t)P(t)+(\sigma_x(t))^2P(t)+2\sigma_x(t)Q(t)-G(t)\Big)\\
&\quad+\frac{1}{2}(X^{1,\varepsilon}(t-l))^2\Big(
b_{x'x'}(t)p(t)+\sigma_{x'x'}(t)q(t)+(\sigma_{x'}(t))^2P(t)+2b_{x'}(t)P_1(t)\\
&\quad+2\sigma_{x'}(t)\sigma_{x}(t-l)P_1(t)+2\sigma_{x'}(t)Q_1(t)\Big)
+\frac{1}{2}\widehat{E}[(X^{1,\varepsilon}(t))^2(p(t)\widehat{b}_{\nu a}(t)+q(t)\widehat{\sigma}_{\nu a}(t))],\\
\overline{A}_3(t)&=(X^{1,\varepsilon}(t)X^{1,\varepsilon}(t-l))
\Big\{b_{xx'}(t)p(t)+\sigma_{xx'}(t)q(t)+b_{x'}(t)P(t)+\sigma_x(t)\sigma_{x'}(t)P(t)+\sigma_{x'}(t)Q(t) \\
&\quad+b_x(t)P_1(t)+b_x(t-l)P_1(t)+\sigma_x(t)\sigma_x(t-l)P_1(t)+\sigma_x(t)Q_1(t)+\sigma_x(t-l)Q_1(t)-G_1(t)\Big\},\\
A_4(t)&=\mathbbm{1}_{E_\varepsilon}(t)X^{1,\varepsilon}(t)
\Big\{p(t)\Delta b_x(t)+q(t)\Delta \sigma_x(t)\Big\}
+\mathbbm{1}_{E_\varepsilon}(t)X^{1,\varepsilon}(t-l)\Big\{p(t)\Delta b_{x'}(t-l)+q(t)\Delta \sigma_{x'}(t-l)\Big\}\\
&\quad+\widehat{E}[\widehat{X}^{1,\varepsilon}(t)(p(t)\Delta \widehat{b}_\nu(t)+q(t)\Delta \widehat{\sigma}_\nu(t))]
+\frac{1}{2}P(t)L_1(t)+\frac{1}{2}P_1(t)L_3(t),\\
\end{aligned}
$$
$$
\begin{aligned}
B_2(t)&=(X^{1,\varepsilon}(t)+X^{2,\varepsilon}(t))\Big(p(t)\sigma_x(t)+q(t)\Big)
+(X^{1,\varepsilon}(t-l)+X^{2,\varepsilon}(t-l))p(t)\sigma_{x'}(t)\\
&\quad\  \widehat{E}[(\widehat{X}^{1,\varepsilon}(t)+\widehat{X}^{2,\varepsilon}(t))p(t)\widehat{\sigma}_{\nu}(t)],\\
B_3(t)&=\frac{1}{2}(X^{1,\varepsilon}(t))^2\Big(p(t)\sigma_{xx}(t)+2P(t)\sigma_x(t)+Q(t)\Big)
+\frac{1}{2}(X^{1,\varepsilon}(t-l))^2\Big(p(t)\sigma_{x'x'}(t)+2P_1(t)\sigma_{x'}(t)\Big)\\
&\quad+(X^{1,\varepsilon}(t)X^{1,\varepsilon}(t-l))
\Big(p(t)\sigma_{xx'}(t)+P(t)\sigma_{x'}(t)
+P_1(t)\sigma_x(t)+P_1(t)\sigma_x(t-l)+Q_1(t)\Big)\\
&\quad+\frac{1}{2}\widehat{E}[p(t)\widehat{\sigma}_{\nu a}(t)(\widehat{X}^{1,\varepsilon})^2],\\
\overline{B}_3(t)&=(X^{1,\varepsilon}(t)X^{1,\varepsilon}(t-l))
\Big(p(t)\sigma_{x'}(t)+P_1(t)\sigma_x(t)+P_1(t)\sigma_x(t-l)+Q_1(t)\Big),\\
B_4(t)&=\mathbbm{1}_{E_\varepsilon}(t)X^{1,\varepsilon}(t)p(t)\Delta \sigma_x(t)
+\mathbbm{1}_{E_\varepsilon}(t)X^{1,\varepsilon}(t-l)p(t)\Delta \sigma_{x'}(t)
+\mathbbm{1}_{E_\varepsilon}(t)\widehat{E}[\Delta \widehat{\sigma}_\nu(t)\widehat{X}^{1,\varepsilon}]\\
&\quad+\frac{1}{2}P(t)L_2(t)+\frac{1}{2}P_1(t)L_4(t).
\end{aligned}
$$
Here $F(t), G(t), G_1(t)$ are given in (\ref{equ 3.9}),  (\ref{equ 3.17});
$L_i(t), i=1,2,3,4$ are introduced in (\ref{equ 3.11}) and (\ref{equ 3.14}).

\subsection{{\protect \large {The proof of (\ref{equ 3.12})}}}
In order to prove (\ref{equ 3.12}), we just need to prove the following three estimates since the
similar argument can be applied to calculate the rest terms:
\begin{equation}\label{equ 7.1}
\begin{aligned}
 &\mathrm{a)}\  E[(\int_0^T|X^{1,\varepsilon}(t)2\Delta b(t)\mathbbm{1}_{E_\varepsilon}(t)|dt)^2]\leq \varepsilon^2\rho(\varepsilon);\\
&\mathrm{b)}\  E[(\int_0^T|2\sigma_{x'}(t)X^{1,\varepsilon}(t-l)\widehat{E}[\widehat{\sigma}_\nu^*(t)\widehat{X}^{1,\varepsilon}(t)]|dt)^2]\leq \varepsilon^2\rho(\varepsilon);\\
&\mathrm{c)}\  E[(\int_0^T|2\Delta\sigma(t)\widehat{E}[\widehat{\sigma}_\nu^*(t)\widehat{X}^{1,\varepsilon}(t)]\mathbbm{1}_{E_\varepsilon}(t)|dt)^2]\leq \varepsilon^2\rho(\varepsilon).
\end{aligned}
\end{equation}

 For  $\mathrm{a)}$, due to the coefficient $b$ being continuous with respect to $v$,  (\ref{equ 3.4}) and H\"{o}lder inequality
 allow to show
\begin{equation}\label{equ 7.2}
\begin{aligned}
&E[(\int_0^T|X^{1,\varepsilon}(t)2\Delta b(t)\mathbbm{1}_{E_\varepsilon}(t)|dt)^2]\\
&\leq C\varepsilon^2\Big\{E[\sup_{t\in[0,T]}|X^{1,\varepsilon}(t)|^4]\Big\}^\frac{1}{2}
\Big\{E[\sup_{t\in[0,T]}|v(t)-u^*(t)|^4]\Big\}^\frac{1}{2}\leq C\varepsilon^3.
\end{aligned}
\end{equation}
As for $\mathrm{b)}$, from the boundness of $\sigma_{x'}$, (\ref{equ 3.3-2})  and (\ref{equ 3.4}), one can
check
\begin{equation}\label{equ 7.4}
\begin{aligned}
&E[(\int_0^T|2\sigma_{x'}(t)X^{1,\varepsilon}(t-l)\widehat{E}[\widehat{\sigma}_\nu^*(t)\widehat{X}^{1,\varepsilon}(t)]|dt)^2]\\
&\leq CT^{\frac{3}{2}}\Big\{E\Big[\sup_{t\in[0,T]}|X^{1,\varepsilon}(t-l)|^4\Big]\Big\}^\frac{1}{2}
\Big\{E\Big[ \int_0^T |\widehat{E}[\widehat{\sigma}_\nu^*(t)\widehat{X}^{1,\varepsilon}(t)]|^4dt\Big]\Big\}^\frac{1}{2}
\leq \varepsilon^2\rho(\varepsilon).
\end{aligned}
\end{equation}
We now analyse $\mathrm{c)}$. Notice that  $\sigma$ is continuous in $v$, and $\sigma_\nu$
is bounded, then (\ref{equ 3.3-2}) implies
\begin{equation}\label{equ 7.5}
\begin{aligned}
&E[(\int_0^T|2\Delta\sigma(t)\widehat{E}[\widehat{\sigma}_\nu^*(t)\widehat{X}^{1,\varepsilon}(t)]\mathbbm{1}_{E_\varepsilon}(t)|dt)^2]\\
&\leq C \varepsilon^2 E\Big[\sup_{t\in[0,T]}|X^{1,\varepsilon}(t)|^2\Big]
E\Big[\sup_{t\in[0,T]}|v(t)-u^*(t)|^2\Big]\leq C\varepsilon^3.
\end{aligned}
\end{equation}

\subsection{{\protect \large {The proof of (\ref{equ 4.8-4})}}}

\emph{Proof of (\ref{equ 4.8-4}).}  Notice that
\begin{equation}\label{equ 7.2-1}
\begin{aligned}
&E[(\int_0^T|I_{2,1}(s)-I_{2,2}(s)|\mathbbm{1}_{E_\varepsilon}(s)ds)^2]\\
&=E\Big[\Big(\int_0^T\mathbbm{1}_{E_\varepsilon}(s)\Big(
|X^{1,\varepsilon}(s)+X^{2,\varepsilon}(s)|+|X^{1,\varepsilon}(s-l)+X^{2,\varepsilon}(s-l)|
+|\breve{Y}(s)+M(s)|+|\breve{Z}(s)\\
&\quad+B_2(s)+B_3(s)+\overline{B}_3(s)|+\{E|\breve{Y}(s)+M(s)|^2\}^{\frac{1}{2}}
+\{E|X^{1,\varepsilon}(s)+X^{2,\varepsilon}(s)|^2\}^\frac{1}{2}\Big)ds\Big)^2\Big].
\end{aligned}
\end{equation}
We mainly prove
\begin{equation}\label{equ 7.2-2}
\begin{aligned}
E\Big[\Big(\int_0^T\mathbbm{1}_{E_\varepsilon}(s)
|\breve{Z}(s)+B_2(s)+B_3(s)+\overline{B}_3(s)|ds\Big)^2\Big]\leq\varepsilon^2\rho(\varepsilon),
\end{aligned}
\end{equation}
since the other terms can be argued similarly.\\
Observe that
\begin{equation}\label{equ 7.2-3}
\begin{aligned}
&E\Big[\Big(\int_0^T\mathbbm{1}_{E_\varepsilon}(s)|\breve{Z}(s)+B_2(s)+B_3(s)+\overline{B}_3(s)|ds\Big)^2\Big]\\
&\leq C E[(\int_0^T\mathbbm{1}_{E_\varepsilon}(s)|\breve{Z}(s)|ds)^2]
+CE[(\int_0^T\mathbbm{1}_{E_\varepsilon}(s)|B_2(s)|ds)^2]\\
&\quad+CE[(\int_0^T\mathbbm{1}_{E_\varepsilon}(s)|B_3(s)|ds)^2]
+CE[(\int_0^T\mathbbm{1}_{E_\varepsilon}(s)|\overline{B}_3(s)|ds)^2]\Big\}\\
&=: \triangle_1(\varepsilon)+\triangle_2(\varepsilon)+\triangle_3(\varepsilon)+\triangle_4(\varepsilon) .
\end{aligned}
\end{equation}
From the definitions of $B_3$ and $\overline{B}_3$, it is not difficult to check
\begin{equation}\label{equ 7.2-4}
\bigtriangleup_1(\varepsilon)+\bigtriangleup_3(\varepsilon)+\bigtriangleup_4(\varepsilon)
\leq \varepsilon^2\rho(\varepsilon).
\end{equation}

As for $\triangle_2(\varepsilon)$, it can be written as
\begin{equation}\label{equ 7.2-5}
\begin{aligned}
\bigtriangleup_2(\varepsilon)
&\leq C E[(\int_0^T\mathbbm{1}_{E_\varepsilon}(s)|(X^{1,\varepsilon}(s)+X^{2,\varepsilon}(s))(p(s)\sigma_x(s)+q(s))|ds)^2] \\
& +CE[(\int_0^T\mathbbm{1}_{E_\varepsilon}(s)|(X^{1,\varepsilon}(s-l)+X^{2,\varepsilon}(s-l))p(s)\sigma_{x'}(s)|ds)^2] \\
&+ CE[(\int_0^T\mathbbm{1}_{E_\varepsilon}(s)|\widehat{E}[(\widehat{X}^{1,\varepsilon}(s)
+\widehat{X}^{2,\varepsilon}(s))p(s)\widehat{\sigma}_{\nu}(s)]|ds)^2]\\
&=:\triangle_{2,1}(\varepsilon)+\triangle_{2,2}(\varepsilon)+\triangle_{2,3}(\varepsilon).
\end{aligned}
\end{equation}
However, from the boundness of $\sigma_x$, (\ref{equ 3.4}) and (\ref{equ 3.9-1}), it follows
\begin{equation}\label{equ 7.2-5}
\begin{aligned}
\triangle_{2,1}(\varepsilon)
&\leq CE\Big[(\sup_{t\in[0,T]}|X^{1,\varepsilon}(s)|^2+\sup_{t\in[0,T]}|X^{2,\varepsilon}(s)|^2)
\int_0^T|p(s)\sigma_x(s)+q(s)|^2\mathbbm{1}_{E_\varepsilon}(s)ds\cdot\varepsilon\Big]\\
&\leq C\varepsilon
\Big(\{E[\sup_{t\in[0,T]}|X^{1,\varepsilon}(s)|^4]\}^\frac{1}{2}+\{E[\sup_{t\in[0,T]}|X^{2,\varepsilon}(s)|^4]\}^\frac{1}{2}\Big)\cdot\\
&\qquad\qquad\qquad\qquad\qquad\qquad\qquad\qquad\quad\{E[(\int_0^T|p(s)\sigma_x(s)+q(s)|^2\mathbbm{1}_{E_\varepsilon}(s)ds)^2]\}^\frac{1}{2}\\
&\leq C\varepsilon(\varepsilon+\varepsilon^2)\Big(\varepsilon+(E[(\int_0^T|q(s)|^2\mathbbm{1}_{E_\varepsilon}(s)ds)^2])^\frac{1}{2}\Big).
\end{aligned}
\end{equation}
The Dominated Convergence Theorem implies $(E[(\int_0^T|q(s)|^2\mathbbm{1}_{E_\varepsilon}(s)ds)^2])^\frac{1}{2}\rightarrow0$ as $\varepsilon\downarrow0$. Consequently, $\rho_1(\varepsilon):=C(1+\varepsilon)\Big(\varepsilon+(E[(\int_0^T|q(s)|^2\mathbbm{1}_{E_\varepsilon}(s)ds)^2])^\frac{1}{2}\Big)\rightarrow0$
as $\varepsilon\downarrow0$.
$\triangle_{2,2}(\varepsilon)$ can be estimated with the similar argument.\\
We now turn to analyse $\triangle_{2,3}(\varepsilon)$. Notice
\begin{equation}\label{equ 7.2-6}
\begin{aligned}
\triangle_{2,3}(\varepsilon)& \leq CE[(\int_0^T\mathbbm{1}_{E_\varepsilon}(s)|p(s)||\widehat{E}[\widehat{X}^{1,\varepsilon}(s)\widehat{\sigma}_{\nu}(s)]|ds)^2]
+CE[(\int_0^T\mathbbm{1}_{E_\varepsilon}(s)|p(s)||\widehat{E}[\widehat{X}^{2,\varepsilon}(s)\widehat{\sigma}_{\nu}(s)]|ds)^2]\\
&:=\triangle_{2,3,1}(\varepsilon)+\triangle_{2,3,2}(\varepsilon).
\end{aligned}
\end{equation}
Thanks to  (\ref{equ 3.3-2}), we have
\begin{equation}\label{equ 7.2-7}
\begin{aligned}
\triangle_{2,3,1}(\varepsilon)&\leq CE\Big[\sup_{s\in[0,T]}|p(s)|^2\cdot\varepsilon\cdot \int_0^T|\widehat{E}[\widehat{X}^{1,\varepsilon}(s)\widehat{\sigma}_{\nu}(s)|^2ds\Big]\\
&\leq CT^\frac{1}{2}\varepsilon\Big\{E[\int_0^T|\widehat{E}[\widehat{X}^{1,\varepsilon}(s)\widehat{\sigma}_{\nu}(s)|^4ds]\Big\}^\frac{1}{2}
\leq \varepsilon^2\rho_2(\varepsilon),
\end{aligned}
\end{equation}
where $\rho_2(\varepsilon)\rightarrow0$ as  $\varepsilon\rightarrow0$.\\
Besides, obviously,
\ $\triangle_{2,3,2}(\varepsilon)\leq  C\varepsilon^4.$
Hence, we obtain
$\bigtriangleup_2(\varepsilon)\leq \varepsilon^2\rho(\varepsilon).$

\subsection{{\protect \large {The expansion of $I_{2,2}(s)$}}}

The first-order Taylor expansion of $I_{2,2}(s)$ allows  to show
\begin{equation}\label{equ 7.3}
\begin{aligned}
I_{2,2}&=f_x(s)(X^{1,\varepsilon}(s)+X^{2,\varepsilon}(s))+f_{x'}(s)(X^{1,\varepsilon}(s-l)+X^{2,\varepsilon}(s-l))+f_y(s)(\breve{Y}(s)+M(s))\\
&\quad+f_z(s)(\breve{Z}(s)+B_2(s)+B_3(s)+\overline{B}_3(s))
+\widehat{E}[\widehat{f}_{\mu_1}(s)(\widehat{X}^{1,\varepsilon}(s)+\widehat{X}^{2,\varepsilon}(s))]\\
&\quad+\widehat{E}[\widehat{f}_{\mu_2}(s)(\widehat{\breve{Y}}(s)+\widehat{M}(s))]+R(s),
\end{aligned}
\end{equation}
where
$$
\begin{aligned}
R(s)&=\int_0^1d\rho \Big\{(f^\rho_x(s)-f_x(s))(X^{1,\varepsilon}(s)+X^{2,\varepsilon}(s))
+(f^\rho_{x'}(s)-f_{x'}(s))(X^{1,\varepsilon}(s-l)+X^{2,\varepsilon}(s-l))\\
&\quad+(f^\rho_{y}(s)-f_{y}(s))(\breve{Y}(s)+M(s))
+(f^\rho_{z}(s)-f_{z}(s))(\breve{Z}(s)+B_2(s)+B_3(s)+\overline{B}_3(s))\\
&\quad+ \widehat{E}[(\widehat{f}^\rho_{\mu_1}(s)-\widehat{f}_{\mu_1}(s))(\widehat{X}^{1,\varepsilon}(s)+\widehat{X}^{2,\varepsilon}(s))]
+ \widehat{E}[(\widehat{f}^\rho_{\mu_2}(s)-\widehat{f}_{\mu_2}(s))(\widehat{\breve{Y}}(s)+M(s))]
\Big\},
\end{aligned}
$$
and we define, for $0\leq \lambda \leq 1$,
$$
\begin{aligned}
f^\lambda_{k}(s)&=f(s,X^*(s)+\lambda(X^{1,\varepsilon}(s)+X^{2,\varepsilon}(s)),
       X^*(s-l)+\lambda(X^{1,\varepsilon}(s-l)+X^{2,\varepsilon}(s-l)),\\
       &\quad \ Y^*(s)+\lambda(\breve{Y}(s)+M(s)),Z^*(s)+\lambda(\breve{Z}(s)+B_2(s)+B_3(s)+\overline{B}_3(s)),\\
        &\quad \ P_{(X^*(s)+\lambda(X^{1,\varepsilon}(s)+X^{2,\varepsilon}(s)),Y^*(s)+\lambda(\breve{Y}(s)+M(s)))},u^*(s)).
\end{aligned}
$$
Recall the definitions of $M(s), B_2(s),B_3(s),\overline{B}_3(s)$, $I_{2,2}$ changes to
\begin{equation}\label{equ 7.4}
\begin{aligned}
I_{2,2}(s)&=f_y(s)\breve{Y}(s)+f_z(s)\breve{Z}(s)+\widehat{E}[\widehat{f}_{\mu_2}(s)\widehat{\breve{Y}}(s)]
+(X^{1,\varepsilon}(s)+X^{2,\varepsilon}(s))\Big\{f_x(s)+f_y(s)p(s)\\
&\quad+f_z(s)(p(s)\sigma_x(s)+q(s))\Big\}
+(X^{1,\varepsilon}(s-l)+X^{2,\varepsilon}(s-l))\Big\{f_{x'}(s)+f_z(s)p(s)\sigma_{x'}(s)\Big\}\\
&\quad+\widehat{E}\Big[(\widehat{X}^{1,\varepsilon}(s)+\widehat{X}^{2,\varepsilon}(s))\Big(\widehat{\sigma}_\nu(s)f_z(s)p(s)
+\widehat{f}_{\mu_1}(s)+\widehat{f}_{\mu_2}(s)\widehat{p}(s)\Big)\Big]\\
&\quad+\frac{1}{2}(X^{1,\varepsilon}(s))^2
\Big\{f_y(s)P(s)+f_z(s)(p(s)\sigma_{xx}(s)+2P(s)\sigma_x(s)+Q(s))\Big\}\\
\end{aligned}
\end{equation}
$$
\begin{aligned}
&\quad+\frac{1}{2}f_z(s)p(s)E\Big[(\widehat{X}^{1,\varepsilon})^2\Big(\widehat{\sigma}_{\nu y}(s)+\widehat{f}_{\mu_2}(s)\widehat{P}(s)\Big)\Big]
+\frac{1}{2}(X^{1,\varepsilon}(s-l))^2\Big\{f_z(s)(p(s)\sigma_{x'x'}(s)\\
&\quad+2P_1(s)\sigma_{x'}(s))\Big\}
 +X^{1,\varepsilon}(s)X^{1,\varepsilon}(s-l)\Big\{
f_y(s)P_1(s)+f_z(s)\Big(p(s)\sigma_{xx'}(s)+P(s)\sigma_{x'}(s)\\
&\quad+P_1(s)\sigma_x(s)+P_1(s)\sigma_x(s-l)+Q_1(s)\Big)\Big\}
+E[\widehat{f}_{\mu_2}(s)\widehat{P}_1(s)(\widehat{X}^{1,\varepsilon}(s)\widehat{X}^{1,\varepsilon}(s-l))]+R(s).
\end{aligned}
$$

We want to prove $E[(\int_0^T|R(s)ds|)^2]\leq \varepsilon^2\rho(\varepsilon).$ Firstly,
notice that
\begin{equation}\label{equ 7.5}
\begin{aligned}
\bullet\quad&\int_0^1(f_x^\rho(s)-f_x(s))(X^{1,\varepsilon}(s)+X^{2,\varepsilon}(s))d\rho\\
&=\frac{1}{2} (X^{1,\varepsilon}(s))^2  \Big\{
f_{xx}(s)
+f_{xy}(s)p(s)
+f_{xz}(s)(p(s)\sigma_x(s)+q(s))\Big\}+\frac{1}{2}f_{xx'}(s)X^{1,\varepsilon}(s-l)X^{1,\varepsilon}(s)\\
&\quad+R_1(s,\varepsilon),
\end{aligned}
\end{equation}
where
\begin{equation}\label{equ 7.5-1}
\begin{aligned}
R_1(s,\varepsilon)&=\int_0^1\rho d\rho\int_0^1d\theta\Big\{
f_{xx}^{\rho\theta}(s)(X^{1,\varepsilon}(s)+X^{2,\varepsilon}(s))^2-f_{xx}(s)(X^{1,\varepsilon}(s))^2
\Big\}\\
&+\int_0^1\rho d\rho\int_0^1d\theta\Big\{
f_{xx'}^{\rho\theta}(s)(X^{1,\varepsilon}(s-l)+X^{2,\varepsilon}(s-l))(X^{1,\varepsilon}(s)+X^{2,\varepsilon}(s))\\
&\qquad\qquad\qquad\qquad\qquad\qquad\qquad\qquad\qquad\qquad\qquad\quad-f_{xx'}(s)X^{1,\varepsilon}(s)X^{1,\varepsilon}(s-l)
\Big\}\\
&+\int_0^1\rho d\rho\int_0^1d\theta\Big\{
f_{xy}^{\rho\theta}(s)(\breve{Y}(s)+M(s))(X^{1,\varepsilon}(s)+X^{2,\varepsilon}(s))
-f_{xy}(s)p(s)(X^{1,\varepsilon}(s))^2
\Big\}\\
&+\int_0^1\rho d\rho\int_0^1d\theta\Big\{
f_{xz}^{\rho\theta}(s)(\breve{Z}(s)+B_2(s)+B_3(s)+\overline{B}_3(s))(X^{1,\varepsilon}(s)+X^{2,\varepsilon}(s))\\
&\qquad\qquad\qquad\qquad\qquad\qquad\qquad\qquad\qquad\qquad-f_{xz}(s)(p(s)\sigma_x(s)+q(s))(X^{1,\varepsilon}(s))^2
\Big\}\\
&+\int_0^1\rho d\rho\int_0^1d\theta\Big\{
\widehat{E}[\widehat{f}_{x\mu_1}^{\rho\theta}(\widehat{X}^{1,\varepsilon}(s)+\widehat{X}^{2,\varepsilon}(s))](X^{1,\varepsilon}(s)+X^{2,\varepsilon}(s))
\Big\}\\
&+\int_0^1\rho d\rho\int_0^1d\theta\Big\{
\widehat{E}[\widehat{f}_{x\mu_2}^{\rho\theta}(\widehat{\breve{Y}}(s)+\widehat{M}(s))](X^{1,\varepsilon}(s)+X^{2,\varepsilon}(s))
\Big\}.\\
\end{aligned}
\end{equation}
It is clear that
\begin{equation}\label{equ 7.5-2}
E[(\int_0^T|R_1(s,\varepsilon)| ds)^2]\leq \varepsilon^2\rho(\varepsilon).
\end{equation}
In fact, the kernel to deal with (\ref{equ 7.5-2}) is to prove the following
 inequality involving the second-order derivative $f_{x\mu_2}$ and $X^{1,\varepsilon}(s)$:
\begin{equation}\label{equ 7.5-4}
E\bigg[\int_0^T\Big|\int_0^1\rho d\rho\int_0^1d\theta \Big\{\widehat{E}[\widehat{f}_{x\mu_2}^{\rho\theta}(s)\widehat{p}(s)\widehat{X}^{1,\varepsilon}(s)]
X^{1,\varepsilon}(s)\Big\}\Big|^2ds\bigg]\leq \varepsilon^2\rho(\varepsilon).
\end{equation}
From  H\"{o}lder  inequality, Fubini theorem, (\ref{equ 3.4}) and  (\ref{equ 3.32}),  one can check
\begin{equation}\label{equ 7.5-5}
\begin{aligned}
&E\Big[\int_0^T\Big|\int_0^1\rho d\rho\int_0^1d\theta \Big\{\widehat{E}[\widehat{f}_{x\mu_2}^{\rho\theta}(s)\widehat{p}(s)\widehat{X}^{1,\varepsilon}(s)]
X^{1,\varepsilon}(s)\Big\}\Big|^2ds\Big]\\
&\leq CE\Big[\sup_{s\in[0,T]}|X^{1,\varepsilon}(s)|^2\cdot
\int_0^T\int_0^1\rho\int_0^1\Big|\widehat{E}[\widehat{f}_{x\mu_2}^{\rho\theta}(s)
\widehat{p}(s)\widehat{X}^{1,\varepsilon}(s)]\Big|^2 d \theta d\rho ds\Big]\\
\end{aligned}
\end{equation}
$$
\begin{aligned}
&\leq C\Big\{E\Big[\sup_{s\in[0,T]}|X^{1,\varepsilon}(s)|^4\Big\}^\frac{1}{2}\cdot
\Big\{E\Big[\Big(\int_0^T\int_0^1\rho\int_0^1\Big|\widehat{E}[\widehat{f}_{x\mu_2}^{\rho\theta}(s)
\widehat{p}(s)\widehat{X}^{1,\varepsilon}(s)]\Big|^2 d \theta d\rho ds\Big)^2\Big]\Big\}^\frac{1}{2}\\
&\leq C\varepsilon\Big\{\int_0^1\rho\int_0^1 E\Big[\int_0^T\Big|\widehat{E}[\widehat{f}_{x\mu_2}^{\rho\theta}(s)
\widehat{p}(s)\widehat{X}^{1,\varepsilon}(s)]\Big|^4 ds\Big]d \theta d\rho\Big\}^\frac{1}{2}\leq \varepsilon^2\rho(\varepsilon).
\end{aligned}
$$
Similar to (\ref{equ 7.5}), it is easy to check
\begin{equation}\label{equ 7.6}
\begin{aligned}
&\bullet\quad\int_0^1(f_y^\rho(s)-f_y(s))(X^{1,\varepsilon}(s)+X^{2,\varepsilon}(s))d\rho\\
&\qquad=\frac{1}{2}p(s)(X^{1,\varepsilon}(s))^2  \Big\{
f_{yx}(s)+f_{yy}(s)p(s)+f_{yz}(s)(p(s)\sigma_x(s)+q(s))\Big\}\\
&\qquad\qquad\qquad\qquad\qquad\qquad\qquad\qquad\qquad\qquad
+\frac{1}{2}p(s)f_{yx'}(s)X^{1,\varepsilon}(s-l)X^{1,\varepsilon}(s)+R_2(s,\varepsilon);\\
&\bullet\quad\int_0^1(f_z^\rho(s)-f_z(s))(X^{1,\varepsilon}(s)+X^{2,\varepsilon}(s))d\rho\\
&\qquad=\frac{1}{2}(p(s)\sigma_x(s)+q(s))(X^{1,\varepsilon}(s))^2  \Big\{
f_{zx}(s)
+f_{zy}(s)p(s)
+f_{zz}(s)(p(s)\sigma_x(s)+q(s))\Big\}\\
&\qquad\qquad\qquad\qquad\qquad\qquad\qquad\qquad
+\frac{1}{2}(p(s)\sigma_x(s)+q(s))f_{zx'}(s)X^{1,\varepsilon}(s-l)X^{1,\varepsilon}(s)+R_3(s,\varepsilon);\\
&\bullet\quad\int_0^1d\rho(f_{x'}^\rho(s)-f_{x'}(s))(X^{1,\varepsilon}(s-l)+X^{2,\varepsilon}(s-l))\\
&\qquad=\frac{1}{2}\Big\{
f_{x'x}(s)(X^{1,\varepsilon}(s-l)X^{1,\varepsilon}(s))
+f_{x'x'}(s)(X^{1,\varepsilon}(s-l))^2
+f_{x'y}(s)p(s)(X^{1,\varepsilon}(s-l)X^{1,\varepsilon}(s))\\
&\qquad\qquad\qquad\qquad\qquad\qquad\qquad\qquad
+f_{x'z}(s)(p(s)\sigma_x(s)+q(s))(X^{1,\varepsilon}(s-l)X^{1,\varepsilon}(s))\Big\}
+R_4(s,\varepsilon),\\
\end{aligned}
\end{equation}
where $E[(\int_0^T|R_i(s,\varepsilon)| ds)^2]\leq \varepsilon^2\rho(\varepsilon),\ i=2,3,4.$

In order to complete our proof, we still need to prove the following estimates£»
\begin{equation}\label{equ 7.7}
\begin{aligned}
&\bullet\quad\int_0^1\widehat{E}[(\widehat{f}_{\mu_1}^\rho(s)-\widehat{f}_{\mu_1}(s))
(\widehat{X}^{1,\varepsilon}(s)+\widehat{X}^{2,\varepsilon}(s))]d\rho
=\frac{1}{2}\widehat{E}\Big[
\widehat{f}_{\mu_1a_1}(s)(\widehat{X}^{1,\varepsilon}(s))^2\Big]
+R_5(s,\varepsilon),\\
&\bullet\quad\int_0^1\widehat{E}[(\widehat{f}_{\mu_2}^\rho(s)-\widehat{f}_{\mu_2}(s))(\widehat{\breve{Y}}(s)+\widehat{M}(s))]d\rho
=\frac{1}{2}\widehat{E}\Big[
\widehat{f}_{\mu_2a_2}(s)(\widehat{p}(s))^2(\widehat{X}^{1,\varepsilon}(s))^2\Big]
+R_6(s,\varepsilon),
\end{aligned}
\end{equation}
where $E[(\int_0^T|R_5(s,e)|+|R_6(s,e)|ds)^2]\leq \varepsilon^2\rho(\varepsilon)$.

We mainly deal with the first one in (\ref{equ 7.7}). The second one can be dealt with  the similar argument.\\
Notice that
\begin{equation}\label{equ 7.8}
\begin{aligned}
&\int_0^1\widehat{E}[(\widehat{f}_{\mu_1}^\rho(s)-\widehat{f}_{\mu_1}(s))
(\widehat{X}^{1,\varepsilon}(s)+\widehat{X}^{2,\varepsilon}(s))]d\rho-\frac{1}{2}\widehat{E}\Big[
\widehat{f}_{\mu_1a_1}(s)(\widehat{X}^{1,\varepsilon}(s))^2\Big]\\
&= \Gamma_1(s,\varepsilon)+\Gamma_2(s,\varepsilon)+\Gamma_3(s,\varepsilon)+\Gamma_4(s,\varepsilon),
\end{aligned}
\end{equation}
where
$$
\begin{aligned}
&\Gamma_1(s,\varepsilon):=\int_0^1d\rho\int_0^1d\theta
\widehat{E}[\widehat{f}_{\mu_1a_1}^{\rho \theta}(s)(\widehat{X}^{1,\varepsilon}(s)+\widehat{X}^{2,\varepsilon}(s))^2]
-\frac{1}{2}\widehat{E}\Big[
\widehat{f}_{\mu_1a_1}(s)(\widehat{X}^{1,\varepsilon}(s))^2\Big],\\
&\Gamma_2(s,\varepsilon):=\int_0^1\rho d\rho\int_0^1d\theta
\widehat{E}\Big[(\widehat{X}^{1,\varepsilon}(s)+\widehat{X}^{2,\varepsilon}(s))\widetilde{E}[\widetilde{\widehat{f}}_{\mu_1\mu_1}^{\rho \theta}(s)(\widetilde{X}^{1,\varepsilon}(s)+\widetilde{X}^{2,\varepsilon}(s))]\Big],\\
&\Gamma_3(s,\varepsilon):=\int_0^1\rho d\rho\int_0^1d\theta
\widehat{E}\Big[(\widehat{X}^{1,\varepsilon}(s)+\widehat{X}^{2,\varepsilon}(s))\widetilde{E}[\widetilde{\widehat{f}}_{\mu_1\mu_2}^{\rho \theta}(s)(\widetilde{\breve{Y}}(s)+\widetilde{M}(s))]\Big],\\
\end{aligned}
$$
$$
\begin{aligned}
&\Gamma_4(s,\varepsilon):=\int_0^1\rho d\rho\int_0^1d\theta
\widehat{E}\Big[(\widehat{X}^{1,\varepsilon}(s)+\widehat{X}^{2,\varepsilon}(s))
\Big\{
\widehat{f}_{\mu_1x}^{\rho \theta}(s)(X^{1,\varepsilon}(s)+X^{2,\varepsilon}(s))\\
&\qquad\qquad\qquad\qquad\qquad\quad+\widehat{f}_{\mu_1x'}^{\rho \theta}(s)(X^{1,\varepsilon}(s-l)+X^{2,\varepsilon}(s-l))
+\widehat{f}_{\mu_1y}^{\rho \theta}(s)(\breve{Y}(s)+M(s))\\
&\qquad\qquad\qquad\qquad\qquad\quad+\widehat{f}_{\mu_1z}^{\rho \theta}(s)(\breve{Z}(s)+B_2(s)+B_3(s)+\overline{B}_3(s))
\Big\}\Big].\\
\end{aligned}
$$
Let us rewrite $\Gamma_1(s,\varepsilon)$ as follows
\begin{equation}\label{equ 7.9}
\begin{aligned}
\Gamma_1(s,\varepsilon)&=\int_0^1d\rho\int_0^1d\theta
\widehat{E}[(\widehat{f}_{\mu_1a_1}^{\rho \theta}(s)-\widehat{f}_{\mu_1a_1}(s))
(\widehat{X}^{1,\varepsilon}(s))^2]\\
&\quad+\int_0^1d\rho\int_0^1d\theta
\widehat{E}[\widehat{f}_{\mu_1a_1}^{\rho \theta}(s) (2\widehat{X}^{1,\varepsilon}(s)\widehat{X}^{2,\varepsilon}(s)+
(\widehat{X}^{2,\varepsilon}(s))^2).\\
\end{aligned}
\end{equation}
From the Lipschitz property of $\widehat{f}_{\mu_1a_1}(s)$ and (\ref{equ 3.4}),
one can show $E[(\int_0^T|\Gamma_1(s,\varepsilon)| ds)^2]\leq \varepsilon^2\rho(\varepsilon)$ easily.

As for $\Gamma_3(s,\varepsilon)$, according to (\ref{equ 3.4}) and (\ref{equ 3.32}), one has
\begin{equation}\label{equ 7.10}
\begin{aligned}
&E\Big[\int_0^T \Big|\int_0^1\rho d\rho\int_0^1d\theta
\widehat{E} [\widehat{X}^{1,\varepsilon}(s)\widetilde{E}[\widetilde{\widehat{f}}_{\mu_1\mu_2}^{\rho \theta}(s)\widetilde{p}(s)\widetilde{X}^{1,\varepsilon}(s)] ]\Big|^2ds\Big]\\
&\leq\frac{1}{2}E\Big[\int_0^T \int_0^1\rho d\rho\int_0^1d\theta
E[\sup_{s\in[0,T]}|X^{1,\varepsilon}(s)|^2]\cdot
\widehat{E}\Big |\widetilde{E}[\widetilde{\widehat{f}}_{\mu_1\mu_2}^{\rho \theta}(s)\widetilde{p}(s)\widetilde{X}^{1,\varepsilon}(s)]\Big|^2ds\Big]\\
&\leq\frac{1}{2}\varepsilon\int_0^1\rho d\rho\int_0^1 d\theta
\Big\{E\widehat{E}[\int_0^T|\widetilde{E}[\widetilde{\widehat{f}}_{\mu_1\mu_2}^{\rho \theta}(s)\widetilde{p}(s)\widetilde{X}^{1,\varepsilon}(s)]|^2ds]\Big\}\\
&\leq\frac{\sqrt{T}}{4}\varepsilon \int_0^1\rho d\rho\int_0^1 d\theta
\Big\{E\widehat{E}[\int_0^T|\widetilde{E}[\widetilde{\widehat{f}}_{\mu_1\mu_2}^{\rho \theta}(s)\widetilde{p}(s)\widetilde{X}^{1,\varepsilon}(s)]|^4ds]\Big\}^\frac{1}{2}\\
&\leq\varepsilon^2\rho(\varepsilon).
\end{aligned}
\end{equation}
$\Gamma_2(s,\varepsilon)$ can be calculated with the similar method. We now turn our attention to $\Gamma_4(s,\varepsilon)$.
 The critical component of  $\Gamma_4(s,\varepsilon)$
 is
  $$\Gamma_{4,1}(s,\varepsilon):=\int_0^1\rho\int_0^1
\widehat{E} [\widehat{X}^{1,\varepsilon}(s)\widehat{f}_{\mu_1z}^{\rho \theta}(s)q(s)X^{1,\varepsilon}(s)]d\theta d\rho.$$
Thanks to H\"{o}lder inequality,  (\ref{equ 3.4}) and (\ref{equ 3.32}) again, it yields
 \begin{equation}\label{equ 7.11}
\begin{aligned}
 &E[(\int_0^T|\Gamma_{4,1}(s,\varepsilon)| ds)^2]\\
 &\leq E\Big[\sup_{s\in[0,T]}|X^{1,\varepsilon}(s)|^2
 (\int_0^T\int_0^1\rho\int_0^1
\widehat{E}[ \widehat{f} _{\mu_1z}^{\rho \theta}(s)\widehat{X}^{1,\varepsilon}(s)]\cdot q(s)d\theta d\rho ds)^2\Big]\\
 &\leq CE\Big[\sup_{s\in[0,T]}|X^{1,\varepsilon}(s)|^2\cdot
 \int_0^T|q(s)|^2ds\cdot
 \int_0^T\int_0^1\rho\int_0^1
 |\widehat{E}[ \widehat{f}_{\mu_1z}^{\rho \theta}(s)\widehat{X}^{1,\varepsilon}(s)]|^2d\theta d\rho ds\Big]\\
  &\leq C \Big\{E[\sup_{s\in[0,T]}|X^{1,\varepsilon}(s)|^8]\Big\}^\frac{1}{4}\cdot
  \Big\{E[(\int_0^T|q(s)|^2ds)^4]\Big\}^\frac{1}{4}\cdot\\
&\qquad\qquad\qquad\qquad\qquad\qquad\qquad   \Big\{E[(\int_0^T\int_0^1\rho\int_0^1
 |\widehat{E}[ \widehat{f}_{\mu_1z}^{\rho \theta}(s)\widehat{X}^{1,\varepsilon}(s)]|^2d\theta d\rho ds)^2]\Big\}^\frac{1}{2}\\
 &\leq C\varepsilon  \Big\{\int_0^1\rho\int_0^1\int_0^T
 E[|\widehat{E}[ \widehat{f}_{\mu_1z}^{\rho \theta}(s)\widehat{X}^{1,\varepsilon}(s)]|^4]dsd\theta d\rho\Big\}^\frac{1}{2}
 \leq \varepsilon^2\rho(\varepsilon).
\end{aligned}
\end{equation}
Hence, if define $R_5(s,\varepsilon):=\Gamma_1(s,\varepsilon)+\Gamma_2(s,\varepsilon)+\Gamma_3(s,\varepsilon)+\Gamma_4(s,\varepsilon)$,
then from (\ref{equ 7.9})-(\ref{equ 7.11})  it follows
$E[(\int_0^T|R_5(s,\varepsilon)| ds)^2]\leq \varepsilon^2\rho(\varepsilon).$

Making use of the above argument, one can also have the second estimate in (\ref{equ 7.7}) with
$E[(\int_0^T|R_6(s,\varepsilon)| ds)^2]\leq \varepsilon^2\rho(\varepsilon).$
At last, combing  (\ref{equ 7.5}),\ (\ref{equ 7.6}),\ (\ref{equ 7.7}) we have the desired result.


\renewcommand{\refname}{\large References}{\normalsize \ }


\begin{thebibliography}{99}

\bibitem{AHMP}
M. Arriojas, Y. Hu, S. Mohammed, G. Pap,
(2007)
\emph{A delay Black and  Scholes formula},
Stoch. Analysis Appl.,
\textbf{25} (2): 471-492.

\bibitem{BDLP}
R. Buckdahn, B. Djehiche, J.  Li,  S. Peng,
(2009)
\emph{Mean-field backward stochastic differential equations: A limit approach},
Ann. Probab.,
\textbf{37}(4): 1524-1565.


\bibitem{BLP}
R. Buckdahn,  J. Li, S. Peng,
(2009)
\emph{Mean-field backward stochastic differential equations and related partial differential equations},
Stoch. Proc. Appl.,
\textbf{119}: 3133-3154.

\bibitem{BLPR}
R. Buckdahn, J.  Li, S.  Peng, C. Rainer,
(2017)
\emph{Mean-field stochastic differential equations and associated PDEs},
Ann. Probab.,
\textbf{45}: 824-874.


\bibitem{BLM1}
R. Buckdahn,  J. Li, J. Ma,
(2016)
\emph{A stochastic maximum principle for general   mean-field systems},
Appl. Math. Optim.,
\textbf{74}: 507-534.

\bibitem{BLM2}
R. Buckdahn,  J. Li, J. Ma,
(2017)
\emph{A mean-field stochastic control problem with partial observations},
Ann. Appl. Probab.,
\textbf{27}(5).

\bibitem{CD} R. Carnoma, F. Delarue,
(2013)
\emph{Probabilistic analysis of mean-field games},
SIAM J. Control Optim.,
\textbf{51}(4): 2705-2734.

\bibitem{CD1}
R. Carnoma, F. Delarue, A. Lachapelle,
(2013)
\emph{Control of McKean-Vlasov dynamics versus mean field games},
 Math. Financ. Econ.,
\textbf{7}(2): 131-166.



\bibitem{CD2}
 R. Carmona, F. Delarue,
 (2015)
\emph{Forward-backward stochastic differential equations and controlled McKean Vlasov dynamics},
Ann. Probab.,
\textbf{43}(5): 2647-2700.




\bibitem{CCD}
J.F. Chassagneux, D. Crisan, F. Delarue,
(2015)
\emph{A probabilistic approach to classical solutions
of the master equation for large population equilibria},
http://arxiv.org/abs/1411.3009v2.


\bibitem{CH}
L. Chen, J. Huang,
(2015)
\emph{ Stochstic maximum principle for controlled backward delayed system via advanced stochastic differential equation},
J. Optim. Theory Appl.
\textbf{167}: 1112-1135.




\bibitem{CW}
L. Chen, Z. Wu,
(2010)
\emph{Maximum principle for the stochastic optimal control problem with delay and applications},
Automatica,
\textbf{46}: 1074-1080.

\bibitem{EOS}
I. Elsanosi, B. $\phi$ksendal, A. Sulem,
(2000)
\emph{Some solvable stochastic control problems with delay},
Stochastics,
\textbf{71}: 69-89.


\bibitem{KPQ}
N. El Karoui, S. Peng, M. Quenez,
(1997)
\emph{Backward Stochastic Differential Equation in Finance},
Math. Finance,
\textbf{7}(1): 1-71.




\bibitem{GM} G. Guatteri, F. Masiero
(2018)
\emph{Stochastic maximum principle for equations with delay: the non-convex case},
http://cn.arxiv.org/pdf/1805.07957.


\bibitem{GXZ} H. Guo, J. Xiong, J. Zheng,
(2017)
\emph{Stochastic maximum principle for generalized mean-field delay control problem},
https://arxiv.org/pdf/1708.03622.



\bibitem{HL} T. Hao, J. Li
(2016)
\emph{Mean-field SDEs with jumps and nonlocal integral-PDEs},
Nonlinear Differ. Equ. Appl.,
\textbf{23}(2): 1-51.




\bibitem{Hu} M.  Hu,
(2017)
\emph{Stochastic global maximum principle for optimization with recursive utilities},
Probab. Unce. Quanti. Risk,
\textbf{2}(1): 1-20.


\bibitem{HJX} M.  Hu, S. Ji, X. Xue
(2018)
\emph{A global stochastic maximum principle for fully coupled forward-backward stochastic systems},
arxiv: 1803.02109v3.

\bibitem{Kac} M. Kac,
(1956)
\emph {Foundations of kinetic theory},
In Proceedings of the 3rd Berkeley Symposium
on Mathematical Statistics and Probability,
\textbf{3}: 171-197.

\bibitem{Li}
J. Li,
(2018)
\emph{Mean-field forward and backward SDEs with jumps. Associated nonlocal quasi-linear integral-PDEs},
Stoch. Proc. Appl.,
\textbf{128}(9): 3118-3180.



\bibitem{Lions}
P.L. Lions,
(2013)
\emph {Cours au Coll\`{e}ge de France : Th\'eorie des jeu \`{a} champs moyens},
http://www.college-de-france.fr/default/EN/all/equ[1]der/audiovideo.jsp.



\bibitem{Mckean}H.P. McKean,
(1966)
\emph {A class of Markov processes associated with nonlinear parabolic equations},
Proceedings of the National Academy of Sciences,
\textbf{56}: 1907¨C1911.

\bibitem{Peng1}
S. Peng,
(1990)
\emph{A general stochastic maximum principle for optimal control problems},
SIAM J. Control Optim.,
\textbf{28}: 966-979.



\bibitem{SMS}
Y. Shen, Q. Meng, P. Shi
(2014)
\emph{Maximum principle for mean-field jump-diffusion stochastic delay differential equations and its application to finance},
Automatica,
\textbf{50}(6):1565-1579.

\bibitem{Yong} J. Yong,
(2010)
\emph{Optimality variational principle for controlled forward-backward stochastic differential equations with mixed initial-terminal conditions}
SIAM J, Control Optim.,
\textbf{48}(6): 4119-4156.




\end{thebibliography}
\end{document}